\documentclass[OJMO,manuscript]{cedram} 

\usepackage{mymacros}

\title{Exploiting Agent Symmetries for Performance Analysis of Distributed Optimization Methods}

\author{\firstname{Sebastien} \lastname{Colla}}
\address{ICTEAM Institute, Department of applied mathematics, UCLouvain, Louvain-la-Neuve, Beglium}
\author{\firstname{Julien} \middlename{M.} \lastname{Hendrickx}}
\address{ICTEAM Institute, Department of applied mathematics, UCLouvain, Louvain-la-Neuve, Beglium}
\thanks{S. Colla is a FRIA grantee of the Fonds de la Recherche Scientifique - FNRS. J. M. Hendrickx is supported by the Federation
Wallonie-Bruxelles through the “\emph{RevealFlight}” Concerted Research Action (ARC) and by the F.R.S.-FNRS via the research project KORNET and the Incentive Grant for Scientific Research (MIS) "\emph{Learning from Pairwise Comparisons}". \vspace*{-4mm}}

  
\begin{abstract}
  We show that, in many settings, the worst-case performance of a distributed optimization algorithm is independent of the number of agents in the system, and can thus be computed in the fundamental case with just two agents. 
  This result relies on a novel approach that systematically exploits symmetries in worst-case performance computation, framed as Semidefinite Programming (SDP) via the Performance Estimation Problem (PEP) framework. 
  Harnessing agent symmetries in the PEP yields compact problems whose size is independent of the number of agents in the system. When all agents are equivalent in the problem, we establish the explicit conditions under which the resulting worst-case performance is independent of the number of agents and is therefore equivalent to the basic case with two agents. 
  Our compact PEP formulation also allows the consideration of multiple equivalence classes of agents, and its size only depends on the number of equivalence classes. This enables practical and automated performance analysis of distributed algorithms in numerous complex and realistic settings, such as the analysis of the worst agent performance.
  We leverage this new tool to analyze the performance of the EXTRA algorithm in advanced settings and its scalability with the number of agents, providing a tighter analysis and deeper understanding of the algorithm performance.
\end{abstract}

\begin{document}
\maketitle
  
\section{Introduction}

We consider the distributed optimization problem in which a set of agents $\V = \{1,\dots,\N\}$ is connected through a communication network and works together to minimize the average of their local functions $f_i: \Rvec{d}\to\mathbb{R}$, \vspace{-1mm}
\begin{equation} \label{opt:dec_prob}
  \underset{\text{\normalsize $x \in \Rvec{d}$}}{\mathrm{min}} \quad f(x) = \frac{1}{\N}\sum_{i=1}^\N f_i(x), \vspace{-1mm}
\end{equation}
Each agent $i$ performs local computation on its own local function $f_i$ and exchanges local information with its neighbors to update its own guess $x_i$ of the solution. The agents all want to come to an agreement on the minimizer $x^*$ of the global function $f$.
One of the first methods proposed to solve \eqref{opt:dec_prob} is the distributed (sub)gradient descent (DGD) \cite{DsubGD} where agents successively perform an average consensus step \eqref{eq:DGD_cons} and a local gradient step \eqref{eq:DGD_comp}:  \vspace{-1mm}
\begin{align}
  y_i^k &= \sum_{j=1}^\N w_{ij}^k x_j^k,   \hspace{30mm} \text{for $i=1,\dots,\N$,} \label{eq:DGD_cons} \\[-0.5mm]
  x_i^{k+1} &= y_i^k - \alpha^k \nabla f_i(x_i^k),  \hspace{20mm} \text{for $i=1,\dots,\N$,} \label{eq:DGD_comp} \vspace{-0.5mm}
\end{align}
for given step-sizes $\alpha^k > 0$ and matrices of weights $W^k = [w_{ij}^k] \in \Rmat{\N}{\N}$, typically assumed symmetric and with rows and columns summing to one. We call such matrices \emph{averaging matrices}. 
\begin{definition}
  We say that a matrix $W \in \Rmat{\N}{\N}$ is an \emph{averaging matrix} if 
  \begin{enumerate}
    \item $W^T = W$,  \hfill  (Symmetry)
    \item  $W \mathbf{1} = \mathbf{1}$ and $\mathbf{1}^T W  = \mathbf{1}^T$, \hfill (Averaging Consensus)
  \end{enumerate}
\end{definition}
There are many other methods combining gradient sampling and consensus steps. We call $\Ad$ the set of all these decentralized optimization algorithms. 
\begin{definition}[Class of distributed optimization methods $\Ad$] \label{def:Ad}
  We define $\Ad$ as the set of all the decentralized optimization algorithms built based on the three following types of operations and which may involve an arbitrary number of local variables:
  \begin{enumerate}[label=(\roman*)]
    \item \emph{Gradient:}
 Each agent samples the (sub)gradient of its local function at any of its local variables $x_i$
 \[ g_i = \nabla f_i(x_i) \qquad \text{for $i=1,\dots,\N$}. \]
  \item \emph{Consensus:}
  All the agents perform a consensus step on any of their local variables $x_i$
  \begin{equation}
    y_i = \sum_{j=1}^\N w_{ij} x_j, \quad \text{for $i=1,\dots,\N$.}  \label{eq:cons_step}
  \end{equation}
  with weights $w_{ij}$ given by an averaging matrix $W$.
  Different consensus steps can possibly use the same averaging matrix. 
  \item\emph{Linear Combinations:}
  Each agent declares that a linear combination of its local variables holds, with coefficients known in advance. 
  Depending on the settings, the linear coefficients should be identical for all agents (coordinated) or not (uncoordinated). 
  \end{enumerate}
\end{definition}
This class of methods includes many algorithms. These three operations even allow implicit (or proximal) updates, e.g. updates where the point at which the gradient is evaluated is not explicitly known:
\[x_i^{k+1} = x_i^k - \alpha^k \nabla f_i(x_i^{k+1}) \qquad \text{for all $i =1,\dots,\N$}.\]
Among the primal-based algorithms, $\Ad$ includes, for example, DGD \cite{DsubGD}, DIGing \cite{DIGing}, EXTRA \cite{EXTRA}, NIDS \cite{NIDS}, Acc-DNGD \cite{AccDNGD}, OGT \cite{optimal_gradient_tracking} and their variations \cite{ATC-DIGing, AugDGM, li2020revisiting, jakovetic2018unification}.
Among the dual-based algorithms, $\Ad$ includes, for example, the Distributed Dual Dveraging \cite{DDA}, MSDA \cite{MSDA}, MSDP \cite{MSPD}, APAPC \cite{APAPC}, OPTRA \cite{OPTRA}, APM \cite{APM} and others \cite{uribe2020dual}. The classical decentralized version of ADMM \cite{ADMM_1, ADMM} does not fit in $\Ad$ because the agents do not explicitly use averaging consensus when interacting, but the weighted decentralized version of ADMM \cite{weighted_ADMM} does fit in $\Ad$.

In optimization, the assessment of the performance of a method is generally based on worst-case guarantees. Accurate worst-case guarantees on the performance of decentralized algorithms are crucial for a comprehensive understanding of how their performance is influenced by their parameters and the network topology, which then allows to correctly tune and compare the different algorithms.
However, deriving such bounds can often be a challenging task, requiring a proper combination of the impact of the optimization component and of the communication network, which may result in conservative or highly complex bounds. Moreover, theoretical analyses often require specific settings that differ from one algorithm to another, making them difficult to compare. 

\subsection{Contributions and paper organization}
In this work, we propose a simplified and unified way of analyzing the worst-case performance of distributed optimization methods from $\Ad$, based on the Performance Estimation Problem (PEP) framework and the exploitation of agent symmetries in the problem. This allows us, in particular, to identify and characterize situations in which performance is independent of the number of agents $\N$ in the network, in which case the performance computation can be reduced to the basic case with two agents.

The Performance Estimation Problem (PEP) formulates the computation of a worst-case performance guarantee as an optimization problem itself, by searching for the iterates and the function leading to the largest error after a given number of iterations of the algorithm. The PEP approach has led to many results in centralized optimization, see e.g. \cite{PEP_Smooth,PEP_composite}, and we have recently proposed a formulation tailored for decentralized optimization that searches for the worst local iterates and local functions for each agent, see \cite{PEP_dec} (and its preliminary conference version \cite{PEP_dec_CDC}). 
This formulation considers explicitly each agent in the PEP and so the size of the optimization problem increases with the number of agents $\N$, and the results are valid only for a given $\N$.
We therefore call this formulation the agent-dependent formulation. Section \ref{sec:agentPEP} summarizes this formulation and adds some new results and interpretations.
To obtain guarantees that are valid for any symmetric and stochastic averaging matrix, with given bounds on the eigenvalues, in \cite{PEP_dec} we proposed a way of representing a consensus step \eqref{eq:cons_step} in PEP via necessary constraints that consensus variables $x_i$ and $y_i$ should satisfy. Subsection \ref{sec:consensusPEP} describes these constraints and analyzes their sufficiency. We first show that no convex description can tightly describe the consensus iterates, and then we characterize the generalized steps that are tightly described by the proposed convex constraints. This helps explain why these necessary conditions enable accurate performance calculations in PEP, as observed in \cite{PEP_dec}. 

In \cite{PEP_dec}, we have also observed that for many performance settings, the worst-case guarantees obtained with PEP are independent of the number of agents, while the PEP problems are not.  
 This motivates us to find compact ways of formulating the agent-dependent PEP. We have performed a first attempt in \cite{PEP_dec_CDC22}, where we built a relaxation of the problem whose size is independent of the number of agents $\N$. In particular, the relaxation does not exploit separability of the objective function $f(x)$ in \eqref{opt:dec_prob}.
 While the relaxation in \cite{PEP_dec_CDC22} gives good worst-case bounds, it remains an open question whether the equivalence with the agent-dependent formulation holds and how to interpret the resulting worst-case solution for the decentralized problem.
 In this paper, we provide an intuitive and systematic way of exploiting the agent symmetries in PEP to make the problem compact, with a size independent of the number of agents. In Section \ref{sec:equicl}, we define \emph{equivalent} agents in a PEP for decentralized optimization, and we leverage the convexity of the PEP to show that we can restrict it to solutions symmetrized over equivalent agents, without impacting its worst-case value. 
 Section \ref{sec:symPEP} focuses on the case where all the agents are equivalent, meaning none of them play a specific role in the algorithm or the performance evaluation. In that case, we show that the agent-dependent PEP can be written compactly, i.e. with an SDP whose size is independent of $\N$. Moreover, we show that the worst-case value of this compact formulation is also independent of $\N$ in many common performance settings of distributed optimization which are scale-invariant.
 This result is particularly powerful since it allows determining situations where the performance analysis of a distributed algorithm can be reduced to the fundamental case with only two agents. We further leverage agent equivalence to draw general conclusions about the symmetry of worst-case local functions and iterates in decentralized algorithms.

Then, Section \ref{sec:subsetsAgPEP} generalizes the results from Section \ref{sec:symPEP} to situations where there are several equivalence classes of agents in the PEP, leading to a compact PEP formulation whose size only depends on the number of equivalence classes $\U$, and not directly on the total number of agents $\N$.
Indeed, while many performance estimation problems for decentralized optimization have all agents equivalent, there are more complex ones that involve multiple equivalence classes of agents, for example, when different groups of agents use different (uncoordinated) step-sizes, function classes, initial conditions, or even different algorithms.
It can also happen that the performance measure focuses on a specific group of agents, e.g. the performance of the worst agent.
Efficiently and accurately assessing distributed optimization algorithm performance in such advanced settings would enhance comprehensive analysis and deepen our understanding of their behavior. 

We demonstrate this in Section \ref{sec:showEXTRA}, where we analyze the performance of the EXTRA algorithm \cite{EXTRA} in advanced settings and its evolution with the number of agents in the problem. We choose EXTRA because it is a well-known decentralized optimization algorithm, one of the first to converge with constant steps, and has served as an inspiration and building block for other algorithms. Our analysis of EXTRA first confirms our theoretical results predicting which performance settings would lead to agent-independent guarantees. It also reveals that the performance of the worst agent scales sublinearly with $\N$ and can benefit from an appropriate step-size decrease with $\N$. Inspired by statistical approaches, we go further by analyzing the 80-th percentile of the agent performance, i.e. the error at or below which 80\% of the agents fall in the worst-case scenario. This performance measure presents a better dependence on the number of agents $\N$ than the performance of the worst agent and quickly reaches a plateau when $\N$ increases, which can be validated by solving compact PEP for $\N \to \infty$. To the best of our knowledge, such percentile analysis is beyond the reach of current classical analysis techniques. Finally, we analyze the performance of EXTRA under agent heterogeneity in the classes of local functions and observe that it does not depend on the total number of agents but only on the proportion of each class of functions that are present in the system.

\subsection{Related work}
An alternative approach for automatic computation of performance guarantees of optimization methods is proposed in \cite{IQC} and relies on the formulation of optimization algorithms as dynamical systems. Integral quadratic constraints (IQC), generally used to obtain stability guarantees on complex dynamical systems, are adapted to provide sufficient conditions for the convergence of optimization methods and deduce numerical bounds on the convergence rates. This theory has been extended to decentralized optimization in \cite{IQC_dec}.
The methodology analyzes the convergence rate of a single iterate, which is beneficial for the problem size that remains small, and that is also independent of the number of agents in the decentralized case. However, this does not allow dealing with non-geometric convergences, e.g. on smooth convex functions, nor to analyze cases where a property applies over several iterations, e.g. when averaging matrices are constant. These are possible with the PEP approach which computes the worst-case performance on a given number $\K$ of iterations, but solving problems whose size grows with $\K$. Moreover, the current IQC approach only applies to settings where all the agents are equivalent, which is not the case, for example, if we want to analyze the performance of the worst agent. In this work, we propose a practical way, via the PEP approach, to compute performance where there are several equivalence classes of agents. This expands the range of situations we can automatically analyze efficiently in decentralized optimization, including more complex or realistic settings. 

\subsection{Notations}
Let $x_i^k \in \Rvec{d}$ denote the $k$-th local variable $x$ of agent $i$. The local variable of all the agents can be stacked vertically in $\x^k \in \Rvec{\N d}$:
\[ \x^k = \begin{bmatrix} x_1^k \\\vdots\\ x_\N^k \end{bmatrix}, \]
We can also gather different iterates $k=0,\dots,\K$ in a matrix $X \in \Rmat{\N d}{\K+1}$:
\[ X = \begin{bmatrix} \x^0 &\dots& \x^\K \end{bmatrix}. \]
These vector notations apply to any variables that are used in the decentralized algorithms and assume that the local variables exist for each agent and each iteration, but our work also applies if the variables are not defined for all agents or all iterations of the algorithm. Furthermore, our approach also applies if there is no clear concept of iterations. Using these vector notations, we can write the consensus step \eqref{eq:cons_step} for all agents at once as 
\begin{equation}
  \y^k = (W^k \otimes I_d) \x^k, \qquad \text{for $k=0,\dots,\K$} 
  \label{eq:cons}
\end{equation}
where $W^k \in \Rmat{\N}{\N}$ is the averaging matrix used in the consensus,  $I_d$ denotes the identity matrix of size $d$ and $\otimes$ the Kronecker product. This latter notation means that we apply the same matrix $W^k$ for each dimension of the agent variables. Moreover, if the same averaging matrix $W$ is used for different consensus steps ($k=0,\dots,\K$), we can write all the consensus steps at once using matrix notations:
\begin{equation}
  Y = (W \otimes I_d) X. \label{eq:consensus}
\end{equation}
The $i^{\mathrm{th}}$ largest eigenvalues of a matrix $W$ is denoted $\lam_i(W)$.
The agent average of iterate $\x^k \in \Rvec{\N d}$ is denoted $\xb^k \in \Rvec{d}$ and is defined as $\xb^k = \frac{1}{\N} \sum_{i=1}^\N x_i^k.$
We can also gather the agent average of different iterates $k=0,\dots,\K$ in a matrix $\overline{X} \in \Rmat{d}{\K+1}$: 
$\overline{X} = \begin{bmatrix} \xb^0 &\dots& \xb^\K \end{bmatrix}. $
Finally, $\mathbf{1}$ is the vector full of 1.

\section{Agent-dependent performance estimation problem for distributed optimization} \label{sec:agentPEP}

\subsection{Performance Estimation Problem (PEP) framework}

To obtain a tight bound on the performance of a distributed optimization algorithm $\Al$, the conceptual idea is to find instances of local functions and starting points for all agents, allowed by the setting considered, that give the largest error after a given number $\K$ of iterations of the algorithm.
The performance estimation problem (PEP) formulates this idea as a real optimization problem that maximizes the error measure $\P$ of the algorithm result, over all possible functions and initial points allowed \cite{PEP_Drori}:
\begin{align}
  w(\Setting_\N)~=~&\underset{x^*,\{x_i^k, y_i^k, f_i\}_{i\in \V}}{\max} ~ \P(f_i,x_i^0,\dots,x_i^\K,x^*) && \stackrel{\text{e.g.}}{=}~ \text{\small{$\frac{1}{\N} \sum_{i=1}^\N \qty(f_i(\xb^\K) - f_i(x^*))$}} & \label{eq:PEP}\\
  \text{ s.t. } \quad& x_i^{k}, y_i^k \quad \underset{\text{\qquad e.g. DGD \eqref{eq:DGD_cons}-\eqref{eq:DGD_comp}}}{\text{from algorithm $\Al$,}} && \text{for }~i \in \V, ~ k=1,\dots,\K, & \text{(algorithm)} \\
  &  f_i \in \F,&& \text{for }~i \in \V & \text{(class of functions)} \\
  & x_i^0 \quad \underset{\text{e.g. $\|x_i^0 - x^*\|^2 \le 1$,~ $i \in \V$}}{\text{satisfies $\I$, \hspace{2cm}}}&&& \text{(initial conditions)} \\[-1mm]
  & \frac{1}{n} \sum_{i=1}^n \nabla f_i(x^*) = 0, &&&\text{(optimality condition for \eqref{opt:dec_prob})}
\end{align}
where $\Setting_\N$ is the performance evaluation setting which specifies the number of agents $\N$, the algorithm $\Al$, the number of steps $\K$, the performance criterion $\P$, the class of functions $\F$ for the local function, the initial conditions $\I$ and the class of averaging matrices $\W$ that can be used in the algorithm:
\[ \Setting_\N = \{\N,\Al,\K,\P,\F,\I, \W\}. \]
We assume here that the local functions all belong to the same function class. This assumption is usually made in the literature but is not necessary in this PEP formulation.
The problem can also have other auxiliary iterates as variables if needed for the analyzed algorithm. Here we have chosen to show auxiliary iterates $y_i^k$ which are the results of the consensus steps \eqref{eq:cons_step} that may be used at each iteration of the algorithm.

The optimal value of problem \eqref{eq:PEP}, denoted $w(\Setting_\N)$, gives by definition a tight worst-case performance bound for the given setting $\Setting_\N$. Moreover, the optimal solution corresponds to an instance of local functions and initial points actually reaching this upper bound, which can provide very relevant information on the bottlenecks faced by the algorithm.
Solving a PEP such as \eqref{eq:PEP} is in general not easy because the problem is inherently infinite-dimensional, as it contains continuous functions $f_i$ among its variables. Nevertheless, Taylor et al. have shown \cite{PEP_Smooth,PEP_composite} that PEPs can be formulated as a finite semidefinite program (SDP) and can thus be solved exactly, for a wide class of \emph{centralized} first-order algorithms and different classes of functions. The reformulation techniques developed for classical optimization algorithms can also be applied for distributed optimization, as detailed in \cite{PEP_dec}. In what follows, we briefly explain how to reformulate the problem \eqref{eq:PEP} into an SDP. One of the differences in PEP for distributed optimization is that the problem must find the worst-case for each of the $\N$ local functions $f_i$ and the $\N$ sequences of local iterates $x_i^0 \dots x_i^\K$ ($i=1,\dots,\N$). 
To render problem \eqref{eq:PEP} finite, for each agent $i$, rather than considering its local function $f_i$ as a whole, we only consider the discrete set $\{\qty(x_i^k,g_i^k,f_i^k)\}_{k\in I = \{0,\dots,\K,*\}}$ of local iterates $x_i^k$ together with their local gradient-vectors $g_i^k$ and local function values $f_i^k$. We then impose interpolation constraints on this set to ensure its consistency with an actual function $f_i \in \F$, in the sense that the set is $\F$-\emph{interpolable}: there exists a function $f_i \in \F$ such that $f_i(x_i^k) = f_i^k$ and $\nabla f_i(x_i^k) = g_i^k$. 
 Such interpolation constraints are provided for many classes of functions in \cite[Section 3]{PEP_composite}, such as the class $L$-smooth and $\mu$-strongly convex functions. 
\begin{proposition}[Interpolation constraints for $\F_{\mu,L}$ \cite{PEP_composite}]
  Let $I$ be a finite index set and $\F_{\mu,L}$ the set of $L$-smooth and $\mu$-strongly convex functions.
  A set of triplets $\{(x_k, g_k, f_k)\}_{k \in I}$ is $\F_{\mu,L}$-interpolable if and only if the following conditions hold for every pair of indices $k \in I$ and $l \in I$:
  \begin{equation} \label{eq:interpmuL} f_k-f_l-g_l^T(x_k-x_l) \ge \frac{1}{2(1-\mu/L)}\qty(\frac{1}{L}\|g_k-g_l\|^2 + \mu \|x_k-x_l\|^2 - 2\frac{\mu}{L}(g_k-g_l)^T(x_k-x_l)).
  \end{equation}
\end{proposition}

Then, all the constraints from \eqref{eq:PEP} can be expressed in terms of $\{y_i^k,x_i^k,g_i^k,f_i^k\}_{k\in I = \{0,\dots,\K,*\}}$. 
When the averaging matrix $W$ is given, i.e. $\W = \{ W \}$, the algorithm constraints simply consist of linear constraints between these variables. 
The optimality condition for \eqref{opt:dec_prob} can be expressed as a linear constraint on the local gradients at $x^*$.
The performance criterion, the initial conditions, and the interpolation constraints are usually quadratic and potentially non-convex in the local iterates and local gradients but they are \new{often} linear in the scalar products of these and in the function values, \new{see \eqref{eq:interpmuL} for example}. \new{In such cases}, letting the decision variables of the PEP be these scalar products and the function values, one can reformulate it as a semi-definite program (SDP) which can be solved efficiently. For this purpose, we define a vector of function values $\f$ and a Gram matrix $G$ that contains scalar products between all vectors, e.g. the local iterates $y_i^k, x_i^k \in \Rvec{d}$ and the local gradient vectors $g_i^k \in \Rvec{d}$.
\begin{align} \label{eq:FGdef}
  \f= [f_i^k]_{i \in \V,~k \in I = \{0,\dots,\K,*\}}  \new{~\in \Rvec{\N(\K+2)}},
   \hspace{1.2cm} G = P^TP \new{~\in \Rmat{3\N(\K+2)}{3\N(\K+2)}}, \text{ with } P = \qty[y_i^k~ x_i^k~ g_i^k]_{i \in \V, k \in I}. \\[-5.5mm]
\end{align} 
By definition, $G$ is symmetric and positive semidefinite.
Moreover, every matrix $G \succeq 0$ corresponds to a matrix $P$ whose number of rows is equal to $\text{rank}~G$. It has been shown in \cite{PEP_composite} that the reformulation of a PEP, such as \eqref{eq:PEP}, with $G$ as variable is lossless provided that we use necessary and sufficient interpolation constraints for $\F$ and that we do not impose the dimension $d$ of the worst-case, and indeed look for the worst-case over all possible dimensions (imposing the dimension would correspond to adding a typically less tractable rank constraint on $G$).
The resulting SDP-PEP thus takes the form of
\begin{align} \label{eq:SDP_PEP}
  w(\Setting_\N) \quad = \hspace{1cm} & \underset{\f, G}{\mathrm{max}} \quad \mc{P}(\f,G) &&  \\
 &\text{ s.t.} \qquad   G \succeq 0, && \\
 & \quad \f,G \quad \text{satisfy} && \hspace{-5cm} \text{algorithm constraints for $\Al$} \\[-1mm]
 &&& \hspace{-5cm} \text{interpolation constraints for $\F$}, \\[-1mm]
 &&& \hspace{-5cm} \text{initial conditions $\I$}, \\[-1mm]
 &&& \hspace{-5cm} \text{optimality condition for \eqref{opt:dec_prob}},
\end{align}
\new{where the objective function $\P$ and the constraints are linear in $\f$ and $G$.}
\new{The performance criterion function $\P$ in \eqref{eq:SDP_PEP} corresponds to the same function as the function $\P$ in \eqref{eq:PEP}, up to a change of variables. We slightly abuse notations to denote these two functions with the same symbol $\P$, for readibility.} 

Such an SDP formulation \eqref{eq:SDP_PEP} is convenient because it can be solved numerically to global optimality, leading to a tight worst-case bound, and it also provides the worst-case solution over all possible problem dimensions.
We refer the reader to \cite{PEP_composite} for more details about the SDP formulation of a PEP, including ways of reducing the size of matrix $G$. However, the dimension of $G$ always depends on the number of iterations $\K$, and on the number of agents $\N$.

From a solution $G$, $\f$ of the SDP formulation, we can construct a solution for the variables $\{y_i^k,x_i^k,g_i^k,f_i^k\}_{i \in \V, k \in I}$, e.g. using the Cholesky decomposition of $G$.
For each agent $i\in\V$, its resulting set of points satisfies the interpolation constraints, so we can construct its corresponding worst-case local function from $\F$ interpolating these points.
Proposition \ref{prop:GramPEP} below states sufficient conditions under which a PEP on a distributed optimization method can be formulated as an SDP. These conditions are satisfied in many PEP settings, allowing to formulate and solve the SDP PEP formulation for any distributed first-order methods from $\Ad$, with many classes of local functions, initial conditions, and performance criteria, see \cite{PEP_composite,PEP_dec,colla2024computer}. The proposition uses the following definition.\smallskip
\begin{definition}[Gram-representable \cite{PEP_composite}] \label{def:Gram-rep}
Consider a Gram matrix $G$ and a vector $\f$, as defined in \eqref{eq:FGdef}. We say that
an expression, such as a constraint or an objective, is linearly (resp. LMI) Gram-representable if it can be expressed using a finite set of linear (resp. LMI) constraints  or expressions involving (part of) $G$ and $\f$.
\end{definition} \smallskip
\begin{proposition}[\hspace{-0.05mm}{\cite[Proposition 2.6]{PEP_composite}}] \label{prop:GramPEP}
  Let $w(\Setting_\N = \{\N,\Al,\K,\P,\F, \I,\W\})$ be the worst-case performance of the execution of $\K$ iterations of distributed algorithm $\Al$ with $\N \ge 2$ agents, with respect to the performance criterion $\P$, valid for any starting point satisfying the set of initial conditions $\I$, any local functions in a given set of functions $\F$ and any averaging matrix in the set of matrices $\W$.
  If $\mc{A},\mc{P},\mc{I}$ and interpolation constraints for $\mc{F}$ and for $\W$ are linearly (or LMI) Gram representable, then the computation of $w(\Setting_\N)$ can be formulated as an SDP, with $G \succeq 0$ and $\f$ as variables, see \eqref{eq:SDP_PEP}.
\end{proposition}
\begin{remark} 
In \cite{PEP_composite}, Definition \ref{def:Gram-rep} and Proposition \ref{prop:GramPEP} were only formulated for linearly Gram-representable constraints, but their extensions to LMI Gram-representable constraints are direct and were already introduced in \cite{PEP_dec} in which LMI constraints are used to represent consensus steps in PEP. 
\end{remark} \smallskip

As briefly mentioned, when the set of possible averaging matrices $\W$ contains only one given matrix $W$, i.e. we know the averaging matrix used, the PEP framework presented here is complete and provides worst-case bounds specific to this value of $W$. However, the literature on distributed optimization generally provides bounds on the performance of an algorithm that are valid over a larger set of averaging matrices $\W$, typically based on its spectral properties, see \cite{DGD} for a survey. Such bounds allow better characterization of the general behavior of the algorithm and can be applied in a wider range of settings. Obtaining such bounds with PEP requires having a tractable representation of the set of all pairs of variables $\{(\x^k, \y^k)\}$ that can be involved in consensus steps with the same matrix from the given matrix class. 
Our previous work \cite{PEP_dec} proposes and uses necessary constraints for describing the set of symmetric and doubly-stochastic matrices with a given range on the non-principal eigenvalues. The next section describes these constraints and analyzes their sufficiency.

\subsection{Representation of consensus steps in PEP} \label{sec:consensusPEP}
In PEP for a decentralized optimization algorithm $\Al \in \Ad$, with $\N \ge 2$ agents, we would like to find Gram-representable constraints to represent all the pairs of points $\{ \x^k, \y^k \}$ such that
\begin{align}
  \y^k = (W \kron I_d) \x^k, \quad \text{for $k=1,\dots,\K$}, \qquad \text{with }~ W \in \Wcl{\lm}{\lp}, \label{eq:cons_1}
\end{align}
where $\Wcl{\lm}{\lp}$ is the following set of symmetric averaging matrices, with fixed bounds $\lm, \lp \in (-1,1)$ on the eigenvalues (except for $\lam_1=1$):
\begin{equation} \label{eq:Wcl}
  \hspace{-3mm} \Wcl{\lm}{\lp} = 
  \left\{\!\begin{aligned}
      &W^T = W  \\
      W \in \Rmat{\N}{\N}: \quad &\lam_1(W) = 1,~ v_1(W) = \mathbf{1}/{\scriptstyle \sqrt{\N}} \\
      &\lm \le \lam_\N(W) \le \dots \le \lam_2(W) \le \lp
  \end{aligned} \right\}.
\end{equation}
Such a set is often used to derive theoretical bounds on distributed algorithms \cite{DGD}. You may notice that it is not restricted to non-negative matrices, even if the literature often assumes non-negativity through doubly-stochasticity. 
We motivate this choice by two arguments. Firstly, non-negativity is not needed for the convergence of pure consensus steps \cite{xiaoboyd2004fast}. Secondly, among results using doubly-stochasticity, we found no results exploiting the non-negativity of the matrix, see e.g. \cite{DGD, EXTRA, NIDS}. Such results are thus about \emph{generalized doubly-stochastic} matrices \cite{gds}, which refers to matrices whose rows and columns sum to one, as the ones contained in $\Wcl{\lm}{\lp}$.
 
A set of pairs of points $\{(\x^k,\y^k)\}$ satisfying the consensus steps constraints \eqref{eq:cons_1} is called $\Wcl{\lm}{\lp}$-interpolable.
\begin{definition}[$\Wcl{\lm}{\lp}$-interpolability] \label{def:Winterp} Let $I$ be a finite set of indices. A set of pairs $\{(\x^k,\y^k)\}_{k\in I}$ is $\Wcl{\lm}{\lp}$-interpolable if, \vspace{-1mm}
  \begin{align}
      \exists W\in \Wcl{\lm}{\lp}~:~ \y^k = (W \kron I_d) \x^k \quad \text{ for all $k \in I.$} \\[-7mm]
  \end{align}
\end{definition}
\noindent In this definition, the pairs of points $\{(\x^k,\y^k)\}_{k\in I}$ are all linked by the same averaging matrix $W \in \Wcl{\lm}{\lp}$. Thus, settings or algorithms using different averaging matrices work with different sets of pairs of points. \smallskip


This section first shows why there is no Gram-representable constraints to characterize $\Wcl{\lm}{\lp}$-interpolability of a set of pairs of points $\{ \x^k, \y^k \}$ involved in consensus steps of the form of \eqref{eq:cons_1}, preventing its tight representation in an SDP PEP formulation. Then we present a relaxed Gram-representable description of $\Wcl{\lm}{\lp}$-interpolable pairs of points and characterize the points allowed in this relaxed description. Finally, we also explain how it can be used to build the \textit{spectral agent-dependent PEP formulation}, which provides worst-case guarantees for a decentralized algorithm, that are valid for any averaging matrix from $\Wcl{\lm}{\lp}$. 

\subsubsection{The intractable quest for tightness} 
It would be ideal to find necessary and sufficient conditions for a set of pairs $\{(\x^k,\y^k)\}$ to be $\Wcl{\lm}{\lp}$-interpolable, see Definition \ref{def:Winterp}. Moreover, the conditions should be Gram-representable, i.e., linear in the scalar products of $\x_i$ and $\y_i$, to be able to use them in an SDP PEP formulation (see Proposition \ref{prop:GramPEP}).
In this subsection, we will show, by a non-convexity argument, that there are no tight conditions for $\Wcl{\lm}{\lp}$-interpolability that are Gram-representable.
\begin{theorem} \label{thm:noGram}
  There is no Gram-representable necessary and sufficient conditions for $\Wcl{\lm}{\lp}$-interpolability of a set of  pairs of variables
  $\{(\x^k,\y^k)\}_{k\in I}$. 
\end{theorem}
The proof of this theorem relies on a lemma showing the non-convexity of the set of Gram matrices we are considering. To define this set properly, 
we consider matrices $\Px,\Py \in \Rmat{d}{\N \K}$ aggregating consensus iterates of $\N$ agents over $\K$ consensus steps. 
\begin{definition}[$\G{\lm}{\lp}$] \label{def:Gram}
  We define $\G{\lm}{\lp}$ as the set of symmetric and positive semidefinite Gram matrices of the form \vspace{-3mm}
  \begin{align} \label{eq:Gram}
    G_c = \begin{bmatrix} \Px^T\Px & \Px^T\Py \\ \Py^T\Px & \Py^T\Py \end{bmatrix}, \quad \parbox[c]{7.5cm}{
     where \quad $\Px = [\overbrace{x_1^0 \dots x_\N^0}^{\x^0} \dots x_1^\K \dots x_\N^\K ] \in \Rmat{d}{\N \K},$ \\
     \hspace*{1.27cm} $\Py = [\underbrace{y_1^0 \dots y_\N^0}_{\y^0} \dots y_1^\K \dots y_\N^\K ]  \in \Rmat{d}{\N \K},$
    } \\[-8mm]
\end{align}
such that the set of pairs $\{(\x^k,\y^k)\}_{k=1,\dots,\K}$ is $\Wcl{\lm}{\lp}$-interpolable, in the sense of Definition \ref{def:Winterp}.
\end{definition}
In practice, in the PEPs, the Gram matrices can involve other variables than those of the consensus (e.g., gradients). We can therefore consider $G_c$ as a submatrix of the full Gram matrix, which contains only the scalar products related to consensus steps. \smallskip

\begin{lemma} \label{lem:nonconvex} Let $\lm < \lp \in (-1,1)$. The set $\G{\lm}{\lp}$ is non-convex. \vspace{-2mm}
\end{lemma}
\begin{proof} 
We build two Gram matrices $G_1, G_2 \in \G{\lm}{\lp}$ and show that their combination $G_3 = \frac{1}{2}( G_1 + G_2)$ is not in $\G{\lm}{\lp}$. We build a counter-example for one consensus step, i.e., for $\K=1$, and then we explain why it can also be used for $\K > 1$. In the set $\G{\lm}{\lp}$, there is no constraint on the rank $d$ of the matrices, so we choose to build $G_1$ and $G_2$ with rank $d=1$, as follows
\begin{align} \label{eq:defxW}
  &\text{let } &
  \x &= \small{\begin{bmatrix}
    \N & 0 & \dots & 0
    \end{bmatrix}^T} \in \Rvec{\N},  \\ 
  &&[W_1]_{ij} &= \begin{cases} 
      \frac{1+(\N-1)\lm}{\N}  & \text{for $i=j$} \\ 
      \frac{1-\lm}{\N}        & \text{for $i\ne j$}
    \end{cases}, & 
    [W_2]_{ij} &= \begin{cases} 
      \frac{1+(\N-1)\lp}{\N}  & \text{for $i=j$} \\ 
      \frac{1-\lp}{\N}        & \text{for $i\ne j$}
    \end{cases}, \\
   &\text{and } &  \y_1 &=  W_1 \x = 
    \footnotesize{\begin{bmatrix} 
      1+(\N-1)\lm \\ 
      1-\lm \\[-1mm] \vdots \\ 1-\lm 
    \end{bmatrix}}, & 
    &\hspace{-10mm} \y_2 =  W_2 \x =  
    \footnotesize{\begin{bmatrix} 
      1+(\N-1)\lp \\ 
      1-\lp \\[-1mm] \vdots \\ 1-\lp 
    \end{bmatrix}}, \label{eq:defy}
  \end{align}
  we build $G_1$ and $G_2$ as 
  \begin{align}
    G_1 &= \begin{bmatrix} \x\x^T & \x\y_1^T \\ \y_1\x^T & \y_1\y_1^T \end{bmatrix}, \quad
    &  G_2 &= \begin{bmatrix} \x\x^T & \x\y_2^T \\ \y_2\x^T & \y_2\y_2^T \end{bmatrix}. \label{eq:defG}
\end{align}
\sloppy The eigenvalues of $W_1$ are $\{\lm,\dots,\lm, 1\}$, and those of $W_2$ are $\{\lp, \dots, \lp, 1\}$, so that $W_1, W_2 \in \Wcl{\lm}{\lp}$. Therefore, by construction, both pairs $(\x,\y_1)$ and $(\x,\y_2)$ are $\Wcl{\lm}{\lp}$-interpolable, and so $G_1$ and $G_2$ are in $\G{\lm}{\lp}$, and can be combined as $G_3 = \frac{1}{2}( G_1 + G_2)$,
\begin{equation} 
  G_3 =  \begin{bmatrix} \x\x^T & \frac{1}{2}\x(\y_{1} + \y_{2})^T \\
  \frac{1}{2}(\y_{1} + \y_{2})\x^T & \frac{1}{2}(\y_{1}\y_{1}^T + \y_{2}\y_{2}^T)
  \end{bmatrix} = \begin{bmatrix} \Pxi{3}^T\Pxi{3} & \Pxi{3}^T\Pyi{3} \\ \Pyi{3}^T\Pxi{3} & \Pyi{3}^T\Pyi{3}
  \end{bmatrix}.\label{eq:combG}
  \end{equation}
We show that, for any dimension $d_3 > 0$, there is no $\Pxi{3}, \Pyi{3} \in \Rmat{d_3}{\N}$ satisfying \eqref{eq:combG} and such that $G_3 \in \G{\lm}{\lp}$. Let us proceed by contradiction. Suppose there is some $\Pxi{3}, \Pyi{3}$ satisfying \eqref{eq:combG} and 
\begin{equation} \label{eq:proofnonconvex}
  \Pyi{3}^T = W_3 \Pxi{3}^T, \quad \text{for $W_3 \in \Wcl{\lm}{\lp}$}.
\end{equation}
Then, the block $\Pxi{3}^T\Pxi{3} = \x\x^T$ has rank one, because $\x \in \Rvec{n}$, and thus the block $\Pyi{3}^T\Pyi{3}$, which can be expressed as $\Pyi{3}^T\Pyi{3} = W_3 \Pxi{3}^T \Pxi{3} W_3$ using \eqref{eq:proofnonconvex}, also has rank one\footnote{This is due to $\mathrm{rank}(AB) \le \min(\mathrm{rank}(A), \mathrm{rank}(B))$.}.
However, the block is also defined as $\Pyi{3}^T\Pyi{3} = \frac{1}{2}(\y_{1}\y_{1}^T+ \y_{2}\y_{2}^T)$ from \eqref{eq:combG}, which has rank 2 for any $\lm \ne \lp \in (-1,1)$, because vectors $\y_1$ and $\y_2$ \eqref{eq:defy} are linearly independent in that case.
By this contradiction argument, we have proved that $ G_3 = \frac{1}{2}( G_1 + G_2)$ is not in $\G{\lm}{\lp}$, meaning that this set is not convex for $\K=1$. 

When we consider multiple consensus steps using the same matrix, i.e., when $\K > 1$, the counter-example developed above for $\K=1$ can still be chosen for one of the steps, and the Gram matrices $G_1$, $G_2$, and $G_3$ would then be submatrices of the Gram matrices in $\G{\lm}{\lp}$, showing the non-convexity of $\G{\lm}{\lp}$ for $\K > 1$ as well. \vspace{-1mm} 
\end{proof}

\begin{proof}[Proof of Theorem \ref{thm:noGram}]
\new{  By definition, Gram-representable constraints define a convex subset of symmetric and semidefinite Gram-matrices. However, by Lemma \ref{lem:nonconvex}, we know that set $\G{\lm}{\lp}$ of Gram-matrices such that the pairs of points $\{\x^k,\y^k\}$ are $\Wcl{\lm}{\lp}$-interpolable is a non-convex. Therefore, we know that there exists no Gram-representable constraints that tightly describe all the points $\{\x^k,\y^k\}$ that are $\Wcl{\lm}{\lp}$-interpolable.}
\end{proof}

\subsubsection{A relaxed formulation} \label{sec:PEPrelax} 
In previous work \cite{PEP_dec}, necessary constraints have been used to represent the effect of consensus steps in PEPs. These constraints describe a convex set of Gram matrices that contains the non-convex set $\G{\lm}{\lp}$, and can therefore be exploited in PEP to derive upper bounds on the worst-case performance of distributed algorithms. However, the precise meaning of these constraints was not known. In this subsection, we show that these constraints correspond to tight interpolation constraints for a larger class of steps, containing the consensus step as a particular case.

\paragraph*{Generalized consensus steps} \new{The consensus steps \eqref{eq:cons_1} can be seen as a linear operation on vectors $\x^k \in \Rvec{\N d}$, with a matrix $M \in \Rmat{\N d}{\N d}$ that has a specific Kronecker structure: 
\begin{align} 
  \y^k = M \x^k, \quad \text{for $k=1,\dots,\K$} \qquad \text{with }~ M = W \kron I_d ~\text{ and }~ W \in \Wcl{\lm}{\lp}. 
\end{align}
By definition, matrix $M$ has the same eigenvalues as $W \in \Wcl{\lm}{\lp}$ \eqref{eq:Wcl}, with multiplicity $d$ times larger. Therefore, it has $d$ eigenvalues equal to $1$, and their associated eigenvectors form a basis of the consensus subspace~$\C$.
\begin{definition}[Consensus subspace] \label{def:cons_space}
  The consensus subspace of $\Rvec{\N d}$ is denoted $\C$ and is defined as \vspace{-1mm}
  \begin{align} 
    \C = \left\{ \x \in \Rvec{\N d} ~|~ x_1 = \dots = x_\N \in \Rvec{d} \right\}. \\[-7mm]
  \end{align}
  The dimension of $\C$ is $|\C| = d$.
  The orthogonal complement of $\C$ is denoted $\Cp$ and has dimension $|\Cp| = (\N-1)d$:\vspace{-1mm} 
  \begin{align} 
    \Cp = \left\{ \x \in \Rvec{\N d} ~|~ \x^T\y = 0, \text{ for all } \y \in \C  \right\}. \\[-7mm]
  \end{align}
\end{definition}
We consider generalized consensus steps where the Kronecker structure of $M$ is relaxed, so that the vectors $\x^k$ are multiplied by a full matrix with similar properties: \vspace{-1mm}
\begin{align}
  \y^k = M \x^k, \quad \text{for $k=1,\dots,\K$}, \qquad \text{with }~ M \in \Mcl{\lm}{\lp}, \label{eq:gen_cons} \\[-7mm]
\end{align}
where $\Mcl{\lm}{\lp}$ is a set of symmetric matrices that have the same properties as $W \kron I_d$ with $W \in \Wcl{\lm}{\lp}$, but without constraint on their structure. 
In particular, matrices in $\Mcl{\lm}{\lp}$ are still symmetric, generalized doubly stochastic and with the same range of eigenvalues. The set $\Mcl{\lm}{\lp}$ is formally defined as follows, for a given range of eigenvalue $\lm, \lp \in (-1,1)$, $\lm \le \lp$:
\begin{align} \label{eq:Mcl}
  \Mcl{\lm}{\lp} = 
  \left\{ \! \begin{aligned}
    &M^T = M \\
    ~ M \in \Rmat{\N d}{\N d}: \quad&\lam_{1}(M) = \dots = \lam_{|\C|}(M) = 1,\\[-2mm]
    &\text{\small{with eigenvectors forming a basis of $\C$}} \\
    &\lm \le \lam_{\N d}(M) \le \dots \le \lam_{|\C|+1}(M) \le \lp
    \end{aligned}
   \right\},
\end{align}
We say that a matrix $M \in \Mcl{\lm}{\lp}$ preserves the consensus subspace $\C$ as we have $M \mathbf{v} = \mathbf{v}$ for any $\mathbf{v} \in \C$. 
Relaxation of the Kronecker structure in a consensus step allows different weights to be applied to each component of the agent vectors, and different components to be mixed. \vspace{-1mm} }

\paragraph*{Interpolation constraints for $\Mcl{\lm}{\lp}$ \vspace{-0.5mm}}
We provide tight interpolation constraints on a set of pairs of points $\{\x^k,\y^k\}$ for such generalized consensus steps \eqref{eq:gen_cons}, hence providing necessary constraints for describing classical consensus steps \eqref{eq:cons_1}.
Interpolation constraints are necessary and sufficient conditions for a set of pairs of points $\{\x^k,\y^k\}$ to be interpolable by a matrix $M \in \Mcl{\lm}{\lp}$.
\begin{definition}[$\Mcl{\lm}{\lp}$-interpolability]\label{def:Minterp} Let $I$ be a finite set of indices. A set of pairs of points $\{(\x^k,\y^k)\}_{k\in I}$ is $\Mcl{\lm}{\lp}$-interpolable if, \vspace*{-1mm}
    \begin{align}
        \exists M\in \Mcl{\lm}{\lp}~:~ \y^k = M \x^k \quad \text{ for all $k \in I$.} \\[-7mm]
    \end{align}
\end{definition}
This definition is related to the interpolation of a symmetric linear operator, developed in \cite{bousselmi2023interpolation}. In this work, we focus on linear operators preserving a given subspace, namely the consensus subspace $\C$. 
Based on results from \cite{PEP_dec} and \cite{bousselmi2023interpolation}, we give necessary and sufficient conditions to have $\Mcl{\lm}{\lp}$-interpolability of a set of pairs of points.
\new{Before stating these interpolation conditions, we first introduce a few notions. We aggregate all the vectors $\{(\x^k,\y^k)\}_{k\in I}$ in matrices $X, Y \in \Rmat{\N d}{\K}$: \vspace{-1mm} 
\begin{align} 
  X = [\x^1 \dots \x^\K] \quad \text{ and } \quad Y = [\y^1 \dots \y^\K] \\[-7mm]
\end{align}}
\new{We also consider the decomposition of matrices $X$ and $Y$ in two orthogonal terms: \vspace{-1mm}
\begin{align} 
   X = \Xb + \Xp \text{ and } Y = \Yb + \Yp, \\[-7mm]
\end{align}
where $\Xb, \Yb \in \Rmat{\N d}{\K}$ and $\Yp, \Xp \in \Rmat{\N d}{\K}$ are obtained by projecting each column of $X$ and $Y$ respectively on the consensus subspace $\C$ and on its orthogonal subspace $\Cp$. 
This decomposition allows decoupling the effect of the generalized consensus steps \eqref{eq:gen_cons} as follows: \vspace*{-1mm}
\begin{align}
  Y = MX = M(\Xb+\Xp) = \Xb + M\Xp  \quad \Leftrightarrow \quad \Yb = \Xb ~\text{ and }~ \Yp = M\Xp,
\end{align}
where $M \Xb = \Xb$ because $\Xb \in \C$ and $M \in \Mcl{\lm}{\lp}$.}

\begin{theorem}[$\Mcl{\lm}{\lp}$ interpolation constraints] \label{thm:sufficiency}
Let $I$ be a set of $\K$ indices and $\lm,\lp \in (-1,1)$, $\lm\le\lp$.
A set of pairs of points \new{$\{(\x^k,\y^k)\}_{k\in I}$ is $\Mcl{\lm}{\lp}$-interpolable} if, and only if, \vspace{-3.5mm}
 \begin{align}
    \Xb &= \Yb, \label{eq:subset_pres} \\
    (\Yp-\lm\Xp)^T(\Yp-\lp\Xp) & \preceq 0, \label{eq:SDP_cons} \\
    \Xp^T\Yp &= \Yp^T\Xp, \label{eq:sym_cons}
 \end{align}
 where $\Xb, \Yb \in \Rmat{\N d}{\K}$ contains columns of $X$ and $Y$ projected onto the subspace \new{$\C$}, and $\Xp, \Yp \in \Rmat{\N d}{\K}$ the columns projected on its orthogonal complement \new{$\Cp$}.
\end{theorem} \smallskip
\begin{remark}
  \new{Theorem \ref{thm:sufficiency} is stated for the case where symmetric linear operators from $\Mcl{\lm}{\lp}$ preserve the consensus subspace $\C$ (Definition \ref{def:cons_space}), but it actually holds for any general preserved subspace $\S \subseteq \Rvec{\N d}$. Indeed, the proof below remains valid if we replace the consensus subspace $\C$ by any subspace $\S$ (and adapt the dimensions where needed).}
\end{remark}
\begin{proof}[Proof of Theorem \ref{thm:sufficiency}]
  We need to prove that conditions \eqref{eq:subset_pres}, \eqref{eq:SDP_cons} and \eqref{eq:sym_cons} are necessary and sufficient for the existence of a matrix $M \in \Mcl{\lm}{\lp}$ such that $Y=MX$.
  The proof relies on the use of \cite[Theorem 3.3]{bousselmi2023interpolation} which proved the necessity and sufficiency of interpolation constraints similar to \eqref{eq:SDP_cons} and \eqref{eq:sym_cons}, for general linear operators, without any preserved subspace. 
  In our case, we know that the matrix $M$ has no impact on the consensus subspace $\C$ and only acts on the part in $\Cp$. We thus need to use \cite[Theorem 3.3]{bousselmi2023interpolation} on this part only. Decoupling both parts can be done using an appropriate change of variables: \vspace*{-0.5mm}
  \begin{align} \label{eq:change_of_var_proof} \tilde{X} = \begin{bmatrix} \Xbc \\ \Xpc \end{bmatrix} = \begin{bmatrix} \Qb^T \\ \Qp^T \end{bmatrix} X = Q^TX, \\[-6.5mm]
  \end{align}
  where $Q = [\Qb~\Qp] \in \Rmat{\N d}{\N d}$ is an orthogonal matrix of change of variables and is split into two sets of columns: $\Qb \in \Rmat{\N d}{|\C|}$ form a basis of the consensus subspace $\C$ and $\Qp \in \Rmat{\N d}{|\Cp|}$ a basis of its orthogonal complement $\Cp$.
  The new variable $\tilde{X}$ has thus two sets of components, $\Xbc \in \Rmat{|\C|}{\K}$ describing the coordinates along the subspace $\C$, and $\Xpc \in \Rmat{|\Cp|}{\K}$ describing the coordinates along the orthogonal complement $\Cp$.  
  By definition, $\Xb$ and $\Yb$ have components only along $\C$ and $\Xp$ and $\Yp$ only along $\Cp$: \vspace*{-1mm}
  \begin{align} 
    \Xb &= Q\begin{bmatrix} \Xbc \\ 0 \end{bmatrix},&
    \Yb &= Q\begin{bmatrix} \Ybc \\ 0 \end{bmatrix}, & \text{and}&& 
    \Xp &= Q\begin{bmatrix} 0 \\ \Xpc \end{bmatrix}, & 
    \Yp &= Q\begin{bmatrix} 0 \\ \Ypc \end{bmatrix}. \\[-7mm]
    \label{eq:change_of_var_split}
\end{align}
By applying this change of variable to condition \eqref{eq:subset_pres}, we obtain $\Xbc = \Ybc$. The two conditions are well equivalent because the change of variable is reversible.
 By applying this change of variable \eqref{eq:change_of_var_split} to conditions \eqref{eq:SDP_cons} and \eqref{eq:sym_cons}, since $Q$ is orthogonal, i.e., $Q^TQ = I$, we obtain 
  \begin{align}
    (\Ypc-\lm\Xpc)^T(\Ypc-\lp\Xpc) & \preceq 0, \label{eq:yyperp} \\
    \Xpc^T\Ypc &= \Ypc^T\Xpc. \label{eq:xyperp} \\[-6.5mm]
 \end{align}
 As the change of variable is reversible, the two sets of constraints are equivalent. By \cite[Theorem 3.3]{bousselmi2023interpolation}, these two constraints are necessary and sufficient for the existence of a symmetric matrix $\Wpc \in \Rmat{|\Cp|}{|\Cp|}$ with eigenvalues between $\lm$ and $\lp$ such that $\Ypc = \Wpc \Xpc$. Therefore, we have proved that conditions \eqref{eq:subset_pres}, \eqref{eq:SDP_cons} and \eqref{eq:sym_cons} are equivalent to \vspace*{-1mm}
 \begin{equation} \label{eq:equiv_cst}
  \begin{bmatrix} 
  \Ybc \\ \Ypc \end{bmatrix} = \begin{bmatrix} I_d& 0 \\ 0& \Wpc \end{bmatrix} \begin{bmatrix} \Xbc\\\Xpc \end{bmatrix} ~~ \text{with $\Wpc=\Wpc^T$,~~ $\lam(\Wpc) \in [\lm,\lp]$}.
 \end{equation}
 We now need to come back to the initial variables $X$ and $Y$ and characterize the matrix $M$ in terms of $\Wpc$. 
 We can use the (reversible) change of variable \eqref{eq:change_of_var_proof} to express equation \eqref{eq:equiv_cst} in terms of $Y$ and $X$ 
  \begin{align} \label{eq:equiv}
    Y = \underbrace{Q\begin{bmatrix} I_d& 0 \\ 0& \Wpc \end{bmatrix} Q^T}_{M} X \qquad \text{ with $\lam(\Wpc) \in [\lm,\lp]$}. \\[-8mm]
\end{align}
 By splitting $Q$ as $Q = [\Qb~\Qp]$, matrix $M$ can be written as
 \begin{equation}
    M = \Qb\Qb^T + \Qp\Wpc\Qp^T \qquad \text{ with $\lam(\Wpc) \in [\lm,\lp]$}.
  \end{equation}
 Since $\Wpc$ is symmetric, this expression shows that $M$ is also symmetric. 
 We now show that the symmetric matrix $M$ has $|\C|$ eigenvalues equal to $1$, with associated eigenvectors $\Qb$. Let $v_{\paral}$ denote any column of $\Qb$, i.e., any basis vector of the consensus subspace $\C$, we have well that \vspace{-1mm}
  \begin{equation} \label{eq:evparal}
     M v_{\paral} = \Qb\Qb^Tv_{\paral} + \Qp\Wpc\Qp^T v_{\paral} =  v_{\paral}.
 \end{equation}
 Indeed, $\Qp^T v_{\paral} = 0$ and $\Qb\Qb^Tv_{\paral} = v_{\paral}$, since $v_{\paral}$ is a column of $\Qb$ which satisfies $\Qp^T\Qb = 0$ and $\Qb^T\Qb = I$.  \\
 The other eigenvalues of $M$ are equal to the ones of $\Wpc$. Indeed, Let $(v_{\perp}, \lam_{\perp})$ be a pair of eigenvector and eigenvalue of matrix $\Wpc \in \Rmat{\Cp}{\Cp}$, then $(\Qp v_{\perp}, \lam_{\perp})$ is a pair of eigenvector and eigenvalue of matrix $M \in \Rmat{\N d}{\N d}$. This can be verified easily as \vspace{-1mm}
 \[ M \qty(\Qp v_{\perp}) =  \Qb\Qb^T \qty(\Qp v_{\perp}) + \Qp\Wpc\Qp^T\qty(\Qp v_{\perp}) = \lam_{\perp} \Qp v_{\perp}, \]
 since $\Qb^T\Qp = 0$ and $\Qp^T \Qp = I$.
\end{proof} \medskip 
 
In Theorem \ref{thm:sufficiency}, we have proved that conditions \eqref{eq:subset_pres}, \eqref{eq:SDP_cons} and \eqref{eq:sym_cons} are necessary and sufficient for the existence of a matrix $M \in \Mcl{\lm}{\lp}$ such that $Y=MX$. \new{We now give some interpretation of the constraints, explain why they are only necessary for $\Wcl{\lm}{\lp}$-interpolability, and describe the impacts in the PEPs for decentralized optimization.}


\paragraph*{Interpretation of the constraints}
\new{By Definition \ref{def:cons_space} of the consensus subspace, the orthogonal decomposition terms can be expressed using the agent averages:\vspace*{-1mm}  
\begin{align}
  \hspace*{2cm} \Xb &= \Xbl \kron \mathbf{1}_\N  & \Xp &= X - \Xb, \label{eq:Xdec1} \\
  \hspace*{2cm} \Yb &= \Ybl \kron \mathbf{1}_\N  & \Yp &= X - \Yb, \label{eq:Xdec2}
\end{align}
where $\Xbl,~\Ybl \in \Rmat{d}{\K}$ are agent average of $X$ and $Y$, i.e., $\overline{X}_{\cdot,k} = \frac{1}{\N} \sum_{i=1}^\N x_i^k$. Therefore, constraint \eqref{eq:subset_pres} is equivalent to $\Xbl = \Ybl$.
This constraint is related to the $1$ eigenvalues of the matrix and their corresponding eigenvectors, which form a basis for the consensus space $\C$. The constraint imposes that each vector $\x^k$ have the same agents average as $\y^k$, for each generalized consensus step \eqref{eq:gen_cons} ($k=1,\dots,\K$).}

\new{Interpretations of the linear matrix inequality (LMI) constraints \eqref{eq:SDP_cons} and \eqref{eq:sym_cons} are detailed in \cite{PEP_dec}. In summary, \eqref{eq:SDP_cons} requires the disagreement between the agents, measured by $\yc^T \yc$ for $y$ and $\xc^T \xc$ for $x$ to be reduced by a factor $\lam^2 \in [0,1)$ after each consensus step. Constraints \eqref{eq:SDP_cons} and  \eqref{eq:sym_cons} also allow linking different consensus steps to each other, via the impact of off-diagonal terms which exploit the fact that these steps use the same averaging matrix.}

\paragraph*{Necessary interpolation conditions for $\Wcl{\lm}{\lp}$}
\new{In Theorem \ref{thm:sufficiency}, we have proved that conditions \eqref{eq:subset_pres}, \eqref{eq:SDP_cons} and \eqref{eq:sym_cons} are necessary and sufficient for the existence of a matrix $M \in \Mcl{\lm}{\lp}$ such that $Y=MX$. By definition of $\Mcl{\lm}{\lp}$ and $\Wcl{\lm}{\lp}$, these conditions are necessary for the existence of an averaging matrix $W \in \Wcl{\lm}{\lp}$ such that $Y = (W\kron I_d)X$, as shown in the following corollary.  
\begin{corollary} \label{cor:conscons}
  Let $\{(\x^k,\y^k)\}_{k\in I}$ be a set of $\K$ pairs of points and $\lm,\lp \in (-1,1)$, $\lm\le\lp$.
  If $\{(\x^k,\y^k)\}_{k\in I}$ is $\Wcl{\lm}{\lp}$-interpolable (Definition \ref{def:Winterp}), then we have \vspace*{-2mm}
   \begin{align}
      \Xbl &= \Ybl, \label{eq:avg_pres2} \\
      (\Yp-\lm\Xp)^T(\Yp-\lp\Xp) & \preceq 0, \label{eq:SDP_cons2} \\
      \Xp^T\Yp &= \Yp^T\Xp, \label{eq:sym_cons2}
   \end{align}
   with $\Xbl, \Ybl \in \Rmat{d}{\K}$ and $\Xp, \Yp \in \Rmat{\N d}{\K}$ defined in \eqref{eq:Xdec1}-\eqref{eq:Xdec2}.
  \end{corollary}
\begin{proof}
  From a matrix $W \in \Wcl{\lm}{\lp}$, we can always build a matrix $M = (W\kron I_d) \in \Mcl{\lm}{\lp}$, meaning that any set of pairs of points that is $\Wcl{\lm}{\lp}$-interpolable is also $\Mcl{\lm}{\lp}$-interpolable. Therefore, the points $\{(\x^k,\y^k)\}_{k\in I}$ are $\Mcl{\lm}{\lp}$-interpolable and satisfy constraints \eqref{eq:avg_pres2}, \eqref{eq:SDP_cons2} and \eqref{eq:sym_cons2} by Theorem \ref{thm:sufficiency}. 
\end{proof}
The reverse statement does not hold since constraints \eqref{eq:avg_pres2}, \eqref{eq:SDP_cons2} and \eqref{eq:sym_cons2} guarantee the interpolation by a matrix $M \in \Mcl{\lm}{\lp}$, that may not have the Kronecker structure required to extract an averaging matrix $W \in \Wcl{\lm}{\lp}$. The full mask of $M \in \Mcl{\lm}{\lp}$ allows different weights to be applied to each component of the agent and can also mix different components together, which is not the case for consensus steps \eqref{eq:cons_1} that impose a sparse mask of the form $(W \kron I_d)$.}


\paragraph*{Relaxed Performance Estimations Problems} 
One can show that the interpolation constraints stated in Corollary \ref{cor:conscons} are LMI Gram-representable. The detailed proof is given in \cite[Theorem 1]{PEP_dec} but the main idea is to express the constraints with the scalar products of the local variables $x_i^k, y_i^k$. 
\new{Therefore, by Proposition \ref{prop:GramPEP}, these  constraints can be used in SDP PEP formulations. As explained in \cite{PEP_dec}, using necessary constraints \eqref{eq:avg_pres2}, \eqref{eq:SDP_cons2} and \eqref{eq:sym_cons2} to represent consensus steps of the form of \eqref{eq:cons_1} in PEP allows obtaining worst-case performance guarantees for a decentralized optimization algorithm $\Al \in \Ad$, that are valid for any averaging matrix $W \in \Wcl{\lm}{\lp}$. This formulation is therefore called the \textit{spectral agent-dependent PEP formulation} and its results the \textit{spectral worst-case}. This is an agent-dependent PEP formulation because the PEP depends on each agent individually.}

\new{When different averaging matrices are used for different sets of consensus steps in the algorithm, these constraints from Corollary \ref{cor:conscons} can be applied independently to each set of consensus steps.}

\new{The bounds obtained with this formulation have a priori no tightness guarantees because it considers larger sets of possible consensus iterates $\{\x^k,\y^k\}$, i.e., those that are $\Mcl{\lm}{\lp}$-interpolable. However, this relaxation of the Kronecker structure in consensus steps does not appear to impact the worst-case results, as already observed numerically in \cite{PEP_dec}.}
Moreover, one can always check the tightness of the PEP solution a posteriori, by checking if the provided solution is $\Wcl{\lm}{\lp}$-interpolable. If successful, this check also recovers an instance of the worst matrix $W \in \Wcl{\lm}{\lp}$, see \cite{PEP_dec} for details.


\paragraph*{Impact of the relaxation} \label{sec:impact_relax}
While the literature for distributed optimization usually assumes consensus steps of the form $\y = (W \kron I_d) \x$ with $W \in \Wcl{\lm}{\lp}$, the existing theoretical performance guarantees also appear to use non-tight description of these consensus steps. Indeed, we did not find any theoretical guarantee that exploits the Kronecker structure $(W \kron I_d)$ of the consensus steps. In particular, many results in decentralized optimization remain valid for generalized consensus steps $\y = M \x$ with $M \in \Mcl{\lm}{\lp}$ because the proofs generally rely on the convergence analysis of the consensus step, which is not impacted by the structure of $M \in \Mcl{\lm}{\lp}$, as shown in Proposition \ref{prop:Mcons}. This proposition extends classical results of consensus theory, to the use of matrices $M \in \Rmat{\N d}{\N d}$ that do not have the Kronecker structure of $(W \kron I_d)$. When $d=1$, we have $M=W \in \Rmat{\N}{\N}$ and no extension is needed. 
\begin{proposition} \label{prop:Mcons}
  Let $M \in \Mcl{\lm}{\lp} \subset \Rmat{\N d}{\N d}$ with $\lm, \lp \in (-1,1)$ and $\x^0 \in \Rvec{\N d}$ a vector stacking the variables $x_i^0 \in \Rvec{d}$ of each agent.
  If the vector $\x^0$ is updated with a series of general matrix steps $\x^{k+1} = M \x^{k}$, then the sequence of vectors $\x^k$ converges to the agent average consensus vector, i.e., a vector stacking $\N$ times the agent average $\overline{x}^0 = \frac{1}{\N}\sum_{i=1}^\N x_i^0$: \vspace{-1mm}
  \[\lim_{k\to\infty} \x^{k} = \mathbf{1}_\N \kron \overline{x}^0.\]
\end{proposition}
\begin{proof}
  Since $M \in \Mcl{\lm}{\lp}$, the matrix admits the following spectral decomposition: \vspace{-1mm}
  \begin{equation} \label{eq:Mev}
     M = V \Sigma V^T, \qquad \text{with} \qquad V_{1,\dots,d} = (\mathbf{1}_\N \kron I_d)/{\scriptstyle \sqrt{\N}},
  \end{equation}
  where $V$ is the orthogonal matrix of eigenvectors of $M$ and $\Sigma$ the diagonal matrix of eigenvalues of $M$. We define a new variable $\z^k = V^T \x^k$, for which the matrix step $\x^{k+1} = M \x^{k}$ becomes \vspace{-1mm}
  \begin{align*} \z^{k+1} = \Sigma \z^k. \\[-7mm] \end{align*}
  The $d$ first eigenvalues in $\Sigma$ are associated with eigenvectors $V_{1,\dots,d}$ and are equal to 1; all the others are smaller than 1 in absolute value. Therefore, \vspace{-1mm}
  \[ \lim_{k\to\infty} \z^k = \lim_{k\to\infty} \Sigma^k \z^0 = \begin{bmatrix} z_1^0 & \dots & z_d^0 & 0 & \dots & 0 \end{bmatrix}^T =\tilde{\z}^0,\]
  where vector $\tilde{\z}^0$ contains the $d$ first entries of $\z^0$ and $(\N-1)d$ zeros. Thus, the sequence $\x^k$ converges to $V \tilde{\z}^0$ which is the combination of the columns $V_{1,\dots,d}$ with coefficients given by the $d$ first entries of $\z^0$, i.e.,
  \[ \lim_{k\to\infty} \x^k = \lim_{k\to\infty} V \z^k = V \tilde{\z}^0 = V_{1,\dots,d} ~\z^0_{1,\dots,d}.\]
  Finally, by definition of variable $\z$, we have $\z^0 = V^T \x^0$, meaning that the $d$ first entries of $\z^0$ are given by
  $\z^0_{1,\dots,d} = V_{1,\dots,d}^T \x^0$, and therefore: \vspace{-1.5mm}
  \begin{align*} 
    \lim_{k\to\infty} \x^k = V_{1,\dots,d} V_{1,\dots,d}^T \x^0 = \frac{1}{\N} (\mathbf{1}_\N\mathbf{1}_\N^T \kron I_d) \x^0 = \mathbf{1}_\N \kron \overline{x}^0, \\[-9mm]
  \end{align*}
  where the second equality holds by definition of $V_{1,\dots,d}$ \eqref{eq:Mev}.
\end{proof}

\subsection{\new{Section summary: agent-dependent PEP}}
In summary, in this section, we have seen that the worst-case performance of $\K$ iterations of a decentralized optimization algorithm from $\Ad$ can be computed using an SDP reformulation of the Performance Estimation Problem (PEP) with a size growing with $\K$ and the number of agents $\N$ (see Proposition \ref{prop:GramPEP}). When the averaging matrix $W$ is given, the resulting PEP worst-case bound is tight. 
When computing a worst-case valid for all averaging matrices from $\Wcl{\lm}{\lp}$ (symmetric, generalized doubly-stochastic and with a range on non-principal eigenvalues), the resulting PEP is a relaxation, which may provide non-tight performance bounds. We have proved that some form of relaxation is inevitable if we want to obtain a convex formulation (see Theorem \ref{thm:noGram}). The constraints we provide to represent the consensus steps in PEP allow for more general matrix steps that can mix components of vectors to which they apply (see Theorem \ref{thm:sufficiency}). While this is not natural, it does not alter the convergence of pure consensus, and we do not expect it to be an important source of conservatism in performance limits. (see Proposition \ref{prop:Mcons}). Moreover, the tightness of the PEP solutions can be verified a posteriori.

In the rest of the paper, we will exploit symmetries of agents in this agent-dependent PEP formulation to obtain equivalent PEP formulations whose size is independent of the number of agents $\N$. We also characterize conditions under which the resulting worst-case is also independent of $\N$.
 
\section{Equivalence classes of agents in PEP} \label{sec:equicl}
To exploit agent symmetries in the agent-dependent PEP, let us define an equivalence relation between agents, along with the corresponding equivalence classes. These will be used to build our new compact PEP formulation whose size depends only on the number of equivalence classes.

Let $\f$ and $G=P^TP$ be a solution of an agent-dependent PEP, described in Section \ref{sec:agentPEP}. We can order their elements to regroup them depending on the agent to which they are associated:
\begin{equation} \label{eq:Fagents}
  \f= \begin{bmatrix}f_1^T & \dots & f_\N^T \end{bmatrix}, \quad \text{where $f_i = \qty[f_i^k]_{k \in I} \in \Rvec{q}$ is a vector with the $q$ function values of agent $i$,}
\end{equation}
Similarly, we write $P \in \Rmat{d}{\N p}$ as follows
\begin{equation} \label{eq:defPi}
  P = \begin{bmatrix}P_1 & \dots & P_\N \end{bmatrix}  \quad \text{where $P_i \in \Rmat{d}{p}$ contains the $p$ vector variables related to agent $i$,} 
\end{equation}
e.g. $P_i = \qty[y_i^k~ x_i^k~ g_i^k]_{k \in I}$,
and the Gram matrix $G \in \Rmat{\N p}{\N p}$ is thus defined as
\begin{equation} \label{eq:initialGram}
  G = P^TP =   \begin{bmatrix}P_1^TP_1 & P_1^TP_2 & \dots \\ P_2^TP_1 & \ddots & \\  \vdots& & P_\N^TP_\N \end{bmatrix} =  \begin{bmatrix}G_{11} & G_{12} & \dots \\ G_{21} & \ddots & \\  \vdots& & G_{\N\N} \end{bmatrix}.
\end{equation}
Therefore, diagonal blocks $G_{ii}$ are symmetric and off-diagonal blocks are such that $G_{ij} = G_{ji}^T$. 
We assume that each block $P_i$ (and $f_i$) contains the same type and number of variables for each agent $i$. The variables common to all agents, such as $x^*$, are copied into each agent block $P_i$. This also covers the case where agents are heterogeneous and hold different types or numbers of variables since we can always add variables as columns of $P_i$ (and $f_i$), even if agent $i$ does not use them. Moreover, each agent block $P_i$ (and $f_i$) has the same column order, enabling us to use the same coefficient vectors for column selection in each agent block. 
\begin{notation*}[Coefficient vector $\id_x \in \Rvec{p}$] A coefficient vector for a variable $x$ is denoted $\id_x \in \Rvec{p}$, and contains linear coefficients selecting the correct combination of columns in $P_i \in \Rmat{d}{p}$ to obtain vector $x_i \in \Rvec{d}$, i.e. $P_i \id_x = x_i$, for any $i = 1,\dots,\N$. This notation allows, for example, to write $x_i^Tx_i$ as $x_i^Tx_i = \id_x^T P_i^TP_i \id_x = \id_x^T G_{ii} \id_x$, for any agent $i$. Similarly, let $\id_{f(x)}$ be a vector of coefficients selecting the correct element in the vector $f_i$ such that $\id_{f(x)}^T f_i = f_i(x_i)$. 
These coefficient vectors will be used to explicit our new PEP formulations in Sections \ref{sec:symPEP} and \ref{sec:subsetsAgPEP}, as well as in Appendix \ref{ap:explicit}.
\end{notation*}

\new{Using this agent block notation for $\f$ and $G$, we can define permutations of agents in a PEP solution, leading to 
the definition of an equivalence relation between agents in performance estimation problems (Definition \ref{def:equiva} below).
\begin{notation*}[Permuted solutions]
  For any feasible solution $(\f, G)$ of an agent-dependent PEP formulation \eqref{eq:SDP_PEP}, written in the form of \eqref{eq:Fagents} and \eqref{eq:initialGram}, the solution obtained after permutation of agents $i, j \in \V$ in $\f$ and $G$ is denoted $(\f_{\Pi_{ij}}, G_{\Pi_{ij}})$. 
\end{notation*}
\begin{example*}[Permuted solutions]
  Let $(\f, G)$ be a feasible PEP solution written as \vspace{-1mm}
  \begin{align}
  \hspace{1.75cm}\f &= \begin{bmatrix}f_1^T & f_2^T & \dots & f_\N^T \end{bmatrix} & 
  G&=P^TP \qquad \text{with }\quad P = \begin{bmatrix}P_1 & P_2 & \dots & P_\N \end{bmatrix}.
  \end{align}
  Agents $1$ and $2$ can be swapped to obtain a permuted solution $(\f_{\Pi_{12}}, G_{\Pi_{12}})$: \vspace{-1mm}
  \begin{align}
    \hspace{1.7cm}\f_{\Pi_{12}} &= \begin{bmatrix}f_2^T & f_1^T & \dots & f_\N^T \end{bmatrix} &
    G_{\Pi_{12}}&=P_{\Pi_{12}}^TP_{\Pi_{12}} \qquad \text{with}\quad P_{\Pi_{12}} &= \begin{bmatrix}P_2 & P_1 & \dots & P_\N \end{bmatrix}.
    \end{align}
\end{example*}
\begin{definition}[Equivalent agents] \label{def:equiva} 
  Let $i, j \in \V$ be two distinct agents.
  We define an equivalence relation between the two agents $i \sim j$ if, for any feasible solution $(\f, G)$ \eqref{eq:Fagents}-\eqref{eq:initialGram} of an agent-dependent PEP formulation \eqref{eq:SDP_PEP}, the permuted solution $(\f_{\Pi_{ij}}, G_{\Pi_{ij}})$, with agent $i$ and $j$ swapped, is feasible for the PEP and provides the same performance value,   \vspace{-2mm}
  \[ \P(\f, G) = \P(\f_{\Pi_{ij}}, G_{\Pi_{ij}}). \]
\end{definition} }



Intuitively, this means that equivalent agents have exactly the same role in the algorithm and the performance setting.
By definition, this relation is reflexive, symmetric, and transitive, and thus corresponds well to an equivalence relation, which induces a partition into equivalence classes. \smallskip

\begin{definition}[Equivalence classes of agents] \label{def:equicl}
 Let $\V$ be a set of $n$ agents. The equivalence relation from Definition \ref{def:equiva} implies a partition $\T$ of set $\V$ into $\U$ equivalence classes (or subsets) of agents: $\T = \{\V_1,\dots,\V_\U\}$. Each class only contains agents that are equivalent to each other. The size of each class $\V_u$ is denoted $\N_u$ ($u=1,\dots,\U$). We also denote the equivalence class of a given agent $i$ by $\V_{u_i}$.
\end{definition} \smallskip

 Proposition \ref{prop:sym_sol_2} proves the existence of PEP solutions for which equivalent agents have equal blocks in the solution. This will be the building block of our compact PEP formulations in Sections \ref{sec:symPEP} and \ref{sec:subsetsAgPEP}. \smallskip

\begin{proposition}[Existence of symmetric solutions in PEP] \label{prop:sym_sol_2}
  Let $\f\in \Rvec{\N q}$ and $G \in \Rmat{\N p}{\N p}$ be any feasible solution of an agent-dependent \new{(and LMI Gram-representable)} PEP formulation for distributed optimization \eqref{eq:SDP_PEP}; and $\T = \{\V_1,\dots,\V_\U\}$ be the partition of $\V$ induced by the equivalence relation from Definition \ref{def:equiva}. There is a symmetrized agent-class solution $(\Fss, \Gss)$, with equal blocks for equivalent agents, providing another valid solution for the PEP, 
  \begin{align}
    \hspace*{-4mm} \Fss &= \qty[(f_1^s)^T \dots\, (f_\N^s)^T] \qquad \text{with }&& f_i^s = \new{\fa^{u_i} \defeq} \frac{1}{\N_{u_i}} \sum_{k \in \V_{u_i}} f_k, \hspace{3.5cm} \text{for all } i \in \V, \label{eq:Fsym} \\ 
    \hspace*{-4mm} \Gss &= \begin{bmatrix}G_{11}^s & G_{12}^s & \dots \\ G_{21}^s & \ddots & \\  \vdots& & G_{\N\N}^s \end{bmatrix}   
    \quad \text{with }
    &&
    G_{ij}^s = \begin{cases} 
      \new{\Ga^{u_i} \defeq} \frac{1}{\N_{u_i}} \sum_{k \in \V_{u_i}} G_{kk} & \text{for $i=j$,} \\
      \new{\Gr^{u_i} \defeq} \frac{1}{\N_{u_i}(\N_{u_i}-1)} \sum_{k \in \V_{u_i}} \sum_{\substack{l \in \V_{u_i} \\ l\ne k }} G_{kl} & \text{\new{for $j \ne i,~ j \in \V_{u_i}$}}, \\
      \new{\Gc^{u_iu_j} \defeq} \frac{1}{\N_{u_i}} \sum_{k \in \V_{u_i}} G_{kj},  & \text{\new{for $j \ne i,~j \in \V\setminus\V_{u_i}$}},
    \end{cases} \label{eq:Gsym}
  \end{align}
  Moreover, this symmetrized solution $(\Fss, \Gss)$ provides the same performance value as $(\f,G)$
  \[ \P(\Fss, \Gss) = \P(\f,G).\]
\end{proposition}
\begin{proof}
  By definition, any permutation of equivalent agents $\Pi \in \Rmat{\N}{\N}$ applied on a given PEP solution $\f$, $G$ provides another valid PEP solution, with the same objective value:
  \[\f_\Pi = \f\qty(\Pi \otimes I_{q}), \qquad P_\Pi = P \qty(\Pi \otimes I_p), \qquad G_\Pi = P_\Pi^T P_\Pi = \qty(\Pi \otimes I_p)^TG\qty(\Pi \otimes I_p).\]
  A permutation of agents corresponds to a permutation of (sets of) columns and then the permutation matrix multiplies $\f$ and $P$ on the right. The Kronecker products $\kron$ are used to permute all the variables related to the same agent at the same time.

  \new{Since the PEP problem is LMI Gram-representable}, its objective is linear in $\f$ and $G$ and the combination of PEP solutions with the same objective value gives another PEP solution with the same objective value. We can thus construct a symmetrized PEP solution with the same objective value as $(\f,G)$ by averaging all the possible permuted solutions $(\f_{\Pi}, G_{\Pi})$. Each permutation exchanges only agents from the same equivalence class, and therefore, the average applies to each class separately:
  \begin{align*}
    f_i^s &= \frac{(\N_{u_i}-1)!}{\N_{u_i}!} \sum_{k \in \V_{u_i}} f_k = \frac{1}{\N_{u_i}} \sum_{k \in \V_{u_i}} f_k ~\new{\defeq \fa^{u_i}}, &\text{for all } i \in \V, \\ 
    G_{ii}^s &= \frac{(\N_{u_i}-1)!}{\N_{u_i}!} \sum_{k \in \V_{u_i}}G_{kk} = \frac{1}{\N_{u_i}} \sum_{k \in \V_{u_i}} G_{kk} ~\new{\defeq \Ga^{u_i}} &\text{for all } i \in \V,\\
    G_{ij}^s &= \frac{(\N_{u_i}-2)!}{\N_{u_i}!} \sum_{k \in \V_{u_i}} \sum_{\substack{l \in \V_{u_i} \\ l\ne k }} G_{kl} =  \frac{1}{\N_{u_i}(\N_{u_i}-1)} \sum_{k \in \V_{u_i}} \sum_{\substack{l \in \V_{u_i} \\ l\ne k }} G_{kl} ~\new{\defeq \Gr^{u_i}}
     &\text{for all $i\in \V$ and $j \ne i,~j \in \V_{u_i}$},\\
    G_{ij}^s &= \frac{(\N_{u_i}-1)!}{\N_{u_i}!} \sum_{k \in \V_{u_i}}G_{ik} = \frac{1}{\N_{u_i}} \sum_{k \in \V_{u_i}} G_{kj} ~\new{\defeq \Gc^{u_i u_j}} &\text{for all $i\in \V$ and $j \ne i,~j \in \V\setminus\V_{u_i}$}, \\[-9mm]
  \end{align*} 
\end{proof} 

\begin{definition}[\new{Blocks $\fa^u$, $\Ga^u$, $\Gr^u$ and $\Gc^{uv}$}] \label{def:blocks} The blocks repeating in $\f^s$ \eqref{eq:Fsym} are denoted by $\fa^u$, for each equivalence class $u=1,\dots,\U$. For the symmetrized Gram matrix $ G^s$ \eqref{eq:Gsym}, we denote the three types of blocks composing it by $\Ga^u$, $\Gr^u$, and $\Gc^{uv}$, for all equivalence classes $u,v=1,\dots,\U$.
  \begin{itemize}
    \item Blocks $\Ga^u \in \Rmat{p}{p}$ (for $u=1,\dots,\U$) correspond to the blocks of the Gram matrix containing the scalar products between the variables of one agent from $\V_u$. These blocks lie on the diagonal of $G^s$ \eqref{eq:Gsym} and are symmetric by definition.  
    \item Blocks $\Gr^u \in \Rmat{p}{p}$ (for $u=1,\dots,\U$) correspond to the blocks of the Gram matrix containing the scalar products between variables of two different agents from the same class $\V_u$. These blocks are symmetric because they can be written as a sum of symmetric matrices 
    $\Gr^u = \frac{1}{\N_{u}(\N_{u}-1)} \sum_{k,l \in \V_{u}, k > l } (G_{kl} + G_{lk}^T)$. Moreover, they are only defined for classes with at least 2 agents ($\N_u \ge 2$).
    \item Blocks $\Gc^{uv} \in \Rmat{p}{p}$ (for $u\ne v=1,\dots,\U$) correspond to the blocks of the Gram matrix containing the scalar products between variables of two agents from different classes $\V_u$ and $\V_v$. These blocks are non-symmetric and are only defined when there are at least 2 different classes of agents ($\U\ge2$).
\end{itemize}
\end{definition} 

Therefore, we can solve the agent-dependent PEP problem, described in Section \ref{sec:agentPEP}, restricted to symmetric solutions of the form of Proposition \ref{prop:sym_sol_2}, without impacting the resulting worst-case value. The symmetry in the solutions allows working with a limited number of variables $\fa^u$, $\Ga^u$, $\Gr^u$, and $\Gc^{uv}$ in PEP, depending only on the (smaller) number of equivalent classes of agents $\U$. This requires reformulating all the elements of the agent-dependent PEP (objective and constraints) in terms of these variables. Section \ref{sec:symPEP} shows how to perform such a reformulation when all the agents are equivalent in the PEP, which occurs in many usual settings. This leads to a PEP formulation with only one equivalence class and hence whose size is independent of the number of agents. Section \ref{sec:subsetsAgPEP} generalizes reformulation techniques from Section \ref{sec:symPEP} to a general number of equivalence classes.

\section{Agent-independent PEP formulation when all agents are equivalent} \label{sec:symPEP}

This section focuses on cases where all agents are equivalent.
\begin{assumption}[Equivalent Agents] \label{ass:equiva}
  All the agents are equivalent in the Performance Estimation Problem applied to a distributed algorithm, in the sense of Definition \ref{def:equiva}. Therefore, there is only one equivalence class which is $\V$. 
\end{assumption}

In other words, this assumption means that all the agents can be permuted in a PEP solution without impacting its worst-case value and thus no agent plays a specific role in the algorithm or its performance evaluation. This assumption is satisfied in many usual performance evaluation settings $\Setting_\N$, for which all agents play an identical role in the algorithm and its performance evaluation. Under Assumption \ref{ass:equiva}, the symmetrized solution from Proposition \ref{prop:sym_sol_2} only has a few different blocks that repeat.
\begin{corollary}[Fully symmetric PEP solutions] \label{cor:fullsymsol}
  When all the agents ($\N\ge 2$) are equivalent, see Assumption \ref{ass:equiva}, \new{any solution $(\f,G)$ of an agent-dependent (and LMI Gram-representable) PEP formulation for distributed optimization \eqref{eq:SDP_PEP}} can be fully symmetrized, without impacting its worst-case value:
  \begin{align}
    \f^s &= \qty[\fa^T \dots \fa^T] & \text{with }  ~ \fa &= \frac{1}{\N} \sum_{i=1}^\N f_i ~~\in \Rvec{q},&& \label{eq:Fs} \\
    G^s &= \begin{bmatrix}\Ga & \Gr & \dots \\ \Gr & \ddots & \\  \vdots& & \Ga \end{bmatrix}   & \text{with } ~ \Ga &= \frac{1}{\N} \sum_{i=1}^\N G_{ii} ~~\in \Rmat{p}{p},& \hspace{-3mm} \text{and} ~
    \Gr &= \frac{1}{\N(\N-1)} \sum_{i=1}^\N\sum_{j\ne i}^\N G_{ij} ~\in \Rmat{p}{p}. \label{eq:Gs}
  \end{align}
\end{corollary}
\begin{proof} The results follow from applying Proposition \ref{prop:sym_sol_2} with $\V$ as the only equivalence class. Since there is only one equivalence class of agents, the Gram matrix only has two types of blocks, as explained in Definition \ref{def:blocks}.
\end{proof}

In this section, we consider the agent-dependent PEP formulation from Section \ref{sec:agentPEP}, with a given class of averaging matrix, and we constrain its solution to be fully symmetric, as expressed in \eqref{eq:Fs} and \eqref{eq:Gs}. \new{First, we explain how and when this leads to a compact PEP formulation with size independent of the number of agents. Secondly, we describe the implications of a fully symmetrized PEP solution on the worst-case local iterates and functions.}

\subsection{The PEP formulation restricted to fully symmetric solutions} \label{sec:PEPsymsol}

 When all agents are equivalent (Assumption \ref{ass:equiva}), Corollary \ref{cor:fullsymsol} shows that \new{restricting PEP to fully symmetric solutions} does not impact the value of the worst-case. We will see in Theorem \ref{thm:agent_indep_perf} that it allows writing the PEP in a compact form since the performance measure and the constraints can be expressed only in terms of the smaller blocks $\fa$, $\Ga$, and $\Gr$ that are repeating in $\f^s$ and $G^s$.
Theorem \ref{thm:agent_indep_perf} states sufficient conditions for which a PEP for distributed optimization can be made compact \new{(i.e. with size independent of $\N$)} and for which it provides a worst-case value independent of the number of agents $\N$.

\new{Before stating the theorem, we introduce two types of expressions or constraints that we can allow in PEP to obtain agent-independent worst-case values. Let us consider a simple example to motivate the first type of expressions.
\begin{example*}
If we consider a PEP for distributed optimization with the initial conditions $f_i(x_i^0) - f_i(x^*) \le 1$ for all $i \in \V$ and the performance measure $\P = \sum_{i=1}^\N (f_i(x_i^0) - f_i(x^*)),$
  then the worst-case value would scale with the number of agents $\N$ in the problem: if we consider more agents, the maximum value of $\P$ will increase.
  To obtain a PEP that provides worst-case values independent of $\N$, the performance measure needs to be an expression that is invariant to the number of agents, $\frac{1}{\N}\sum_{i=1}^\N (f_i(x_i^0) - f_i(x^*))$ in our example. We propose a notion called scale-invariant, defined below.
\end{example*}} \smallskip

\begin{definition}[Scale-invariant expression and constraint] \label{def:scale-invariant}
  In a distributed system with $\N$ agents, we call a Gram-representable expression $h$ scale-invariant if it can be written as a linear combination $h = \sum_{j} c_j h_j,$
  where coefficients $c_j$ are independent of $\N$ and terms $h_j$ are any of these three forms, for any variables $x_i$ and $y_j$ assigned to agents $i$ and $j$, including the case $x_i = y_j$: 
  \begin{align}
    \text{(a)} \quad& \frac{1}{\N} \sum_{i=1}^\N f_i(x_i), && \text{the average of function values,} \\
    \text{(b)} \quad& \frac{1}{\N} \sum_{i=1}^\N x_i^Ty_i, && \text{the average of scalar products between two variables of the same agent,} \\
    \text{(c)} \quad & \frac{1}{\N^2} \sum_{i=1}^\N \sum_{j=1}^\N x_i^Ty_j, && \text{the average over the $\N^2$ pairs of agents of the scalar products}\\[-5mm]
    &&& \text{between two variables, each assigned to any agent.}
  \end{align}
  A Gram representable constraint $h \le D$, for $D \in \R$, is scale-invariant if the expression $h$ is scale-invariant.
\end{definition}
\begin{remark} The name \emph{scale-invariant} comes from the fact that, if we duplicate each agent of the system, with its local variables and its local function, then the value of a scale-invariant expression is unchanged. For example, the following expressions can be shown, possibly after algebraic manipulations, to be scale-invariant: 
  \begin{align}
  \frac{1}{\N} \sum_{i=1}^\N \qty(f_i(\xb) - f_i(x^*)), \qquad&
  \frac{1}{\N} \sum_{i=1}^\N \|x_i - x^*\|^2, \qquad
  \frac{1}{\N} \sum_{i=1}^\N \|x_i - \xb\|^2, &
  \left\|\frac{1}{\N} \sum_{i=1}^\N \nabla f_i(x_i) \right\|^2.
  \end{align}
  The variables common to all agents, such as $\xb$ or $x^*$, can indeed be used in scale-invariant expressions because they are assumed to be copied in each agent set of variables:
  \[ \xb_i = \xb = \frac{1}{\N} \sum_{j=1}^\N x_j \quad \text{for all $i\in\V$} \qquad \text{ and } \qquad x^*_i = x^* \quad \text{for all $i\in\V$}. \]
  These constraints can then be written with scale-invariant expressions. See Appendix \ref{ap:explicit} for details.
\end{remark} \smallskip

\new{
Let us consider another simple example to motivate the second type of expressions we define.
\begin{example*}
  If we consider a PEP for distributed optimization with the initial conditions $\| x_i^0 - x^* \|^2 \le \N \quad \text{for all $i\in \V$}$ and the performance measure $\P = \frac{1}{\N} \sum_{i=1}^\N \|x_i^0 - x^* \|^2$, 
  then the worst-case value would scale with the number of agents $\N$ in the problem: if we consider more agents, the maximum value of $\P$ will increase.
  To obtain a PEP that provides worst-case values independent of the number of agents $\N$, the initial constraints involving only one agent at a time need to be independent of $\N$, $\| x_i^0 - x^* \|^2 \le 1$ in this example. We propose the notion of single-agent constraint, defined below.
\end{example*}} \smallskip

\begin{definition}[Single-agent expression and constraint] \label{def:single_agent}
  In a distributed system with $\N$ agents, we call an expression $h$ single-agent if it can be written as a linear combination $h = \sum_{k} c_k h_k(i),$
  where coefficients $c_k$ are independent of $\N$ and all terms $h_k(i)$ only involve the local variables and the local function of a single agent $i$. A constraint $h \le D$, for $D \in \R$, is single-agent if the expression $h$ is single-agent.
\end{definition}
\begin{remark}
  When all the agents are equivalent in the PEP, they all share the same single-agent constraints. Here are examples of single-agent constraints 
\[ \|x_i - x^*\|^2 \le 1, \qquad \text{for all $i \in \V$,} \qquad \qquad \|\nabla f_i(x_i)\|^2 \le 1,\qquad \text{for all $i \in \V$.} \]
\end{remark}
We can now state Theorem \ref{thm:agent_indep_perf} which characterizes the settings where the worst-case value $w(\Setting_\N)$ can be computed (i) in a compact manner for all $\N \ge 2$ and (ii) independently of the value of $\N$. We remind that $w(\Setting_\N=\{\N,\Al,\K,\P,\F,\I,\W\})$ denotes the worst-case performance of the execution of $\K$ iterations of distributed algorithm $\Al$ with $\N$ agents, with respect to the performance criterion $\P$, valid for any starting point satisfying the set of initial conditions $\I$, any local functions in the set of functions $\F$ and any averaging matrix in the set of matrices $\W$. \vspace{2mm}
\begin{theorem}[Agents-independent worst-case performance] \label{thm:agent_indep_perf}
  Let $\N \ge 2$ and $\W = \Wcl{\lm}{\lp}$ with $\lm \le \lp \in (-1,1)$. We assume that the algorithm $\Al \in \Ad$, the performance measure $\P$, the initial conditions $\I$, and the interpolation constraints for $\F$ are linearly (or LMI) Gram-representable. If Assumption \ref{ass:equiva} holds, meaning that all the agents are equivalent in the PEP, then 
  \begin{enumerate}
    \item[1.] The computation of $w(\Setting_\N)$ can be formulated as a compact SDP PEP, with $\fa \in \Rvec{q}$, $\Ga \in \Rmat{p}{p}$ and $\Gr \in \Rmat{p}{p}$ as variables and whose size is independent of the number of agents $\N$.
    \item[2.] If, in addition, the performance criterion $\P$ is scale-invariant and the set of initial conditions $\I$ only contains single-agent constraints, applied to every agent $i\in\V$, and scale-invariant constraints, then the resulting SDP PEP problem and its worst-case value $w(\Setting_\N)$ are fully independent of the number of agents $\N$.
\end{enumerate}
\end{theorem}
 
\begin{corollary} \label{cor:N2}
  Under the conditions stated in part 2. of Theorem 
  \ref{thm:agent_indep_perf}, a general worst-case guarantee, valid for any $\N \ge 2$ (including $\N \to \infty$), can be obtained using $\N=2$:
  \[ w(\Setting_\N) = w(\Setting_2) \qquad \text{for all $\N \ge 2$}. \]
\end{corollary}
\begin{remark}
  The conditions of part 2 of Theorem 
  \ref{thm:agent_indep_perf}, which guarantee agent-independent worst-case performance, are satisfied for many usual settings, as illustrated in Appendix \ref{ap:explicit}.
\end{remark} \smallskip

To prove Theorem \ref{thm:agent_indep_perf}, we rely on Lemma \ref{lem:SDPGs} which gives a way of imposing the SDP condition $G^s \succeq 0$ solely based on its block matrices $\Ga$ and $\Gr$.
\begin{lemma} \label{lem:SDPGs}
Let $G^s \in \Rmat{\N p}{\N p}$ be a fully symmetrized Gram matrix, as defined in  \eqref{eq:Gs}, for $\N \ge 2$. The SDP constraint $G^s \succeq 0$ can be expressed with $\Ga \in \Rmat{p}{p}$ and $\Gr \in \Rmat{p}{p}$:
\begin{equation} \label{eq:Gsequiv}
    G^s \succeq 0 \qquad \Leftrightarrow \qquad \Ga + (\N-1)\Gr \succeq 0 \quad \text{and} \quad \Ga-\Gr \succeq 0.
\end{equation}
\end{lemma}
\begin{proof}
  We will use the fact that
  \[ G^s \succeq 0 \quad \Leftrightarrow \quad z^TG^sz \ge 0 \quad \text{for all $z \in \Rvec{\N p}$}.\]
  Let us separate the vector $z \in \Rvec{\N p}$ into $\N$ sub-vectors $z_i$ ($i=1,\dots,\N$) that can each be written as the sum of an average part $\zb \in \Rvec{p}$ and a centered one $\zc_i$: 
    \[ z_i = \zb + \zc_i, \qquad \text{where $\zb = \frac{1}{\N}\sum_{i=1}^\N z_i$ and $\sum_{i=1}^\N \zc_i = 0$.}\]
   Using this decomposition of $z$, together with the definition of $G^s$ \eqref{eq:Gs}, we can express $z^TG^sz$ as follows
 \begin{align*}
    \begin{bmatrix} \zb + \zc_1 \\ \vdots \\ \zb + \zc_\N \end{bmatrix}^T \begin{bmatrix}\Ga & \Gr & \dots \\ \Gr & \ddots & \\  \vdots& & \Ga \end{bmatrix} \begin{bmatrix} \zb + \zc_1 \\ \vdots \\ \zb + \zc_\N \end{bmatrix}
    &= \sum_{i=1}^\N\sum_{j=1}^\N (\zb + \zc_i)^T \Gr (\zb + \zc_j) + \sum_{i=1}^\N (\zb + \zc_i)^T (\Ga-\Gr) (\zb + \zc_i), \\
    \text{Since $\sum_{i=1}^\N \zc_i = 0$,} \hspace{3cm} & = \N^2 \zb^T (\Gr) \zb + \N \zb^T (\Ga - \Gr) \zb + \sum_{i=1}^\N (\zc_i)^T (\Ga-\Gr) (\zc_i), \\
    & = \zb^T \qty(\N \Ga + \N(\N-1) \Gr) \zb + \sum_{i=1}^\N (\zc_i)^T (\Ga-\Gr) (\zc_i).
\end{align*}
Therefore, $z^TG^sz \ge 0$ for all $z \in \Rvec{\N p}$ when 
\begin{equation} \label{eq:proof_Gs}
  \zb^T \qty(\N \Ga + \N(\N-1) \Gr) \zb + \sum_{i=1}^\N (\zc_i)^T (\Ga-\Gr) (\zc_i) \ge 0, \quad \text{for all $\zb, \zc_i \in \Rvec{p}$ such that $\sum_{i=1}^\N \zc_i = 0$.}
\end{equation}  
Conditions $\Ga+(\N-1)\Gr \succeq 0$ and $\Ga-\Gr \succeq 0$ are sufficient to ensure that \eqref{eq:proof_Gs} is satisfied and hence that $G^s \succeq 0$. To show the necessity of these conditions, let us consider different cases for $\zb$ and $\zc_i$:
\begin{itemize}
  \item When $\zc_i = 0$ for all $i$, then \eqref{eq:proof_Gs} simply imposes that $\zb^T \qty(\N \Ga + \N(\N-1) \Gr) \zb \ge 0$ for all $\zb \in \Rvec{p}$, which is equivalent to $\Ga+(\N-1)\Gr \succeq 0$.
  \item When $\zb = 0$, $\zc_1 = - \zc_2 = s$, for an arbitrary $s \in \Rvec{p}$, and $\zc_i = 0$ for $i=3,\dots,\N$, then \eqref{eq:proof_Gs} imposes that $s^T (\Ga-\Gr)s \ge 0$ for all $s \in \Rvec{p}$, which is equivalent to $\Ga-\Gr \succeq 0$.
\end{itemize}
\end{proof}

\begin{proof}[Proof of Theorem \ref{thm:agent_indep_perf} (part 1)]
  As detailed in Section \ref{sec:consensusPEP}, the set of averaging matrices $\Wcl{\lm}{\lp}$ has Gram-representable necessary interpolation constraints, see Corollary \ref{cor:conscons}. Since $\Al,\P,\I$ and interpolation constraints for $\F$ are also Gram-representable,  Proposition \ref{prop:GramPEP} guarantees that the computation of $w(\N,\Al,\K,\P,\I,\F, \Wcl{\lm}{\lp})$ can be formulated as an agent-dependent SDP PEP problem \eqref{eq:SDP_PEP}. Moreover, since the agents are equivalent, we can restrict to fully symmetric solutions without impacting the worst-case value (see Corollary \ref{cor:fullsymsol}).
  Hence, all the PEP components can be written in terms of $\f^s$ and $G^s$, which are composed of blocks $\fa$, $\Ga$, and $\Gr$. Lemma \ref{lem:SDPGs} shows how the SDP constraint $G^s \succeq 0$, can be expressed in terms of $\Ga$ and $\Gr$. The other elements of the PEP can directly be expressed in terms of $\fa$, $\Ga$, and $\Gr$, using the definition of $\f^s$ and $G^s$ \eqref{eq:Fs} \eqref{eq:Gs}, as detailed in Appendix \ref{ap:explicit}. 
  Since all the PEP elements can be expressed in terms of blocks $\fa\in \Rvec{q}$, $\Ga \in \Rmat{p}{p}$ and $\Gr \in \Rmat{p}{p}$, the problem can be made smaller by considering them as the variables of the SDP PEP, instead of the full matrix $G^s \in \Rmat{\N p}{\N p}$ and vector $\f^s \in \Rvec{\N q}$. The size of the problem is thus independent of the number of agents $\N$, but $\N$ could still appear as a coefficient in the constraints or the objective of the problem.
\end{proof}

Lemma \ref{lem:SDPGs} makes the value of $\N$ appear as a coefficient, but we can define a change of variables allowing to remove the dependency on $\N$ in the new constraints \eqref{eq:Gsequiv}. Let us defined the new matrix variables $\Gt \in \Rmat{p}{p}$ as follows:
\begin{equation} \label{eq:GT}
    \Gt = \frac{1}{\N}\qty(\Ga + (\N-1)\Gr). 
\end{equation}
Moreover, we can also express $\Gt$
directly based on the blocks of the initial Gram matrix $G$ \eqref{eq:initialGram}:
\begin{equation} \label{eq:def_GT}
  \Gt = \frac{1}{\N^2}\sum_{i=1}^\N\sum_{j=1}^\N G_{ij}.
\end{equation}
While $\Ga$ contains the average scalar products between local variables of the same agent, and $\Gr$ the average scalar products between local variables of different agents, this new block $\Gt$ contains the average scalar products between local variables of any two agents. We call $\Gt$ the collective block matrix.
By applying this change of variables \eqref{eq:GT} in Lemma \ref{lem:SDPGs}, we obtain conditions that are independent of $\N$:
\begin{lemma}[Constraint $G^s \succeq 0$] \label{lem:SDPGs2}
  Let $G^s \in \Rmat{\N p}{\N p}$ be a symmetric Gram matrix, as defined in  \eqref{eq:Gs}. The SDP constraint $G^s \succeq 0$ can be expressed with $\Ga \in \Rmat{p}{p}$ and $\Gt \in \Rmat{p}{p}$:
  \begin{equation} \label{eq:Gsequiv2}
      G^s \succeq 0 \qquad \Leftrightarrow \qquad \Gt \succeq 0 \quad \text{and} \quad \Ga-\Gt \succeq 0.
  \end{equation}
\end{lemma}
\begin{proof}
  The result can be obtained by applying change of variables \eqref{eq:GT} in the constraints stated in \eqref{eq:Gsequiv}.
\end{proof}

The new variable $\Gt$ allows writing many constraints without any dependence on $\N$. 
Let us state a lemma that will be used to prove the second part of Theorem \ref{thm:agent_indep_perf}.

\begin{lemma} \label{lem:scale-invariant}
  In an SDP PEP for distributed optimization \eqref{eq:SDP_PEP}, with equivalent agents (Assumption \ref{ass:equiva}), any scale-invariant Gram-representable expression can be written independently of $\N$, using variables $\fa$, $\Ga$, and $\Gt$. 
\end{lemma}
\begin{proof}
  By Corollary \ref{cor:fullsymsol}, since all the agents are equivalent, we can restrict the PEP to a fully symmetric solution $(\f^s, G^s)$ from \eqref{eq:Fs}-\eqref{eq:Gs}.
  Since the expression is scale-invariant and Gram-representable (Definition \ref{def:scale-invariant}), it can only combine three types of terms, with coefficients independent of $\N$,
  \begin{itemize}
    \item[(a)] The average of function values 
    \begin{equation} \label{eq:usefa}
      \frac{1}{\N} \sum_{i=1}^\N f_i(x_i) =  \frac{1}{\N} \sum_{i=1}^\N\id_{f(x)}^T f_i = \id_{f(x)}^T \fa, \qquad \text{for any $j=1,\dots,\N$,}
    \end{equation}
    where $\id_{f(x)}$ is the vector of coefficients such that $\id_{f(x)}^T f_i = f_i(x_i)$. The second equality holds because $f_i = \fa$ for all $i$, by definition of the agent-symmetric $\f^s$ \eqref{eq:Fs}.
    \item[(b)] The average of scalar products between two variables related to the same agent $i$:
    \begin{equation} \label{eq:useGa}
        \frac{1}{\N} \sum_{i=1}^\N x_i^Ty_i = \frac{1}{\N} \sum_{i=1}^\N \id_x^T G_{ii} \id_y = \id_x^T (\Ga) \id_y,
    \end{equation}
    where $\id_x$ (resp. $\id_y$) is the vector of coefficients such that $P_i \id_x = x_i$ (resp. $P_i \id_y = y_i$), with $P_i$ defined in \eqref{eq:defPi}. The second equality holds because $G_{ii} = \Ga$, by definition of the agent-symmetric $G^s$ \eqref{eq:Gs}.
    \item[(c)] The average over the $\N^2$ pairs of agents of the scalar products between two variables, each related to any agent:
    \begin{equation} \label{eq:useGt}
        \frac{1}{\N^2} \sum_{i=1}^\N \sum_{j=1}^\N x_i^Ty_j = \frac{1}{\N^2} \sum_{i=1}^\N \sum_{j=1}^\N \id_x^T G_{ij} \id_y = \id_x^T (\Gt) \id_y,
    \end{equation}
    where the last equality follows from the definition of $\Gt$ \eqref{eq:def_GT}.
  \end{itemize}
  We have shown that the three types of terms (a), (b), (c) can be written independently of $\N$ with $\fa$, $\Ga$, and $\Gt$, and so can any Gram-representable and scale-invariant expression.
\end{proof}
We are now ready to prove the second part of Theorem \ref{thm:agent_indep_perf}, about the worst-case value being independent of the number of agents.
\begin{proof}[Proof of Theorem \ref{thm:agent_indep_perf} (part 2)]
Under the conditions stated in the theorem (part 2.), we will show that the SDP-PEP problem can be written, independently of $\N$, with $\fa$, $\Ga$, and $\Gt$ as variables.
When all the agents are equivalent in the PEP, we can restrict to fully symmetric solutions (Corollary \ref{cor:fullsymsol}).
In that case, Lemma \ref{lem:SDPGs2} shows that SDP constraints $G^s \succeq 0$ can be expressed with $\Ga$ and $\Gt$, independently of $\N$.

Then, we prove that, in the symmetrized PEP, any single-agent constraint (Definition \ref{def:single_agent})
can be written with $\fa$ and $\Ga$, independently of $\N$. 
Since the PEP components are all Gram-representable, any of its single-agent constraints for agent $i$ linearly combines local function values $f_i(x_i)$ and scalar products of local variables $x_i^Ty_i$, with coefficients that are independent of $\N$. Here, $x_i$ and $y_i$ denote any local vector held by agent $i$. By definition of the symmetric $\f^s$ \eqref{eq:Fs}, a local function value term can be expressed with $\fa$ as
\begin{equation} \label{eq:usefa_single}
  f_i(x_i) = \id_{f(x)}^T f_i = \id_{f(x)}^T \fa, \qquad \text{for any $i=1,\dots,\N$.}
 \end{equation}
By definition of the symmetric Gram $G^s$ \eqref{eq:Gs}, a local scalar product term can be expressed with $\Ga$ as
 \begin{equation} \label{eq:useGa_single}
  x_i^Ty_i = \id_x^T G_{ii} \id_y  = \id_x^T \Ga \id_y \qquad \text{for any $i=1,\dots,\N$,}
\end{equation}
We have shown that the two types of terms \eqref{eq:usefa_single} and \eqref{eq:useGa_single} can be written independently of $\N$ with $\fa$, $\Ga$, and so can any Gram-representable single-agent constraint. The resulting constraint, expressed with $\fa$ and $\Ga$, is valid for any agent $i \in \V$. This is consistent with the equivalent agents assumption, which requires that each single-agent constraint is applied similarly to every agent $i\in\V$.
The reformulation of single-agent constraints, independently of $\N$, applies to the interpolation constraints for $\F$, the algorithm constraints for $\Al \in \Ad$ (except consensus), and part of the initial constraints $\I$.

We now show that the other PEP components, see e.g. \eqref{eq:PEP}, can also be expressed independently of $\N$. When agents are equivalent, by Lemma \ref{lem:scale-invariant}, any Gram-representable and scale-invariant expression can be expressed with $\fa$, $\Ga$, and $\Gt$, independently of $\N$. As explained in the proof of part 1 of the theorem, all the PEP components are Gram-representable. Therefore, we will simply prove that all the PEP components that cannot be expressed as single-agent constraints, treated above, can be written with scale-invariant expressions. 
The performance measure $\P$ and the rest of the initial conditions $\I$ are scale-invariant, by assumption of the theorem statement. The optimality condition for \eqref{opt:dec_prob}, can be written in a Gram-representable manner as
\[ \left\|\frac{1}{\N} \sum_{i=1}^\N \nabla f_i(x^*) \right\|^2 = 0, \]
which is well scale-invariant thanks to the $1/\N$ factor.
Finally, the interpolation constraints for consensus steps with averaging matrices from $\Wcl{\lm}{\lp}$ are given by \eqref{eq:avg_pres2}, \eqref{eq:SDP_cons2} and \eqref{eq:sym_cons2}, from Corollary \ref{cor:conscons}. In these constraints, notations hide sums over the agents and can be written with Gram-representable and scale-invariant expressions as 
\begin{align}
  \overline{X} = \overline{Y} \quad &\Leftrightarrow \quad \qty(\frac{1}{\N} \sum_{i=1}^\N (X_i-Y_i))^T\qty(\frac{1}{\N} \sum_{j=1}^\N (X_j-Y_j)) = 0, \\
  (\Yp-\lm\Xp)^T(\Yp-\lp\Xp) \preceq 0 \quad &\Leftrightarrow \quad   \frac{1}{\N}\sum_{i=1}^\N \qty((Y_i - \overline{Y})-\lm(X_i - \overline{X}))^T\qty((Y_i - \overline{Y})-\lp(X_i - \overline{X})) \preceq 0, \\ 
  \Xp^T \Yp - \Yp^T\Xp = 0  \quad &\Leftrightarrow \quad  \frac{1}{\N}\sum_{i=1}^\N (X_i - \overline{X})^T(Y_i - \overline{Y}) -  \frac{1}{\N}\sum_{i=1}^\N (Y_i - \overline{Y})^T(X_i - \overline{X}) = 0.
\end{align}
where matrices $X_i, Y_i \in \Rmat{d}{\K}$ contain the different consensus iterates $x_i^k, y_i^k \in \Rvec{d}$ ($k=1,\dots,\K$) of agent $i$ as columns. The two last constraints have been multiplied by $1/\N$ to obtain scale-invariant expressions. This does not alter these constraints that have left-hand sides equal to zero. 

In the end, all the PEP components can well be expressed independently of $\N$, with $\fa\in \Rvec{q}$, $\Ga\in \Rmat{p}{p}$ and $\Gt\in \Rmat{p}{p}$. We refer to Appendix \ref{ap:explicit} for explicit expressions of different possible PEP constraints and objectives. Therefore, the SDP PEP can be made smaller and fully independent of the number of agents $\N$, by considering these blocks as variables of the problem, instead of the full matrix $G^s \in \Rmat{\N p}{\N p}$ and vector $\f^s \in \Rvec{\N q}$.
\end{proof}

\subsection{Impact of agent symmetry on the worst-case solution} \label{sec:impactSym}
\new{By definition of the fully symmetric PEP solution \eqref{eq:Fs}-\eqref{eq:Gs}, in which the Gram matrix has equal diagonal blocks $\Ga$ and equal off-diagonal blocks $\Gr$, we can prove symmetry results for the worst-case iterates and functions of the agents. These symmetries open up perspectives with implications for understanding, analyzing and designing distributed algorithms yet to be discovered.}
\begin{proposition}
\label{prop:impact_sym} 
  If all the agents are equivalent in the agent-dependent PEP (Assumption \ref{ass:equiva}), then there is a worst-case instance such that,  
  \begin{itemize}
    \item the worst-case sequences of iterates of the agents $x_i^0, \dots, x_i^\K \in \Rvec{d}$ ($i=1,\dots,\N$) are unitary transformations, i.e. rotations or reflections, of each other:
    \[   x_i^k = R_i x_1^k \qquad \text{with $R_i \in \Rmat{d}{d}$ such that $R_i^TR_i = I$,} \hspace{2cm} \text{for all $i \in \V$ and all $k$}.  \]
    \item the worst-case local functions have equal values at their own iterates $f_i(x_i^k) = f_j(x_j^k)$ (for all $i,j \in \V$ and all $k$), and can be chosen identical up to a unitary transformation of variables:
   \[ f_i(x) = f_1(R_i^T x) \qquad \text{for all $i \in \V$.}\]
  \end{itemize}

\end{proposition}
\begin{remark}[]
  Proposition \ref{prop:impact_sym} is stated for the worst-case iterates, but is also valid for any linear combinations of columns of $P_i$ \eqref{eq:defPi}. Indeed, as detailed in the proof below, we have $P_i = R_i P_1$, with $R_i \in \Rmat{d}{d}$ such that $R_i^TR_i = I$, for all $i=1,\dots,\N$. 
  Moreover, the result is also valid for more than one agent at a time and can be written for any groups of agents, e.g. for groups of two agents, we have:
  \[ \qty[P_i ~ P_j] = \tilde{R}_{ij} \qty[P_1 ~ P_2] \qquad \text{with $\tilde{R}_{ij} \in \Rmat{d}{d}$ such that $\tilde{R}_{ij}^T\tilde{R}_{ij} = I$,} \hspace{1.5cm} \text{for all $i,j \in \V$ and all $k$}.\]
  This holds by definition of the fully symmetrized Gram matrix $G^s$ \eqref{eq:Gs} which has equal diagonal blocks: $G_{ii} = \Ga$ for all $i\in\V$, but also equal off-diagonal blocks: $G_{ij} = \Gr$ for $i\ne j \in \V$. 
\end{remark}
\begin{proof}
  By Corollary \ref{cor:fullsymsol}, when the agents are all equivalent, there is a worst-case solution of the agent-dependent PEP, which is fully symmetric: $(\f,G) = (\f^s, G^s)$.  When the solution is agent-symmetric, by definition of $G^s$ \eqref{eq:Gs}, we have
  \begin{equation} \label{eq:ag_sym}
     G_{ii} = P_i^TP_i = P_j^TP_j = G_{jj} \quad \text{for all $i,j \in \V$,} 
  \end{equation}
  which imposes that matrices of agent variables $P_i$ are all unitary transformations of each other. Without loss of generality, we can express all the $P_i$ depending on $P_1$:
    \begin{equation} \label{eq:iso1}
      P_i = R_i P_1 \quad \text{ with } R_i^TR_i = I, \qquad \text{for all $i \in \V$.}
    \end{equation} 
  Matrices $R_i$ are not especially unique. This unitary transformation relation is valid for any column or combination of columns of $P_i$ and $P_1$, so we have it for iterates and gradient vectors of the agents:
  \begin{align} \label{eq:iso2}
    x_i^k &= R_i x_1^k & g_i^k &= R_i g_1^k \quad \text{ with } R_i^TR_i = I, \qquad \text{for all $i \in \V$ and all $k$.}
  \end{align}
  Concerning the local functions, by definition of $\f^s$ \eqref{eq:Fs}, we have, for all $i,j \in \V$,
  \begin{equation} \label{eq:ag_sym2}
    f_i = f_j \quad \text{where $f_i \in \Rvec{q}$ is a vector containing the function values related to agent $i$.}
  \end{equation}
  In particular, we have well that $f_i(x_i^k) = f_j(x_j^k)$, for all $i,j \in \V$ and all $k$.
  Let $f_1$ be a function interpolating the set of triplets $\{x_1^k, g_1^k, f_1^k\}_{k=0,\dots,\K}$. 
  We can choose $f_i$, the local function of agent $i$, as follows:
  \[ f_i(x) = f_1(R_i^T x),\]
  Indeed, $f_i$ interpolates well the set of triplets $\{x_i^k, g_i^k, f_i^k\}_{k=0,\dots,\K}$
  since \eqref{eq:iso2} and \eqref{eq:ag_sym2} guarantee that $R_i^T x_i^k = x_1^k$, $R_i^T g_i^k = g_1^k$ and $f_i^k = f_1^k$, which are well interpolated by $f_1$ by definition.
\end{proof}

\section{Compact PEP with multiple equivalence classes of agents}
\label{sec:subsetsAgPEP}

While all agents are equivalent (Assumption \ref{ass:equiva}) for many common performance evaluation settings, there are advanced settings for which there are several equivalence classes of agents, see Definition \ref{def:equicl}. For example, when different groups of agents use different (uncoordinated) step-sizes, function classes, initial conditions, or even different algorithms. It can also happen that the performance measure focuses on a specific group of agents, e.g. the performance of the worst agent. The ability to efficiently and accurately evaluate the performance of distributed optimization algorithms in such advanced settings would enable a more comprehensive analysis and deeper understanding of the algorithms performance.

In this section, we exploit symmetries in different equivalence classes of agents in PEP, revealed in Proposition \ref{prop:sym_sol_2}, to write agent-dependent PEPs in a compact form whose size only depends on the number of equivalence classes $\U$, and not on the total number of agents $\N$.
By Proposition \ref{prop:sym_sol_2}, we can solve the agent-dependent PEP problem, described in Section \ref{sec:agentPEP}, restricted to symmetric agent-class solutions $(\Fss,\Gss)$, without impacting the worst-case value. These symmetric solutions $(\Fss,\Gss)$, defined in \eqref{eq:Fsym} and \eqref{eq:Gsym}, depend on the partition $\T$ of the set of agents $\V$ into equivalence classes: $\T = \{ \V_1,\dots,\V_\U \}$, see Definition \ref{def:equicl}. 
A symmetrized solution $(\Fss,\Gss)$ has equal blocks for equivalent agents. In this section, we order the agents by equivalence classes in the symmetrized solution $(\Fss,\Gss)$ \eqref{eq:Fsym}-\eqref{eq:Gsym} \new{so that we can partition it into different blocks, associated with the different agent classes. For example, if we have two equivalence classes ($\U=2$): one with two agents and one with one agent, the symmetrized solution can be written as follows: \vspace{-1mm}
\begin{align}
  \Fss &= \begin{bmatrix} \fa^1 & \fa^1 & \fa^2 \makebox(-33,0){\rule[1ex]{0.4pt}{\normalbaselineskip}}\end{bmatrix}^T &\quad &\text{ and }& \quad 
  \Gss &= \left[ \begin{array}{cc|c}
                          \Ga^1 & \Gr^1 & \Gc^{12} \\
                          \Gr^1 & \Ga^1 & \Gc^{12} \\[1mm] 
                          \hline && \\[-3mm]
                          \Gc^{21} & \Gc^{21} & \Ga^{2}
          \end{array} \right],
  \end{align}
  where $\fa^1$, $\fa^2$, $\Ga^1$, $\Gr^1$, $\Ga^{2}$, $\Gc^{12}$ and $\Gc^{21}$ are blocks defined in Definiton \ref{def:blocks}. We define new notations to aggregate the blocks of each equivalence class: \vspace{-1mm}
  \begin{align}
  \Fss &= \begin{bmatrix} \f_{\V_1} & \f_{\V_2} \end{bmatrix}^T &\quad &\text{ and }& \quad 
  \Gss &= \begin{bmatrix}
                          G_{\V_1}^s & G_{\V_1\V_2}^s \\
                          G_{\V_2\V_1}^s & G_{\V_2}^s
                        \end{bmatrix},
\end{align} 
In general, when ordering the agents by equivalence classes,} the symmetrized solution $(\Fss,\Gss)$ \eqref{eq:Fsym}-\eqref{eq:Gsym} can be written as
\begin{align}
  \Fss &= \begin{bmatrix} \f_{\V_1} &\dots& \f_{\V_\U} \end{bmatrix}^T && \hspace{-0.2cm}
  \text{where }  \f_{\V_u} = \mathbf{1}_{\N_u} \kron  \fa^u \label{eq:Fsu} \\
  \Gss &= \begin{bmatrix}
                          G_{\V_1}^s  & G_{\V_1\V_2}^s & \dots & G_{\V_1\V_\U}^s \\
                          G_{\V_2\V_1}^s & G_{\V_2}^s  &  & \vdots \\
                          \vdots   &          & \ddots&  \\
                          G_{\V_\U\V_1}^s  & \dots & & G_{\V_\U}^s
                        \end{bmatrix}
& & \hspace{-0.2cm} \parbox[c]{10cm}{
\text{where }
~ \text{$G_{\V_u}^s = \begin{cases} 
  \Ga^u & \text{if $\N_u = 1$,} \\
  \mathbf{1}_{\N_u} \mathbf{1}_{\N_u}^T \kron \Gr^u + (I_{\N_u} \kron (\Ga^u - \Gr^u)) & \text{if $\N_u \ge 2$},   
\end{cases}$} \\[4mm] 
\hspace{-0.3cm} \text{and} \quad \text{$G_{\V_u\V_v}^s  = \mathbf{1}_{\N_u} \mathbf{1}^T_{\N_v} \kron \Gc^{uv}$},
} \hspace{2mm} \label{eq:Gsu}
\end{align}
and where $\fa^u \in \Rvec{q}$, $\Ga^u, \Gr^u, \Gc^{uv} \in \Rmat{p}{p}$ for $u,v=1,\dots,\U$ are the blocks repeating in $\Fss$ and $\Gss$, see Definition \ref{def:blocks}. 

\begin{remark}[Limit cases for $\U$] ~
  \begin{itemize} 
    \item $\U=\N$:
    When no agents are equivalent, the partition $\T$ contains $\U=\N$ equivalence classes of size 1 and then the symmetrized agent-class solution $\Fss$, $\Gss$, defined in Proposition \ref{prop:sym_sol_2} is identical to the agent-dependent PEP solution $\f$, $G$:
    \[ \Fss = F, \qquad \text{ and } \qquad \Gss = G. \]
    Therefore, the resulting PEP formulation is equivalent to the agent-dependent PEP from Section \ref{sec:agentPEP} and is not compact as its size scales with $\U=\N$.
    \item $\U=1$:  
    When all the agents are equivalent, the partition $\T$ contains $\U=1$ equivalence class of size $\N$ and then the symmetrized agent-class solution $\Fss$, $\Gss$ defined in Proposition \ref{prop:sym_sol_2} is identical to the fully symmetrized solution $\f^s$, $G^s$ defined in Corollary \ref{cor:fullsymsol}
    \[ \Fss = \f^s, \qquad \text{ and } \qquad \Gss = G^s. \]
    Therefore, the resulting compact PEP formulation is identical to the one presented in Section \ref{sec:symPEP}.
  \end{itemize}
\end{remark}

\subsection{The PEP formulation restricted to symmetric agent-class solutions} \label{sec:PEP_restrict_setsym}
Based on the results from Proposition \ref{prop:sym_sol_2}, the definitions of the symmetric PEP solutions $\Fss$ \eqref{eq:Fsu} and $\Gss$ \eqref{eq:Gsu}, 
we now show when and how an agent-dependent PEP can be expressed in a compact form, using only the blocks $\fa^u$, $\Ga^u$, $\Gr^u$, and $\Gc^{uv}$ ($u,v=1,\dots,\U$) as variables. This will enable efficient computing of $w(\Setting_\N)$ for all $\N\ge2$, even when all the agents are not equivalent. We remind that $w(\Setting_\N=\{\N,\Al,\K,\P,\F_u,\I,\W\})$ denotes the worst-case performance of the execution of $\K$ iterations of distributed algorithm $\Al$ with $\N$ agents, with respect to the performance criterion $\P$, valid for any starting point satisfying the set of initial conditions $\I$, any local functions being each in a given set of function $f_i \in \F_{u_i}$ and any averaging matrix in the set of matrices $\W$. \vspace{2mm}

\begin{theorem}[Agent-class worst-case performance] \label{thm:sets_perf}
    Let $\V$ be a set of $\N \ge 2$ agents, $\T$ be a partition of $\V$ into $\U \le n$ equivalence classes of agents, as defined in Definition \ref{def:equiva}, and $\W = \Wcl{\lm}{\lp}$ with $\lm \le \lp \in (-1,1)$. \\
    If the algorithm $\Al \in \Ad$, the performance measure $\P$, the initial conditions $\I$, and the interpolation constraints for each $\F_u$ ($u=1,\dots,\U$) are linearly (or LMI) Gram representable, then the computation of $w(\Setting_\N)$ can be formulated as a compact SDP PEP,
  with $\fa^u$, $\Ga^u$, $\Gr^u$, and $\Gc^{uv}$ ($u,v=1\dots,\U$) as variables and, hence, whose size only depends on the number of agent equivalence classes $\U$ but not directly on the total number of agents $\N$.
\end{theorem}
\begin{remark}
  The blocks $\Gr^u$ contain scalar products between variables of different, but equivalent, agents. Therefore, $\Gr^u$ does not exist for equivalence classes with only one agent $\N_u = 1$.
  The blocks $\Gc^{uv}$ contain scalar products between variables of different agents that are not equivalent. Therefore, these blocks do not exist when $\U = 1$.
  The other blocks $\fa^u$ and $\Ga^u$, related to only one agent, always exist.
\end{remark}
To prove this theorem, we rely on the following lemma which reformulates the SDP condition $G^s \succeq 0$ with the blocks $\Ga^u$, $\Gr^u$, and $\Gc^{uv}$, and generalizes Lemma \ref{lem:SDPGs}.
\begin{lemma} \label{lem:SDPGssets}
Let $\Gss \in \Rmat{\N p}{\N p}$ be a symmetrized agent-class Gram matrix, as defined in \eqref{eq:Gsu}. The SDP constraint $\Gss \succeq 0$ can be expressed with $\Ga^u \in \Rmat{p}{p}$, $\Gr^u  \in \Rmat{p}{p}$ and $\Gc^{uv}  \in \Rmat{p}{p}$, $u,v=1\dots,\U$:
\begin{equation} \label{eq:SDP_Gsym}
    \Gss \succeq 0 \qquad \Leftrightarrow \qquad \begin{cases}
      \Ga^u \succeq 0 & \text{for all $u$ with $\N_u=1$,} \\
      \Ga^u-\Gr^u \succeq 0 & \text{for all $u$ with $\N_u\ge 2$,}
    \end{cases} 
    \qquad \text{ and } \qquad \Ht \succeq 0,
\end{equation}
where $\Ht \in \Rmat{\U p}{\U p}$ is a symmetric matrix composed of $\U^2$ blocks of dimension $p \times p$:
\begin{equation}
  \Ht^{uv} = \begin{cases}
     \N_u\Ga^u + \N_u(\N_u-1)\Gr^u, &\text{when $u=v$,} \\
     \N_u\N_v \Gc^{uv}, &\text{when $u\ne v$}
\end{cases}. \qquad u,v = 1,\dots,\U. \label{eq:def_H}
\end{equation}
\end{lemma}
\begin{proof}
  We have that
  \[ \Gss \succeq 0 \quad \Leftrightarrow z^T\Gss z \ge 0 \quad \text{for all \new{vector} $z \in \Rvec{\N p}.$}\]
  A vector $z \in \Rvec{\N p}$ can be divided into $\U$ sub-vectors of length $\N_u p$ to match with the blocks of matrix $\Gss$ (see Proposition \ref{prop:sym_sol_2}). We also decompose each sub-vector $z_u$ in two terms: the average of $z_u$ over the agents from the class $\V_u$, and the remaining part of $\zc_u$:
  \[z_u = \mathbf{1}_{\N_u} \kron \zb_u + \zc_u \quad \text{ with } \quad  \zb_u = \frac{1}{\N_u}(\mathbf{1}_{\N_u} \kron I_p)^T z_u, \qquad \zb_u \in \Rvec{p},~ \zc_u \in \Rvec{\N_up} \text{ and }  z_u \in \Rvec{\N_up}. \]
  By definition, we have that $\frac{1}{\N_u}(\mathbf{1}_{\N_u} \kron I_p)^T \zc_u = 0$ for all $j=1,\dots,\U$. We now compute $z^T\Gss z$ \new{(for any vector $z \in \Rvec{np}$)} by summing over each block:
  \begin{equation}
     z^T\Gss z = \sum_{u=1}^\U z_u^T G_{\V_u}^s z_u + \sum_{u=1}^\U \sum_{v\ne u} z_u^T G_{\V_u \V_v}^s z_v, \label{eq:block_sum}
\end{equation}
  where $G_{\V_u}^s$ and $G_{\V_u \V_v}^s$ are defined in \eqref{eq:Gsu}. In the definition, we have two cases for $G_{\V_u}^s$, depending on the value of $\N_u$. In the rest of the proof, we assume $\N_u \ge 2$ and only work with $G_{\V_u}^s = \mathbf{1}_{\N_u} \mathbf{1}_{\N_u}^T \kron \Gr^u + (I_{\N_u} \kron (\Ga^u - \Gr^u))$. The special case for $\N_u=1$, $G_{\V_u}^s = \Ga^u$, can be adapted to this case $\N_u \ge 2$, by defining $\Gr^u = 0$ when $\N_u=1$.
  Here is the development for each term of \eqref{eq:block_sum}:
  \begin{align}
    z_u^T G_{\V_u}^s z_u &= \qty(\mathbf{1}_{\N_u} \kron \zb_u + \zc_u)^T \qty(\mathbf{1}_{\N_u} \mathbf{1}_{\N_u}^T \kron \Gr^u + (I_{\N_u} \kron (\Ga^u - \Gr^u))) \qty(\mathbf{1}_{\N_u} \kron \zb_u + \zc_u ) \\
    &= \zb_u^T \N_u^2 \Gr^u \zb_u +  \zb_u^T \N_u(\Ga^u - \Gr^u) \zb_u + (\zc_u)^T \qty(I_{\N_u} \kron (\Ga^u-\Gr^u)) \zc_u,\\
    &= \zb_u^T (\N_u \Ga^u + \N_u(\N_u-1) \Gr^u) \zb_u + (\zc_u)^T \qty(I_{\N_u} \kron (\Ga^u-\Gr^u)) \zc_u, \\[2mm]
    z_u^T G_{\V_u \V_v}^s z_v &= \qty(\mathbf{1}_{\N_u} \kron \zb_u + \zc_u)^T \qty(\mathbf{1}_{\N_u} \mathbf{1}_{\N_v}^T \kron \Gc^{uv} ) \qty(\mathbf{1}_{\N_v} \kron \zb_v + \zc_v ) \\
    &=  \zb_u \N_u\N_v \Gc^{uv} \zb_v.
  \end{align}
  Using these expressions, we can write \eqref{eq:block_sum} as a sum of quadratic forms
  \[ z^T\Gss z = \zb^T \Ht \zb + \sum_{u=1}^\U (\zc_u)^T \qty(I_{\N_u} \kron (\Ga^u-\Gr^u)) \zc_u,\]
  where $\zb \in \Rvec{\U p}$ is a vector stacking the average vectors $\zb_u$ of each equivalence class ($u=1,\dots,\U$), and $\Ht \in \Rmat{\U p}{\U p}$ is the matrix defined in \eqref{eq:def_H}. Therefore, the constraint $\Gss \succeq 0$ can be expressed as
  \begin{equation}
    \zb^T \Ht \zb + \sum_{u=1}^\U (\zc_u)^T \qty(I_{\N_u} \kron (\Ga^u-\Gr^u)) \zc_u \ge 0, \quad \text{\small{for all $\zb \in \Rvec{\U p}$ and all $\zc_u \in \Rvec{\N_up}$ s.t. $(\mathbf{1}_{\N_u} \kron I_p)^T \zc_u = 0$}}  \label{eq:proof_Gsym}
  \end{equation}
  Conditions $\Ht \succeq 0$ and $\Ga^u-\Gr^u \succeq 0$ (for all $u=1,\dots,\U$) are sufficient to ensure that \eqref{eq:proof_Gsym} is satisfied and hence that $G^s \succeq 0$. To show the necessity of these conditions, let us consider different cases for $\zb$ and $\zc_u$:
\begin{itemize}
  \item When $\zc_u = 0$ for all $u$, then \eqref{eq:proof_Gsym} simply imposes that $\zb^T \Ht \zb \ge 0$ for all $\zb \in \Rvec{\U p}$, i.e. $\Ht \succeq 0$.
  \item When $\zb = 0$, $\zc_u = 0$, for all $u$, except for an arbitrary $u=v$, then \eqref{eq:proof_Gsym} imposes \vspace{-1mm}
  \[\sum_{i=1}^{\N_v} (\zc_v(i))^T(\Ga^v-\Gr^v)\zc_v(i) \ge 0, \quad \text{for all $\zc_v \in \Rvec{p}$ such that $(\mathbf{1}_{\N_v} \kron I_p)^T \zc_v = 0$,} \vspace{-1mm} \]
  where $\zc_v(i) \in \Rvec{p}$, for $i=1,\dots,\N_v$, denotes the $i^\mathrm{th}$ part of vector $\zc_v \in \Rvec{n_v p}$. Since $\N_v \ge 2$, we can choose $\zc_v(1) = -\zc_v(2) = s$, for an arbitrary $s \in \Rvec{p}$, and $\zc_v(j) = 0$ for $j\ge 3$, so that the above equation becomes $s^T (\Ga^v-\Gr^v) s \ge 0$ for all $s \in \Rvec{p}$, which is equivalent to $\Ga^v-\Gr^v \succeq 0$. This is valid for an arbitrary $v=1,\dots,\U$. 
\end{itemize}
Therefore, by setting $\Gr^u=0$ when $\N_u=1$, we have well proved the equivalence \eqref{eq:SDP_Gsym} for the SDP condition $\Gss \succeq 0$.
\end{proof}

\begin{proof}[Proof of Theorem \ref{thm:sets_perf}]
  As detailed in Section \ref{sec:consensusPEP}, the set of averaging matrices $\Wcl{\lm}{\lp}$ has Gram-representable necessary interpolation constraints, see Corollary \ref{cor:conscons}. Since $\Al,\P,\I$ and interpolation constraints for each $\F_u$ are also Gram-representable,  Proposition \ref{prop:GramPEP} guarantees that the computation of $w(\N,\Al,\K,\P,\I,\F_u, \Wcl{\lm}{\lp})$ can be formulated as an agent-dependent SDP PEP problem \eqref{eq:SDP_PEP}. Moreover, by Proposition \ref{prop:sym_sol_2}, we can restrict the PEP to symmetric agent-class solutions of the form of $\Fss$ \eqref{eq:Fsym} and $\Gss$ \eqref{eq:Gsym}. The symmetric solutions $\Fss$ and $\Gss$ are composed of blocks $\fa^u$, $\Ga^u$, $\Gr^u$, and $\Gc^{uv}$ ($u,v=1\dots,\U$). Therefore, all the PEP elements can be expressed in terms of these smaller blocks.
  Lemma \ref{lem:SDPGssets} shows how the SDP constraint $\Gss \succeq 0$, can be expressed in terms of $\Ga^u$, $\Gr^u$, and $\Gc^{uv}$. The other elements of the PEP can directly be expressed in terms of $\fa^u$, $\Ga^u$, $\Gr^u$, and $\Gc^{uv}$ ($u,v=1\dots,\U$), using the definition of $\Fss$ \eqref{eq:Fsym} and $\Gss$ \eqref{eq:Gs}, as detailed in Appendix \ref{ap:explicit}. 
  Since all the PEP elements can be expressed in terms of blocks $\fa^u \in \Rvec{q}$, $\Ga^u \in \Rmat{p}{p}$, $\Gr^u  \in \Rmat{p}{p}$ (when $\N_u \ge 2$) and $\Gc^{uv}  \in \Rmat{p}{p}$ ($u,v=1\dots,\U$), the problem can be made smaller by considering them as the variables of the SDP PEP,
  instead of the full matrix $\Gss \in \Rmat{\N p}{\N p}$ and vector $\Fss \in \Rvec{\N q}$. The size of the problem depends on the number $\U$ of equivalence classes of agents, and not directly on the total number of agents.
\end{proof}
Theorem \ref{thm:sets_perf} shows that the worst-case performance of a distributed optimization method can be computed with an SDP PEP whose size only depends on the number $\U$ of equivalence classes of agents, but not directly on the total number of agents $\N$. Therefore, if the number of equivalence classes $\U$ is independent of $\N$, so is the size of the problem. However, the problem can depend on $\N$ and $\N_u$ ($u=1,\dots,\U$) as parameters. Appendix \ref{ap:explicit} provides details to explicitly construct this compact PEP formulation.

\subsubsection*{Infinitely many agents}
If the number $\U$ of equivalence classes of agents is independent of $\N$, so is the size of the compact SDP PEP, and we can solve it efficiently for any value of $\N$. In particular, we can take the limit when $\N$ tend to $\infty$ in the problem, to compute the worst-case performance of an algorithm on an infinitely large network of agents 
\[ \lim_{\N\to\infty} w\qty(\N,\Al,\K,\P,\I,\F_u,\Wcl{\lm}{\lp}),\]
and to determine if this performance is bounded or not. 
In all the cases we treated (see Section \ref{sec:showEXTRA}), we observed that if the worst-case value for $\N\to\infty$ is bounded, then it depends on the size proportions of each equivalence class $\rho_1,\dots,\rho_\U$:
\[ \rho_u = \lim_{\N\to\infty}\frac{\N_u}{\N}.\]

\section{Case study: the EXTRA algorithm} \label{sec:showEXTRA}
In this section, we leverage our new PEP formulation, based on equivalence classes of agents, to evaluate the performance of the EXTRA algorithm \cite{EXTRA} in advanced settings, and the performance evolution with system size $\N$. This provides totally new insights into the worst-case performance of the algorithm.
The EXTRA method, described in Algorithm \ref{algo:EXTRA}, is a well-known distributed optimization method that was introduced in \cite{EXTRA} and further developed in  \cite{shi2015proximal,jakovetic2018unification,li2020revisiting}.
\begin{algorithm}[h]
  \caption{EXTRA}
  \begin{algorithmic}
  \STATE Choose step-size $\alpha > 0$, matrices $W, ~\Wt \in \Rmat{\N}{\N}$ and pick any $x_i^0 \in \Rvec{d}$ ($\forall i$);
  \STATE  $x_i^{1} = \sum_{j} w_{ij} x_j^0 - \alpha \nabla f_i(x_i^0)$, ~$\forall i$;
  \FOR{$k = 0, 1,\dots$}
    \STATE  $x_i^{k+2} = x_i^{k+1} + \sum_{j} \qty(w_{ij} x_j^{k+1} - \tilde{w}_{ij} x_j^k) - \alpha \qty(\nabla f_i(x_i^{k+1}) - \nabla f_i(x_i^k))$, ~$\forall i$;
 \vspace{1mm}
  \ENDFOR
  \end{algorithmic}
  \label{algo:EXTRA}
\end{algorithm}

\noindent Authors in  \cite{EXTRA} recommend to choose $\Wt = \frac{W+I}{2}$. This allows computing only one new consensus step at each iteration. In that case, the sharpest convergence results for EXTRA are given in \cite{li2020revisiting}, and summarized below for the strongly-convex case. 

\begin{proposition}[Theoretical performance guarantee for EXTRA \cite{li2020revisiting}] \label{prop:EXTRA} We consider the decentralized optimization problem \eqref{opt:dec_prob} with optimal solution $\x^* \in \Rvec{d}$. Under the following assumptions,
  \begin{enumerate}
    \item All local functions $f_i$ ($i \in \V$) are $L$-smooth and $\mu$-strongly convex, with $0 \le \mu \le L$;
    \item The averaging matrix $W \in \Wcl{-\lam}{\lam}$, with $\lam \in [0,1)$, and $\Wt = \frac{W+I}{2}$;
    \item The starting points satisfy
            \[\|x_i^0 - x^*\|^2 \le R_1 \qquad \text{and} \qquad \| \nabla f_i(x^*) \|^2 \le R_2 \qquad \text{for all $i=1,\dots,\N$;}\]
  \end{enumerate}
the EXTRA algorithm with $\alpha = \frac{1}{4L}$ guarantees that
\begin{align}
E_f(\K) \coloneq f(\xb^\K) - f(x^*) \le (1-\tau)^\K \frac{L R_1 + R_2/L}{1-\lam}, && \text{(functions error)} \label{eq:Ef_th} \\
E_{x}(\K) \coloneq \frac{1}{\N} \sum_{i=1}^\N \|x_i^\K - x^*\|^2 \le (1-\tau)^\K \frac{R_1 + R_2/L^2}{1-\lam}, && \text{(iterates error)} \label{eq:Ex_th} 
\end{align}
where $\xb^\K = \frac{1}{\N} \sum_{i=1}^\N x_i^\K$ and $\tau = \frac{1}{39\qty(\frac{L}{\mu} + \frac{1}{1-\lam})}$.
\end{proposition}
\begin{proof} Bounds \eqref{eq:Ef_th} and \eqref{eq:Ex_th} directly follows from \cite[Corollary 2]{li2020revisiting}. 
\end{proof}
For comparison purposes, our PEP-based analysis of EXTRA will use the same assumptions as in Proposition \ref{prop:EXTRA}, with $L=1$, $\mu=0.1$, $\lam = 0.5$, $R_1 = 1$ and $R_2 = 1$, unless specified otherwise. 
Our PEP framework with equivalence classes of agents also enables the analysis of a variety of settings other than those of Proposition \ref{prop:EXTRA}, many of which have not yet been studied in the literature. In this section, we explore, in particular, the performance of the worst agent, the k-th percentile performance, and the performance under agent heterogeneity in the classes of local functions. Other possible settings include agent heterogeneity in the initial conditions, the network topology, the algorithm parameters or algorithm execution. For example, it could be used to analyze the performance of an algorithm when there is a subset of malicious agents.

\subsection{Performance of the worst agent} \label{sec:worst_agent}
\begin{figure}[b]
  \centering
  \includegraphics[width=0.6\textwidth]{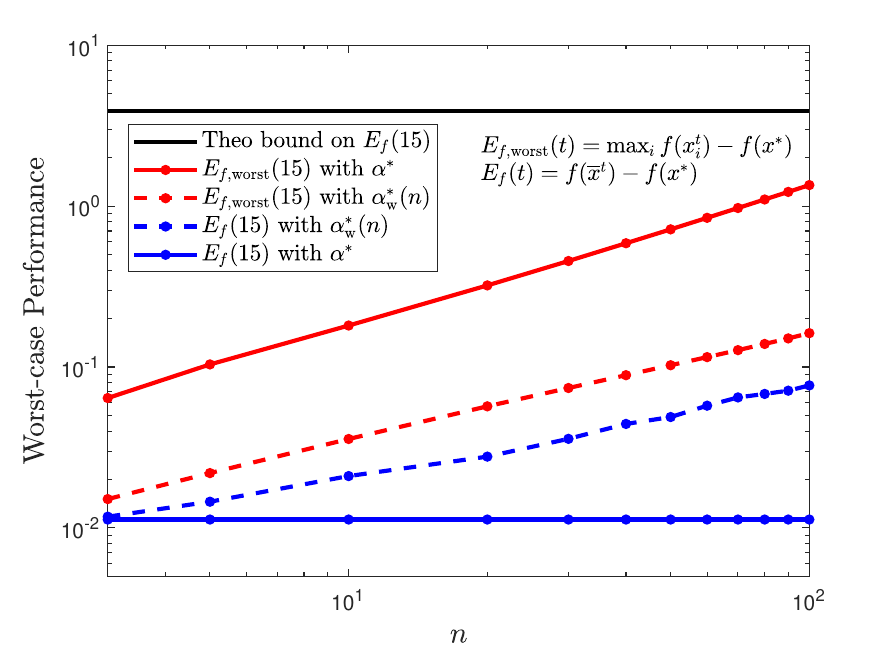}
  \caption{Comparison of the average functional error $E_f$ \eqref{eq:perf_avg} and the worst functional error $E_{f,\mathrm{worst}}$ \eqref{eq:perf_worst} for $\K=15$ iterations of the EXTRA algorithm and their evolution with the number of agents $\N$ in the system. The plot shows (i) the constant theoretical bound on $E_f$ \eqref{eq:Ef_th}, in black, (ii) the PEP bounds on the function error of the worst agent $E_{f,\mathrm{worst}}$, scaling sublinearly with $\N$, in red and
  (iii) the PEP bounds for the function error of the average iterate $E_{f}$, in blue. The constant step-size $\alpha^*$ has been optimized for this performance criterion and leads to a bound constant with $\N$ (plain blue line). Step-sizes can also be optimized with respect to $E_{f,\mathrm{worst}}(15)$, leading to smaller step-sizes $\alpha^*_{\mathrm{w}}(\N)$ decreasing as $\frac{1}{\sqrt{\N}}$. They improve the guarantee for $E_{f,\mathrm{worst}}$ and its scaling with $\N$ but they deteriorate the corresponding guarantee for $E_{f}$ (see dashed lines).
  In this plot, the range of eigenvalues for the averaging matrix is $[-0.5,0.5]$ and the local functions are 1-smooth and 0.1-strongly convex.
  }
  \label{fig:EXTRA_ww_Nevol}
\end{figure}
Classical analysis of distributed optimization algorithms often focuses on performance metrics averaged over all the agents, namely
\begin{equation} \label{eq:perf_avg}
  E_f(\K) = f(\xb^\K) - f(x^*) \qquad \text{ or } \qquad   E_x(\K) = \frac{1}{\N} \sum_{i=1}^\N \|x_i^\K - x^* \|^2,
  \end{equation}
where $\xb^\K = \frac{1}{\N} \sum_{i=1}^\N x_i^\K$. Theoretical guarantees on these criteria for EXTRA are given in \eqref{eq:Ef_th} and \eqref{eq:Ex_th} from Proposition \ref{prop:EXTRA}. As we have seen in Section \ref{sec:symPEP}, these error criteria lead to performance bounds that are independent of the number of agents $\N$ and are therefore easier to treat analytically. However, there exist other relevant performance measures that are more difficult to analyze and for which there are few or no results. This is the case of the performance of the worst agent in the network,
\begin{equation} \label{eq:perf_worst}
  E_{f,\mathrm{worst}}(\K) = \max_{i \in \V } f(x_i^\K) - f(x^*)\qquad \text{ or } \qquad  E_{x,\mathrm{worst}}(\K) = \max_{i \in \V } \|x_i^\K - x^*\|^2,
\end{equation}
where $\V = \{1,\dots,\N\}$ is the set of agents in the network.
To analyze these performance measures using PEP, we need to fix (arbitrarily) the agent that will be the worst one and use it in the PEP objective:
\[ f(x_1^\K) - f(x^*) \qquad \text{ or } \qquad \|x_1^\K - x^*\|^2, \]
The maximization of such objectives ensures that agent 1 will be the one with the largest error (on $f$ or $x$). Indeed, if another agent $j$ has a larger error, then the permutation of variables of agents $1$ and $j$ in the PEP solution would lead to another feasible PEP solution with a larger objective.
Agent 1 will thus receive a specific role in the PEP solution and is not equivalent to the $\N-1$ others. We can formulate a compact PEP using two equivalence classes of agents
\[ \V_1 = \{1\} \qquad \text{ and } \qquad \V_2 = \{2,\dots,\N\},\]
and applying the theory developed in Section \ref{sec:subsetsAgPEP}. \vspace{-1mm}

\paragraph*{The worst agent functions error}
Figure \ref{fig:EXTRA_ww_Nevol} shows the evolution of $E_f$ and $E_{f,\mathrm{worst}}$ with the size of the network, for 15 iterations of the EXTRA algorithm. These results have been validated with the agent-dependent PEP formulation (from Section \ref{sec:agentPEP}) for small values of $\N$, for verification purposes. However, as $\N$ becomes larger, this agent-dependent formulation becomes computationally intractable, reaching an SDP size of $325\times 325$ for $\N=10$ (and $\K=15$). This limitation precisely motivated the development of the compact symmetrized PEP formulations from Sections \ref{sec:symPEP} and \ref{sec:subsetsAgPEP}. 

Firstly, Figure \ref{fig:EXTRA_ww_Nevol} shows that the PEP bound for the average performance $E_f$ is well constant over $\N$, as predicted by Theorem \ref{thm:agent_indep_perf}, and outperforms the theoretical bound from \cite[Corollary 2]{li2020revisiting}, which requires 2750 iterations to guarantee the same error as the PEP bounds after 15 iterations. The theoretical bound is valid for a step-size $\alpha = \frac{1}{4L}$, while the PEP bound is shown for an optimized step-size $\alpha^* = \frac{0.78}{L}$, which minimizes the bound.
Since the bound is constant over $n$, the value of the optimized step-size can be used for any system size, however, it may depend on other settings, such as the range of eigenvalues for $W$: $[\lm, \lp]$. 

There is currently no theoretical guarantee for the worst agent functions error $E_{f,\mathrm{worst}}$ of EXTRA in the literature, but we can analyze it with our compact PEP formulation. Figure \ref{fig:EXTRA_ww_Nevol} shows that the PEP bound for the worst agent $E_{f,\mathrm{worst}}$ (in red) is scaling sublinearly with the number of agents $\N$. Using the constant step-size $\alpha^* = \frac{0.78}{L}$ (plain red line), we obtain a logarithmic slope of $0.92$. The step-size can also be tuned with respect to $E_{f,\mathrm{worst}}$, and would then have a different value for each value of $\N$, denoted $\alpha^*_\mathrm{w}(\N)$. Using $\alpha^*_\mathrm{w}(\N)$ (dashed red line), we obtain a logarithmic slope of $0.67$, which means the scaling with $\N$ is less important when using suitable step-sizes. In our case, the step-sizes $\alpha^*_\mathrm{w}(\N)$ are diminishing in $\frac{1}{\sqrt{\N}}$. However, these smaller step-sizes deteriorate the average performance $E_{f}$ (dashed blue line).

\paragraph*{The worst agent iterates error}
There is currently no theoretical guarantee for the worst agent iterates error $E_{x,\mathrm{worst}}$ of EXTRA in the literature, but one could derive\footnote{using the fact that $\max_{i\in\V} \|x_i^\K - x^*\|^2 \le \sum_{i=1}^\N \|x_i^\K - x^*\|^2$} a conservative bound on $E_{x,\mathrm{worst}}$ from \eqref{eq:Ex_th}, scaling linearly with the number of agents $\N$.
Using our compact PEP formulation, we can analyze such performance metric accurately and show that the performance of the worst agent $E_{x,\mathrm{worst}}$ for EXTRA actually scales sublinearly with $\N$, as shown in Figure \ref{fig:EXTRA_wb_Nevol}. 
Figure \ref{fig:EXTRA_wb_Nevol} shows, in particular, PEP-based bounds for $E_x$ and $E_{x,\mathrm{worst}}$. The observations are similar to those of Figure \ref{fig:EXTRA_ww_Nevol}. The PEP bound for the average performance $E_x$ is constant over $\N$, as predicted by Theorem \ref{thm:agent_indep_perf} and the corresponding optimized step-size is $\alpha^* = \frac{0.78}{L}$. Moreover, the PEP bound for the worst agent $E_{x,\mathrm{worst}}$ is scaling sublinearly with the number of agents $\N$.
Using the constant step-size $\alpha^* = \frac{0.78}{L}$ (plain red line), we obtain a logarithmic slope of $0.82$. As previously, the step-size can also be tuned with respect to $E_{x,\mathrm{worst}}$, and would then have a different value for each value of $\N$, denoted $\alpha^*_\mathrm{w}(\N)$. Using the optimal diminishing step-sizes $\alpha^*_\mathrm{w}(\N)$ (dashed red line), we obtain a logarithmic slope of $0.60$, which means the scaling with $\N$ is less important when using suitable step-sizes. These smaller step-sizes only slightly deteriorate the average performance $E_{x}$ (dashed blue line).

\subsection{The 80-th percentile of the agent performance}
\begin{figure}[b]
  \centering
  \includegraphics[width=0.6\textwidth]{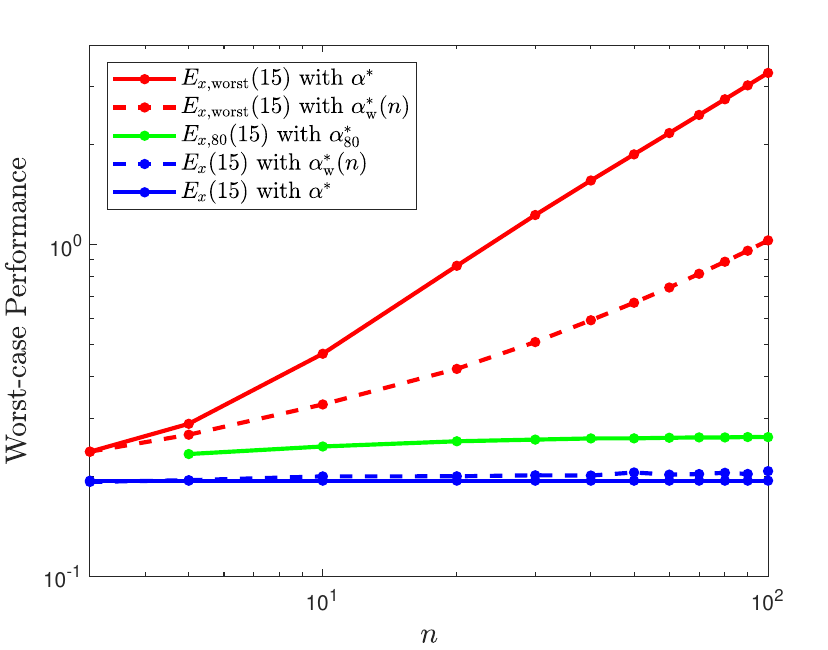}
  \caption{Comparison of the average iterates error $E_x$ \eqref{eq:perf_avg}, the worst iterates error $E_{x,\mathrm{worst}}$ \eqref{eq:perf_worst} and
  the 80-th percentile iterates error $E_{x,80}$ \eqref{eq:PEP_80perc} for $\K=15$ iterations of the EXTRA algorithm and their evolution with the number of agents $\N$ in the system. The plot shows (i) the PEP bounds on the iterates error of the worst agent $E_{x,\mathrm{worst}}$, scaling sublinearly with $\N$, in red,
  (ii) the PEP bound on the 80-th percentile iterates error $E_{x,80}$, in green, and (iii) the PEP bounds for the average iterates error $E_{x}$, in blue. The constant step-size $\alpha^*$ has been hand-tuned for this performance criterion and leads to a bound constant with $\N$ (plain blue line). Step-sizes can also be hand-tuned with respect to $E_{f\mathrm{worst}}(15)$, leading to smaller step-sizes $\alpha^*_{\mathrm{w}}(\N)$. They improve the guarantee for $E_{f,\mathrm{worst}}$ and its scaling with $\N$ without deteriorating too much the corresponding guarantee for $E_{f}$ (see dashed lines).
  In this plot, the range of eigenvalues for the averaging matrix is $[-0.5,0.5]$ and the local functions are 1-smooth and 0.1-strongly convex.}
  \label{fig:EXTRA_wb_Nevol}
\end{figure}

In some cases, the performance of the worst agent may not be the most appropriate performance measure. When the network is very large, the performance of the worst agent can be terribly bad, while the average performance is good.
Using our PEP approach, we can consider other performance measures that were so far completely out of reach of theoretical analysis. Inspired by statistical approaches, we propose here to analyze the k-th percentile in the distribution of the individual performance of each agent. In this experiment, we have chosen the 80-th percentile, that is, the performance of the worst agent, after excluding the 20\% worst agents. In other words, it measures the error at or below which 80\% of the agents fall in the worst-case scenario. The performance of the worst agent, analyzed in the previous subsection, can be seen as the 100-th percentile of the agent performance and can be analyzed via PEP using two equivalence classes of agents. When considering smaller percentiles, such as the 80-th, we can formulate a compact PEP using three equivalence classes of agents:
\[\V_1 = \{ 1, \dots, 0.2\N \}, \qquad \V_2 = \{0.2\N+1\}, \qquad \V_3 = \{0.2\N+2,\dots, \N\},  \]
where we assume $\N$ can be divided by 5. The class $\V_1$ contains the 20\% worst agents to exclude, the class $\V_2$ contains the agent for which the individual performance will provide the 80-th percentile, and $\V_3$ contains all the other agents. The PEP should therefore maximize the performance of agent $i\in \V_2$ and impose that the performance of each agent $j\in\V_j$ is larger:
\begin{align}\label{eq:PEP_80perc}
  E_{x,80}(\K) = \underset{\{\fa^u, \Ga^u, \Gr^u \Gc^{uv}\}_{u,v=1,2,3}}{\mathrm{maximum}} & ~\| x_i^\K - x^* \|^2  \quad \text{ with $i \in \V_2$}\\[1mm]
 \text{ s.t.} \qquad
   \| x_j^\K - x^* \|^2 &\ge  \| x_i^\K - x^* \|^2 \qquad \text{for all $j \in \V_1$,} \\
   ~~\qquad \text{All the other} & \text{ PEP constraints}.
\end{align}
In problem \eqref{eq:PEP_80perc}, the squared norms $~\| x_i^\K - x^* \|^2$ can be written as $(\id_{x^\K} - \id_{x^*})^T \Ga^u (\id_{x^\K} - \id_{x^*})$ with $u$ chosen according to the class $\V_u$ containing agent $i$. Details about the explicit construction of compact PEP formulations are given in Appendix \ref{ap:explicit}.

Figure \ref{fig:EXTRA_wb_Nevol} shows the evolution of $E_{x,80}$ with the size of the network, for 15 iterations of the EXTRA algorithm (green line). 
We observe that the scaling of $E_{x,80}$ with the number of agents $\N$ is very limited and quickly reaches a plateau, making this performance almost as good as the average performance $E_{x}$. Moreover, the optimal step-size for this bound, denoted $\alpha^*_{80}$, is independent of $\N$ and, in this case, it stays close to $\alpha^*$, optimized for $E_x$. 
The value of the plateau for $E_{x,80}$ has been confirmed by solving the compact PEP with $\N$ tending to $\infty$, as described at the end of Section \ref{sec:subsetsAgPEP}.

Figure \ref{fig:EXTRA_wb_Ninf} shows the evolution with k of the k-th percentile of the agent performance for EXTRA, when $\N \to \infty$. We see that the limit of the k-th percentile performance is finite for all values of $\mathrm{k} \in (0,100)$ and has a vertical asymptote in $\mathrm{k}=100$. The 100-th percentile computes the performance of the worst agent and grows well to infinity with $\N$ (see Figure \ref{fig:EXTRA_wb_Nevol}). 
According to Figure \ref{fig:EXTRA_wb_Ninf}, the value of the k-th percentile performance first increases slowly with k,
and then it starts to blow up after some threshold on k, depending on $\K$, to match the vertical asymptote in $\mathrm{k}=100$.
This means that it seems useful to exclude a small part of the worst agents in the performance metric, e.g. 10--20\%  in the settings we are considering, but not more.
Moreover, the k-th percentile of EXTRA for $\N \to \infty$ seems to converge geometrically. Indeed, in Figure \ref{fig:EXTRA_wb_Ninf}, when considering 5 more iterations in total, the k-th percentile improves by a factor between 2.2 and 3.2, depending on k. \smallskip

\begin{remark} Figures \ref{fig:EXTRA_wb_Nevol} and \ref{fig:EXTRA_wb_Ninf} show the performance of EXTRA for rather well-connected averaging matrices ($\lam=0.5$). Similar results and observations can be obtained for larger ranges of eigenvalues, i.e. $\lam \in (0.5,1)$, but this would require more iterations to get small errors.
\end{remark}

\begin{figure}[htb] 
  \centering
  \includegraphics[width=0.6\textwidth]{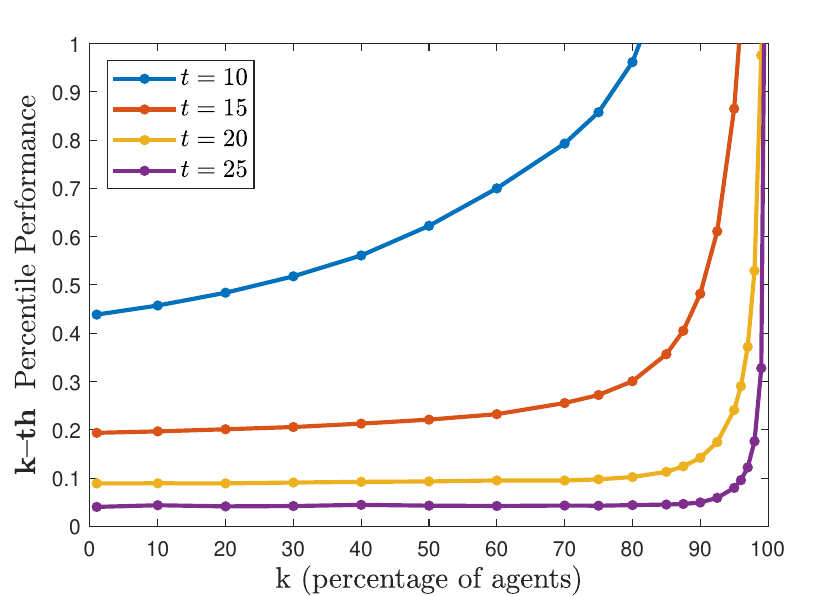}
  \caption{Evolution with k of the k-th percentile of the agent performance for $\K$ iterations of EXTRA, when the number of agents tends to infinity ($\N \to \infty$). Different total numbers of iterations $\K$ are compared. In this plot, the step-size is $\alpha = 0.78$, the range of eigenvalues for the averaging matrix is $[-0.5,0.5]$ and the local functions are 1-smooth and 0.1-strongly convex.
  }
  \label{fig:EXTRA_wb_Ninf}
\end{figure}

\subsection{Performance under local functions heterogeneity}
In many applications, the agents hold local functions sharing general properties, e.g. convexity or smoothness, but it is unlikely that all the local functions have exactly the same characterization for these properties. For example, the local functions can be $\mu_i$-strongly convex and $L_i$-smooth, but with different parameters $\mu_i$ and $L_i$. In general, the performance is then computed by considering that all the local functions have the worst parameters:
\[\mu = \min_i \mu_i \qquad \text{ and } \qquad L = \max_i L_i.\]
In this experiment, we analyze the gain in the performance guarantee when, instead of considering $f_i \in \F_{\mu,L}$ for all $i$, we consider two equivalence classes of agents $\V_1$ and $\V_2$ such that
\[f_i \in \F_{\mu_1,L} \quad \text{for all $i\in \V_1$} \quad \text{and} \quad f_i \in \F_{\mu_2,L} \quad \text{for all $i\in \V_2$}.\]
These two classes of agents manipulate functions with two different condition numbers $\kappa_1 = \frac{L}{\mu_1}$ and $\kappa_2 = \frac{L}{\mu_2}$.
\begin{figure}[tb] 
  \centering
  \includegraphics[width=0.8\textwidth]{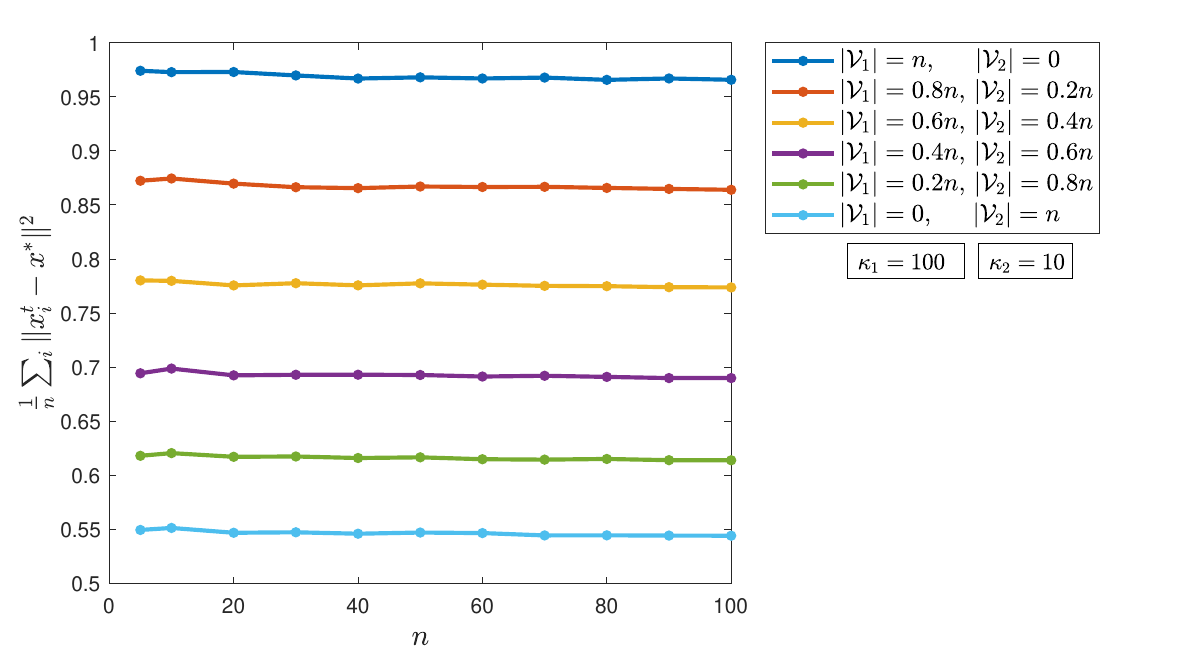}
  \caption{Evolution with the number of agents $\N$ of the worst-case average distance to optimum $E_{x}$ for 15 iterations of EXTRA with two equivalence classes of agents $\V_1$ and $\V_2$ using each a different condition number for their function class. The agents in $\V_1$ hold 1-smooth and 0.01-strongly convex local functions ($\kappa_1 = 100$) and the agents in $\V_2$ hold 1-smooth and 0.1-strongly convex local functions ($\kappa_2 = 10$). The guarantees are independent of the total number of agents $\N$ and only depend on the proportion of agents in each class. In this experiment, the range of eigenvalues for the averaging matrix is $[-0.5,0.5]$.}
  \label{fig:EXTRA_fctParams_Nevol}
\end{figure}
Figure \ref{fig:EXTRA_fctParams_Nevol} shows the situation where $L=1$, $\mu_1 = 0.01$ and $\mu_2=0.1$, for different sizes of classes $\V_1$ and $\V_2$. We observe that the guarantees for EXTRA only depend on the relative composition of each class $\V_1$ and $\V_2$ but not on the total number of agents $\N$. This observation has been confirmed by solving the compact PEPs for $\N$ tending to $\infty$, as described at the end of Section \ref{sec:subsetsAgPEP}.
In Figure \ref{fig:EXTRA_fctParams_Nevol}, when $20\%$ of the agents pass from the class $\V_1$ to the class $\V_2$, the worst-case guarantee is improved by a constant factor $0.9$. This improving factor can be linked to the worst-case guarantees obtained with a uniform value for $\kappa$ in the local functions ($\kappa=100$ or $\kappa=10$ for all the agents) and motivates the following conjecture:
\begin{conjecture}
Let $E_{x}^{\kappa_1,\kappa_2}(\K,\theta)$ be the performance guarantee after $\K$ iterations of EXTRA with $\theta \N$ agents with local functions in  $\F_{\mu_1,L}$ and $(1-\theta)\N$ agents with local functions in  $\F_{\mu_2,L}$, for $\theta \in [0,1]$.
Let $E_{x}^{\kappa_1}(\K)$ and $E_{x}^{\kappa_2}(\K)$ be the worst-case errors of $\K$ iterations of the EXTRA algorithm with uniform conditioning for all the local functions, respectively with $\kappa_1=\frac{L}{\mu_1}$ and $\kappa_2=\frac{L}{\mu_2}$. By definition, we have
\[E_{x}^{\kappa_1}(\K)=E_{x}^{\kappa_1,\kappa_2}(\K,1) \quad \text{ and } E_{x}^{\kappa_2}(\K)=E_{x}^{\kappa_1,\kappa_2}(\K,0)\]
We conjecture that the performance $E_{x}^{\kappa_1,\kappa_2}(\K,\theta)$ can be expressed in function of $E_{x}^{\kappa_1}(\K)$ and $E_{x}^{\kappa_2}(\K)$ as
  \[ E_{x}^{\kappa_1,\kappa_2}(\K,\theta) = \bigl(E_{x}^{\kappa_1}(\K)\bigr)^{\theta} \bigl(E_{x}^{\kappa_2}(\K)\bigr)^{1-\theta} \qquad \text{for all $\K>0$, $\kappa_1,\kappa_2>0$ and $\theta\in[0,1]$ }. \]
\end{conjecture}
In addition to the values of $\kappa_1$, $\kappa_2$ and $\theta$ shown in Figure \ref{fig:EXTRA_fctParams_Nevol}, this conjecture has been verified for different other values of $\kappa_1$, $\kappa_2$ and $\theta$. It can also be extended to any number of agent equivalence classes, associated with different values of $\kappa$. For example, we have verified it with 3 classes of agents. 
This results in much tighter bounds in cases where only a few agents in the system hold ill-conditioned local functions. As future direction, this also allows exploring situations where the agents hold totally different types of local functions, e.g. strongly-convex ($\mu > 0$) and weakly-convex ($\mu < 0$) functions, or smooth and non-smooth functions. 
This conjecture would have been difficult to guess without the PEP tool and the compact formulation developed in this paper, which shows its usefulness.

\subsection{On the numerical resolution of the compact SDP PEP formulation}
In this case study, we used Matlab to model and solve the SDP PEP problems. In particular, we used the YALMIP toolbox for modeling and the Mosek solver for solving. Our code is available on GitHub (\url{https://github.com/sebcolla/Performance-Estimation-Problems-for-distributed-optimization}). The Mosek solver, like most of the other existing ones, relies on the interior-point method and the existence of a Slater point in the feasible set of the problem. Some equality constraints in our problems may prevent the existence of a Slater point, hence leading the solver to find solutions of low numerical quality.
Indeed, the compact PEP formulation, described in Sections \ref{sec:symPEP} and \ref{sec:subsetsAgPEP}, and detailed in Appendix \ref{ap:explicit}, contains equality constraints that impose rank deflection of the matrix variables of the resulting SDP, see for example, the optimality constraints, or the average preserving condition for consensus steps (Propositions \ref{prop:opti_cons} and \ref{prop:consensus_interp_Gs} in Appendix \ref{ap:explicit}). This may be problematic for the numerical conditioning of the problem, and thus the numerical quality of its solution. 
We have verified the accuracy of the worst-case value given by the compact PEP formulation for $\N \le 5$, and always observed a relative error of less than $1\%$. However, the worst-case solutions (i.e. functions and iterates) can be very high-dimensional.

In the agent-dependent PEP formulations, all the problematic equality constraints can easily be removed, by exploiting them to reduce the size of the SDP. This allows the problem and its solution to be well-conditioned. However, such dimensionality reductions are more involved in our new compact PEP SDP formulations. We hope that existing toolboxes and solvers will be extended to automatically reduce the SDP size by exploiting equality constraints.
Therefore, for now, if one wishes to evaluate the performance of a distributed optimization algorithm in settings for which we know the result will be independent of $\N$ by Theorem \ref{thm:agent_indep_perf}, we recommend solving the agent-dependent formulation with $\N=2$ because the current solvers will perform better on this version of the problem. Moreover, formulating such an agent-dependent PEP is more intuitive and is made user-friendly by the toolboxes helping in building and solving PEP, i.e. PESTO in Matlab \cite{PESTO} and PEPit in Python \cite{pepit2022}. By Corollary \ref{cor:N2}, the obtained results for $\N=2$ agents will be valid for any $\N \ge 2$, including for $\N \to \infty$. 

\section{Conclusion}
In this paper, we have harnessed agent symmetries to develop compact SDP PEP formulations for computing the worst-case performance of decentralized optimization algorithms. 
When all the agents are equivalent, we have characterized the situations where the performance is totally independent of the number of agents, which allows analyzing and computing the performance in the fundamental case with only two agents. Such situations include many common settings for the performance evaluation of distributed optimization algorithms. We have also shown that in the worst-case scenario, when all agents are equivalent, their local sequences of iterates are rotations of each other and their local functions are identical up to a change of variables.
Moreover, this new compact PEP formulation also allows analyzing advanced performance settings with several equivalence classes of agents, enabling tighter analysis and deeper understanding of the algorithm performance, as we demonstrated with EXTRA. 
We believe that the different contributions of this paper can significantly help the research in the field of distributed optimization by simplifying the analysis and understanding of the performance of distributed algorithms. 

\bibliographystyle{plain+eid}
 
  
\bibliography{refs}

\newpage
\appendices 
\section{Explicit expressions of the compact PEP components} \label{ap:explicit}
\subsection{When all agents are equivalent} \label{ap:explicit_equiv}
Theorem \ref{thm:agent_indep_perf} characterizes the settings in which we can formulate the performance computation of a distributed optimization algorithm as an SDP PEP which is totally independent of the number of agents $\N$. In particular, we need all the agents to be equivalent (Assumption \ref{ass:equiva}), so we can work with a fully symmetric solution $(\f^s, G^s)$ (Corollary \ref{cor:fullsymsol}), composed of small repeating blocks $\fa$, $\Ga$, and $\Gr$.
The resulting agent-dependent PEP, expressed with blocks $\fa$, $\Ga$, and $\Gr$ has a size independent of the number of agents $\N$, but the value of $\N$ can still appear in the PEP and impact its solution, as any other parameter of the problem. Using appropriate changes of variables, we can obtain a PEP formulation that is fully independent of $\N$, under the condition stated in part 2 of Theorem \ref{thm:agent_indep_perf}. In this appendix, we show how we can explicitly build this agent-independent PEP formulation for many usual settings. For that purpose, we pass through the main components of a PEP, that we have introduced in Section \ref{sec:agentPEP}, and detail how they can be written with the blocks $\fa$, $\Ga$, and $\Gr$ and which changes of variable get rid of $\N$.

Firstly, the SDP condition $G^s \succeq 0$ can be expressed with $\Ga$, and $\Gr$ using Lemma \ref{lem:SDPGs}. Then, Lemma \ref{lem:SDPGs2} makes it independent of $\N$, by introducing a new block $\Gt = \frac{1}{\N}\qty(\Ga+(\N-1)\Gr)$. Similarly to this result, the agent-independent formulation of other PEP components relies on the blocks $\Ga$ \eqref{eq:Gs} and $\Gt$ \eqref{eq:GT} of the Gram matrices. We link these two matrices by defining \emph{the difference block} $\Gd$: 
\begin{equation} \label{eq:Gadt}
  \Gd = \Ga-\Gt.
\end{equation}
We summarize below how these three blocks, related to the symmetrized Gram $G^s$ by definition, can be expressed based on the blocks of any Gram matrix solution $G$ \eqref{eq:initialGram}:
\begin{equation} \label{eq:def_GT_GD}
  \Ga = \frac{1}{\N} \sum_{i=1}^\N G_{ii}, \qquad
  \Gt = \frac{1}{\N^2}\sum_{i=1}^\N\sum_{j=1}^\N G_{ij} \quad \text {and } \quad
  \Gd = \frac{1}{\N^2}\sum_{i=1}^\N\sum_{j=1}^\N (G_{ii} - G_{ij}).
  \end{equation}
Using these three blocks, we can easily write many expressions in PEP, independently of $\N$. However, only two of them should be variables of the resulting compact SDP PEP, as the third one could always be obtained using \eqref{eq:Gadt}. In this appendix, we call  agent-independently Gram-representable the expressions that can be written $\fa$, $\Ga$, $\Gt$, and $\Gd$, without any dependence on the number of agents $\N$.

\begin{definition}[Agent-independently Gram-representable] \label{def:agent-Gram-rep}
  Let $\fa$ be the block of the symmetrized function values vector, as defined in \eqref{eq:Fs}, $\Ga$, $\Gt$, and $\Gd$ the symmetric blocks defined in \eqref{eq:def_GT_GD}.
  We say that an expression, such as a constraint or an objective, is agent-independently Gram-representable if it can be expressed using a finite set of linear or LMI constraints or expressions involving only $\fa$, $\Ga$, $\Gt$, and $\Gd$, without any dependence on the number of agents $\N$.
\end{definition}

The formulation of single-agent constraints (Definition \ref{def:single_agent}) into agent-independently Gram-representable constraints will rely on relations \eqref{eq:usefa_single} and \eqref{eq:useGa_single}, involving $\fa$ and $\Ga$.
 The formulation of scale-invariant expressions (Definition \ref{def:scale-invariant}) into agent-independently Gram-representable expressions will rely on relations \eqref{eq:usefa}, \eqref{eq:useGa}, \eqref{eq:useGt}, involving respectively $\fa$, $\Ga$, and $\Gt$. We also define a useful relation involving $\Gd$, which encapsulates a combination of relations \eqref{eq:useGa} and \eqref{eq:useGt}.
This new relation allows expressing the average over the $\N$ agents of the scalar products of two centered variables related to the same agents: 
    \begin{equation} \label{eq:useGd}
      \frac{1}{\N} \sum_{i=1}^\N \qty(x_i - \xb)^T \qty(y_i - \yb) = \frac{1}{\N^2} \sum_{i=1}^\N \sum_{j=1}^\N (x_i^Ty_i  - x_i^Ty_j) = \frac{1}{\N^2} \sum_{i=1}^\N \sum_{j=1}^\N \id_x^T (G_{ii} - G_{ij}) \id_y = \id_x^T (\Gd) \id_y,
    \end{equation}
where $x_i$ and $y_j$ are any variables of agent $i$ and $j$ and $\xb = \frac{1}{\N} \sum_{j=1}^\N x_i$ and $\yb = \frac{1}{\N} \sum_{j=1}^\N y_i$ are their agent averages. The last equality in \eqref{eq:useGd} follows from the definition of $\Gd$ \eqref{eq:def_GT_GD}. Moreover, $\id_x, \id_y \in \Rvec{p}$ denote coefficient vectors for variables $x$ and $y$. As a reminder, a coefficient vector $\id_x$ contains linear coefficients selecting the correct combination of columns in $P_i \in \Rmat{d}{p}$ \eqref{eq:defPi} to obtain vector $x_i \in \Rvec{d}$, i.e. $P_i \id_x = x_i$, for any $i = 1,\dots,\N$. This notation allows, for example, to write $x_i^Tx_i$ as $x_i^Tx_i = \id_x^T P_i^TP_i \id_x = \id_x^T G_{ii} \id_x$.

In what follows, we will exploit all these relations, to write all the components of a PEP, independently of $\N$.

\noindent The following proposition treats the interpolation constraints of a PEP.
\begin{proposition}[Function interpolation constraints] \label{prop:fct_interp}
  Let $\F$ be a set of functions for which there exist linearly Gram-representable interpolation constraints. In a PEP for a distributed optimization method, with fully symmetrized solutions $(\f^s,G^s)$ \eqref{eq:Fs}-\eqref{eq:Gs}, the constraints $f_i \in \F$, for all $i \in \V$ can be expressed using a set of interpolation constraints that are agent-independently Gram-representable and that only involves $\fa$ and $\Ga$.  
\end{proposition}
\begin{proof}
  Each constraint $f_i \in \F$ can be written with a set of Gram-representable interpolation constraints involving function values of agent $i$ and scalar products between quantities (i.e. points and gradients) related to agent $i$. These are thus single-agent constraints that are all identical when we restrict to agent-symmetric PEP solution $(\f^s,G^s)$ \eqref{eq:Fs}-\eqref{eq:Gs}. Therefore, they can all be expressed with $\fa$ and $\Ga$ using \eqref{eq:usefa_single} and \eqref{eq:useGa_single}.
\end{proof}
\begin{remark}
  For example, when $\F$ is the set of convex functions, we have, 
  \begin{equation} \label{eq:ex_convex}
    f_i \in \F \quad \Leftrightarrow \quad f_i^k \ge f_i^l + (g_i^l)^T(x_i^k - x_i^l) \quad \forall k,l \quad \Leftrightarrow \quad \id_{f^k}^T \fa \ge \id_{f^l}^T \fa + \id_{g^l}^T (\Ga)(\id_{x^k} - \id_{x^l}),
  \end{equation}
  where $f_i^k = f_i(x_i^k)$, $g_i^l = \nabla f_i(x_i^l)$ and $\id_{f^k}$, $\id_{f^l}$, $\id_{g^l}$, $\id_{x^k}$ and $\id_{x^l}$ are appropriate vectors of coefficients.
  This is obtained by applying relation \eqref{eq:usefa_single} for $f_i(x_i^k)$ and $f_i(x_i^l)$ and relation \eqref{eq:useGa_single} for the scalar product between $g_i^l$ and $(x_i^k - x_i^l)$.
\end{remark}

\paragraph*{Algorithm description}
The class of methods $\Ad$ that we consider is defined in Definition \ref{def:Ad} and may combine gradient evaluation, consensus steps, and linear combinations of variables.
In the latter, each agent defines the same linear combination of its local variables. It should be the same combination for all the agents, otherwise, they are not equivalent, and Assumption \ref{ass:equiva} does not hold. All the variables in the combinations are related to the same agent and can therefore be squared and formulated in PEP using $\Ga$, as in \eqref{eq:useGa}. Alternatively, we can reduce the number of variables in the blocks of the Gram matrix by using the linear combination to define a vector of coefficients in terms of the others, as shown in Proposition \ref{prop:combi}:

\begin{proposition}[Linear combinations] \label{prop:combi} Let $P_i = [y_i^1,\dots,y_i^p]$ be the $p$ variables of agent $i$. We consider any linear combination of the variables, coordinated over all the agents $i\in\V$
\begin{equation} \label{eq:combi}
   z_i = \sum_{k=1}^p \beta_k y_i^k \qquad \text{for all $i$},
\end{equation}
  where $\beta_k$ are given coefficients and $y_i^k \in \Rvec{d}$ are any local variables from agent $i$. Such a combination step can be formulated in the PEP by defining the vector of coefficient $\id_z$ based on the vectors $\id_{y^1},\dots,\id_{y^p}$.
  \[ \id_z = \sum_{k=1}^p \beta_k \id_{y^k} \]
  This avoids adding $z_i$ as a new column of $P_i$.
\end{proposition}
\begin{proof} By definition, we have $z_i = P_i \id_z$ and $y_i^k = P_i \id_{y^k}$. Therefore, the linear equality \eqref{eq:combi} can be written as
  \[ P_i \qty(\id_z - \sum_{k=1}^p \beta_k \id_{y^k}) = 0, \]
which allows defining $\id_z$ based on the other vectors of coefficients $\id_z = \sum_{k=1}^p \beta_k \id_{y^k}$, without adding a new column in $P_i$.
\end{proof}

Concerning the consensus steps, Theorem \ref{thm:agent_indep_perf} applies to the set of averaging matrices $\Wcl{\lm}{\lp}$, defined in Section \ref{sec:agentPEP}. Necessary interpolation constraints for this set of matrices are given in \eqref{eq:avg_pres2}, \eqref{eq:SDP_cons2} and \eqref{eq:sym_cons2}, from Corollary \ref{cor:conscons}.
Proposition \ref{prop:consensus_interp_Gs} below shows that these interpolation constraints are agent-independently Gram-representable, by expressing them in terms of $\Gt$ and $\Gd$. In principle, the compact PEP formulation could apply to a general set of averaging matrices $\mc{M}$ for which there exist Gram-representable interpolation constraints. Future work may include the description of other classes of averaging matrices, in particular, the extension to non-symmetric stochastic matrices with a bound on their singular value.
\begin{proposition}[Averaging matrix interpolation constraints] \label{prop:consensus_interp_Gs}
  The averaging matrix interpolation constraints \eqref{eq:avg_pres2}, \eqref{eq:SDP_cons2} and \eqref{eq:sym_cons2}, from Corollary \ref{cor:conscons}, are agent-independently Gram-representable and can be expressed using $\Gt$ and $\Gd \in \Rmat{p}{p}$ as:
  \begin{align}
    (\id_{X}-\id_{Y})^T \Gt (\id_{X}-\id_{Y}) = 0, \label{eq:avg_pres_gt}\\
    (\id_Y - \lm \id_X)^T \Gd (\id_Y - \lp \id_X) \preceq 0, \label{eq:var_red_gd} \\
    \id_Y^T(\Gd)\id_X  -  \id_X^T(\Gd)\id_Y = 0, \label{eq:sym_gd_gt}
  \end{align}
  where $\id_X, \id_Y \in \Rvec{p}$ are matrices of coefficients such that $P_i \id_X = X_i$, $P_i \id_Y = Y_i$ with $X_i, Y_i \in \Rmat{d}{\K}$ the matrices with iterates $x_i^k, y_i^k \in \Rvec{d}$ as columns.
\end{proposition}
\begin{proof}
  Firstly, we express \eqref{eq:avg_pres2} as \eqref{eq:avg_pres_gt}. Equation \eqref{eq:avg_pres2} $\overline{X} = \overline{Y}$ means that the agent average is preserved and can be written with scalar products as:
  \[\qty(\frac{1}{\N} \sum_{i=1}^\N (X_i-Y_i))^T\qty(\frac{1}{\N} \sum_{j=1}^\N (X_j-Y_j)) = 0, \]
  where matrices $X_i, Y_i \in \Rmat{d}{\K}$ contain the different consensus iterates $x_i, y_i \in \Rvec{d}$ of agent $i$ as columns. Using relation \eqref{eq:useGt}, applied to columns of $X_i-Y_i$ and $X_j-Y_j$, this equation can be written as
  \begin{align}
    \frac{1}{\N^2} \sum_{i=1}^\N\sum_{j=1}^\N (X_i-Y_i)^T(X_j-Y_j) = (\id_{X}-\id_{Y})^T \Gt (\id_{X}-\id_{Y}) = 0.
  \end{align}
  Secondly, we express \eqref{eq:SDP_cons} as \eqref{eq:var_red_gd}. Equation $\eqref{eq:SDP_cons}$ involve centered matrices $\Xp, \Yp \in \Rmat{\N d}{\K}$,
  \begin{align}
    (\Yp-\lm\Xp)^T(\Yp-\lp\Xp) & \preceq 0, \label{eq:SDP_cons_proof}
  \end{align}
  We can write this product as a sum over the agents of products of centered variables, which then allows using relation \eqref{eq:useGd} to express it with $\Gd$. To this end, let us define $U_i = Y_i - \lm X_i$, $V_i = Y_i - \lp X_i$, and $\overline{U}, \overline{V}$ their corresponding agent average: $\overline{U} = \frac{1}{\N} \sum_{i=1}^\N U_i$, $\overline{V} = \frac{1}{\N} \sum_{i=1}^\N V_i$. Equation \eqref{eq:SDP_cons_proof} can be written as
  \begin{align}
    \sum_{i=1}^\N (U_i-\overline{U})^T(V_i-\overline{V}) & \preceq 0.
  \end{align}
  This can be reformulated using relation \eqref{eq:useGd} applied to columns of $U_i$ and $V_i$ and dividing the equation by $\N$:
  \[ \frac{1}{\N}\sum_{i=1}^\N (U_i-\overline{U})^T(V_i-\overline{V}) = \id_\U^T \Gd \id_V = (\id_Y - \lm \id_X)^T \Gd (\id_Y - \lp \id_X) \preceq 0 \]
  Finally, the symmetry condition \eqref{eq:sym_cons} can be expressed as \eqref{eq:sym_gd_gt} using a similar technique
  \begin{align}
      \Xp^T \Yp - \Yp^T\Xp = \sum_{i=1}^\N (X_i - \overline{X})^T(Y_i - \overline{Y}) -  \sum_{i=1}^\N (Y_i - \overline{Y})^T(X_i - \overline{X}) = 0.
  \end{align}
  This can be reformulated using relation \eqref{eq:useGd} applied to columns of $X_i$ and $Y_i$ and dividing the equation by $\N$:
  \[ \id_Y^T(\Gd)\id_X  -  \id_X^T(\Gd)\id_Y = 0 \]
\end{proof}

\paragraph*{Points common to all agents ($x^*$, $\xb$, etc)}
As explained in Section \ref{sec:equicl}, to build our compact PEP formulation for distributed optimization, we divide the agent-dependent PEP solutions into blocks of variables related to each agent, see \eqref{eq:Fagents} and \eqref{eq:defPi}, recalled below
\begin{equation} 
  \f= \begin{bmatrix}f_1^T & \dots & f_\N^T \end{bmatrix}, \quad \text{where $f_i = \qty[f_i^k]_{k \in I} \in \Rvec{q}$ is a vector with the $q$ function values of agent $i$,}
\end{equation}
\begin{equation} 
  P = \begin{bmatrix}P_1 & \dots & P_\N \end{bmatrix}  \quad \text{where $P_i \in \Rmat{d}{p}$ contains the $p$ vector variables related to agent $i$,} 
\end{equation}
e.g. $P_i = \qty[y_i^k~ x_i^k~ g_i^k]_{k \in I}$,
and the Gram matrix $G \in \Rmat{\N p}{\N p}$ is thus defined as
\begin{equation} 
  G = P^TP =   \begin{bmatrix}P_1^TP_1 & P_1^TP_2 & \dots \\ P_2^TP_1 & \ddots & \\  \vdots& & P_\N^TP_\N \end{bmatrix} =  \begin{bmatrix}G_{11} & G_{12} & \dots \\ G_{21} & \ddots & \\  \vdots& & G_{\N\N} \end{bmatrix}.
\end{equation}
The variables common to all agents, such as $x^*$ or $\xb^\K$, are copied into each agent block $P_i$ and their definitions are added as constraints of the problem, e.g.
\begin{equation} \label{eq:common_def}
   \xb_i = \xb = \frac{1}{\N} \sum_{j=1}^\N x_j \quad \text{for all $i\in\V$}, \qquad \text{ or } \qquad \frac{1}{n}\sum_{i=1}^\N \nabla f_i(x_i^*)=0, \quad \text{with $x^*_i = x^*_j = x^*$ for all $i,j \in\V$},
\end{equation}
Such definition constraints are agent-independently Gram-representable, as shown in Propositions \ref{prop:opti_cons} and \ref{prop:xb_cons} below.
If used in the problem, each agent $i$ also holds the variables for its local function and gradient values associated with this common point. For example, for $x^*$, we consider a new triplet $(x^*_i, g_i^*, f_i^*)$ for each agent $i$, which is taken into account in the interpolation constraints of its local function. These interpolation constraints ensure that $g_i^*$ and $f_i^*$ correspond to a gradient vector and function value at $x^*$ that are consistent with the given class of functions. Vectors $x^*_i$ and $g_i^*$ are columns of $P_i$ and value $f_i^*$ is an element of vector $f_i$.

\noindent To summarize, for any point $x_c$ common to all the agents in the problem, we need two elements in the PEP: 
\begin{enumerate}[leftmargin=1cm,label=(\roman*)]
  \item the definition constraints for $x_c$, e.g. \eqref{eq:common_def}, 
  \item the interpolation constraints for the new triplet $(x_c,g_c,f_c)$, if $g_c$ or $f_c$ used in the PEP.
\end{enumerate}
The addition of a new triplet of points to consider in the interpolation constraints does not alter the result of Proposition \ref{prop:fct_interp_2} and they can still be written as agent-independently Gram-representable constraints with $\fa$ and $\Ga$. 
The definition constraints can be treated as any other constraint of the problem, e.g. the initial constraints, and these constraints can thus be formulated independently of $\N$ if they can be written with scale-invariant expressions (see Definition \ref{def:scale-invariant} and Theorem \ref{thm:agent_indep_perf}). Proposition \ref{prop:opti_cons} gives an agent-independent formulation for the definition constraint of the optimal solution $x^*$ of the decentralized optimization problem \eqref{opt:dec_prob}.
\begin{proposition}[Optimality constraint] \label{prop:opti_cons}
  Let $x^*$ be an optimal point for the decentralized optimization problem \eqref{opt:dec_prob}. The definition constraints for the common variable $x^*$ in a PEP restricted to fully symmetric solutions $(\f^s,G^s)$ \eqref{eq:Fs}-\eqref{eq:Gs}, given by the system-level stationarity constraint
  \begin{equation} \label{eq:common_def_xs}
    \frac{1}{n}\sum_{i=1}^\N \nabla f_i(x_i^*)=0, \quad \text{with $x^*_i = x^*_j = x^*$ for all $i,j \in\V$},
 \end{equation}
  is agent-independently Gram-representable and can be expressed as
  \begin{equation} \label{eq:opti_cons}
     \id_{g^*}^T (\Gt) \id_{g^*} = 0 \quad \text{ and } \quad \id_{x^*}^T (\Gd) \id_{x^*} = 0,
  \end{equation}
  where $\id_{g^*}, \id_{x^*} \in \Rvec{p}$ are vectors of coefficients such that $P_i \id_{g^*} = g_i^* = \nabla f_i(x^*)$, and $P_i \id_{x^*} = x_i^*$ for all $i=1,\dots,\N$.
\end{proposition}
\begin{proof}
  The first constraint in \eqref{eq:common_def_xs} guarantees that the optimal point $x_i^*$ held by each agent satisfies the system-level stationarity constraint  $\frac{1}{\N}\sum_{i=1}^\N \nabla f_i(x^*_i) = 0.$
  Then, the second set of constraints imposes that the local optimal point $x_i^*$ is the same for any agent $i$. Let $g_i^*$ denote the gradient of function $f_i$ at the optimal point $x^*_i$. 
  We can formulate the stationarity constraint using only scalar products as
  \begin{align}
    \Biggl(\frac{1}{\N}\sum_{i=1}^\N g_i^*\Biggr)^T\Biggl(\frac{1}{\N}\sum_{j=1}^\N g_j^*\Biggr) = \frac{1}{\N^2}\sum_{i=1}^\N \sum_{j=1}^\N (g_i^*)^Tg_j^* = 0,
  \end{align}
  Using relation \eqref{eq:useGt}, applied to vectors $g_i^*$, $g_j^*$, we recover condition $\id_{g^*}^T (\Gt) \id_{g^*} = 0$.
  For the second set of constraints, we can impose them all with one constraint involving only scalar products: 
  \[ \frac{1}{2\N} \sum_{i=1}^\N\sum_{j=1}^\N (x_i^*-x_j^*)^T(x_i^*-x_j^*) = \frac{1}{\N^2} \sum_{i=1}^\N\sum_{j=1}^\N (x_i^*)^T x_i^* - (x_j^*)^Tx_j^* = 0 \]
  which can be written, using relation \eqref{eq:useGd}, as
  \[ \frac{1}{\N^2}\sum_{i=1}^\N\sum_{j=1}^\N (x_i^*)^Tx_i^* - (x_i^*)^Tx_j^* = \frac{1}{\N^2} \sum_{i=1}^\N \sum_{j=1}^\N (\id_{x^*})^T(G_{ii} - G_{ij}) (\id_{x^*}) = (\id_{x^*})^T (\Gd) (\id_{x^*}).\]
\end{proof}
\begin{remark} Without loss of generality, the optimal solution $x^*$ of problem \eqref{opt:dec_prob} can be set to $x^*=0$. In that case, the constraint $x^*_i = x^*_j$, written as $\id_{x^*}^T (\Gd) \id_{x^*} = 0$, is not needed to define $x^*$ properly. Indeed, it is sufficient to say that the coefficient vector is zero: $\id_{x^*} = 0$. This improves the numerical conditioning of the resulting SDP PEP.
\end{remark}

Proposition \ref{prop:xb_cons} gives an agent-independent formulation for the definition constraint of the agent average iterate $\xb^k$. 
To simplify notation, we omit the $k$ index for iterations.
\begin{proposition}[Agent average definition] \label{prop:xb_cons}
  We consider a PEP for distributed optimization, restricted to fully symmetric solutions $(\f^s,G^s)$ \eqref{eq:Fs}-\eqref{eq:Gs}. The definition constraints for $\xb$
  \begin{equation} \label{eq:avg_def}
     \xb_i = \xb = \frac{1}{\N} \sum_{j=1}^\N x_j \quad \text{for all $i\in\V$,}
  \end{equation} 
  are agent-independently Gram-representable, and can be written as 
  \[ \id_{\xb}^T (\Ga)\id_{\xb} + \id_x^T \Gt (\id_x - 2\id_{\xb}) = 0, \]
  where $\id_{\xb}, \id_{x} \in \Rvec{p}$ are vectors of coefficients such that $P_i \id_{\xb} = \xb_i$, and $P_i \id_{x} = x_i$ for all $i=1,\dots,\N$.
\end{proposition}
\begin{proof}
  We can formulate the set of constraints \eqref{eq:avg_def} with one constraint involving only scalar products of variables: 
  \[ \frac{1}{\N}\sum_{i=1}^\N\Big(\xb^\K_i - \frac{1}{\N} \sum_{j=1}^\N x_j^\K\Big)^T\Big(\xb^\K_i - \frac{1}{\N} \sum_{l=1}^\N x_l^\K\Big) = 0\]
  This expression can be expanded to be written with $\Ga$ and $\Gt$:
\begin{align}
  \frac{1}{\N}\sum_{i=1}^\N(\xb_i)^T(\xb_i) + \frac{1}{\N^2} \sum_{j=1}^\N \sum_{l=1}^\N x_j^Tx_l - \frac{2}{\N^2} \sum_{i=1}^\N\sum_{j=1}^\N x_j^T\xb_i = \id_{\xb}^T (\Ga)\id_{\xb} + \id_x^T \Gt (\id_x - 2\id_{\xb}) = 0,
\end{align}
where we used relation \eqref{eq:useGa} for the first term and \eqref{eq:useGt} for the other two.
\end{proof}

\paragraph*{Initial conditions and performance measures}
By combining terms of the form of \eqref{eq:usefa}, \eqref{eq:useGa}, \eqref{eq:useGt}, \eqref{eq:usefa_single}, \eqref{eq:useGa_single}, and \eqref{eq:useGd}, possibly involving common points defined above, it is possible to express many initial conditions and performance measures in terms of $\fa$, $\Ga$, $\Gt$, and $\Gd$, without dependence on $\N$. Examples of usual initial conditions and their reformulation are given in Proposition \ref{prop:initial_cond} below. The proposition focuses on the initial conditions, but all the expressions involved in these constraints can also be adapted as a performance measure, using the last point $x^t$ instead of the initial point $x^0$.
\begin{proposition}[Initial conditions]  \label{prop:initial_cond}
  Let $x^* \in \Rvec{d}$ denote the optimal solution of the distributed optimization problem \eqref{opt:dec_prob}, $x_i^0 \in \Rvec{d}$ the initial iterate of agent $i$, $\xb^0$ the average initial iterate and $R \in \R$ a constant. Here is a list of initial conditions that are agent-independently Gram-representable when restricted to agent-symmetric solutions, together with their reformulations in terms of $\fa$, $\Ga$, $\Gt$, and $\Gd$:
   \begin{align}
     \text{Initial}&\text{ condition} \hspace{1cm}& \text{Reformulation} \\  
      \frac{1}{\N} \sum_{i=1}^\N \|x_i^0 - x^*\|^2 &\le R^2, \quad  & (\id_{x^0}-\id_{x^*})^T (\Ga) (\id_{x^0}-\id_{x^*}) \le R^2, \\[-2mm]
      \text{\small{or \hspace{12mm} $\|x_i^0 - x^*\|^2$}} &~ 
      \text{\small{$\le R^2$ \quad for all $i=1,\dots,\N$}} & \\
     \frac{1}{\N} \sum_{i=1}^\N \|\nabla f_i(x_i^0)\|^2 &\le R^2,  & (\id_{g^0})^T (\Ga) (\id_{g^0}) \le R^2, \\[-2mm]
     \text{\small{or \hspace{12mm} $\|\nabla f_i(x_i^0)\|^2$}} &~ 
      \text{\small{$\le R^2$ \quad for all $i=1,\dots,\N$}} & \\
    \frac{1}{\N} \sum_{i=1}^\N \|x_i^0 - \xb^0\|^2 &\le R^2, & (\id_{x^0})^T \Gd (\id_{x^0}) \le R^2, \\
     x_i^0 - x_j^0 &= 0 \quad \text{for all $i,j=1,\dots,\N$} & (\id_{x^0})^T \Gd (\id_{x^0}) = 0, \\
     \frac{1}{\N}\sum_{i=1}^\N \qty(f_i(\xb^0) - f_i(x^*)) &\le R &  \hspace*{-4cm} \qty(\id_{f(\xb^0)} - \id_{f(x^*)})^T \fa \qty(\id_{f(\xb^0)} - \id_{f(x^*)})  &\le R
   \end{align}
   where $\id_{x^0}, \id_{x^*}, \id_{g^0} \in \Rvec{p}$ and $\id_{f(\xb^0)}, \id_{f(x^*)} \in \Rvec{q}$ are vectors of coefficients such that $P_i \id_{x^0} = x_i^0$, $P_i \id_{x^*} = x^*$, $P_i \id_{g^0} = \nabla f_i(x_i^0)$, $\id_{f(\xb^0)}^T f_i = f_i(\xb^0)$ and $\id_{f(x^*)}^T f_i = f_i(x^*)$, for all $i = 1,\dots,\N$.
\end{proposition}
\begin{proof}
  To obtain the first two reformulations, we can use relation \eqref{eq:useGa} applied respectively to vectors $x_i^0 - x^*$ and $\nabla f_i(x_i^0)$
  \begin{align}
    \frac{1}{\N} \sum_{i=1}^\N \|x_i^0 -x^*\|^2 &= \frac{1}{\N} \sum_{i=1}^\N (x_i^0 - x^*)^T (x_i^0 - x^*) = (\id_{x^0}-\id_{x^*})^T (\Ga) (\id_{x^0}-\id_{x^*}) \le R^2. \\
    \frac{1}{\N} \sum_{i=1}^\N \|\nabla f_i(x_i^0)\|^2 &= \frac{1}{\N} \sum_{i=1}^\N (\nabla f_i(x_i^0))^T (\nabla f_i(x_i^0)) = (\id_{g^0})^T (\Ga) (\id_{g^0}) \le R^2.
  \end{align}
    When considering the alternative uniform conditions for all agents, we obtain the same reformulations using relation \eqref{eq:useGa_single}, applied respectively to vectors $x_i^0 - x^*$ and $\nabla f_i(x_i^0)$. For the third condition, we can directly use relation \eqref{eq:useGd} to obtain it. Then, the fourth condition is equivalent to
  \[ \frac{1}{2\N^2}\sum_{i=1}^\N\sum_{j=1}^\N (x_i^0-x_j^0)^T(x_i^0-x_j^0) = \frac{1}{\N^2}\sum_{i=1}^\N\sum_{j=1}^\N (x_i^0)^Tx_i^0 - (x_i^0)^Tx_j^0 = 0, \]
  which can be written, using relation \eqref{eq:useGd}, as
  \[ \frac{1}{\N^2}\sum_{i=1}^\N\sum_{j=1}^\N (x_i^0)^Tx_i^0 - (x_i^0)^Tx_j^0 = \frac{1}{\N^2} \sum_{i=1}^\N \sum_{j=1}^\N (\id_{x^0})^T(G_{ii} - G_{ij}) (\id_{x^0}) = (\id_{x^0})^T (\Gd) (\id_{x^0}).\]
  Finally, the fifth and last condition can be reformulated with $\fa$ using relation \eqref{eq:usefa} applied to $\frac{1}{\N}\sum_{i=1}^\N f_i(\xb^0)$ and $\frac{1}{\N}\sum_{i=1}^\N f_i(x^*)$.
\end{proof}

\subsection{With multiple equivalence classes of agents}
Let $\T$ be a partition of the agent set $\V$ into $\U$ equivalence classes of agents (see Definition \ref{def:equicl}). Considering the PEP restricted to symmetrized agent-class solutions $\Gss$ and $\Fss$ \eqref{eq:Fsu}-\eqref{eq:Gsu}, we now detail how we can express all the PEP elements in terms of $\fa^u$, $\Ga^u$, $\Gr^u$, and $\Gc^{uv}$ ($u,v=1\dots,\U$), to be able to use this blocks as variable of the compact SDP PEP formulation. 
For the SDP condition $G^s \succeq 0$, we can be express it with $\Ga^u$, $\Gr^u$, and $\Gc^{uv}$ using Lemma \ref{lem:SDPGssets}. For the other PEP elements, we need to adapt the reformulation techniques \eqref{eq:usefa}, \eqref{eq:useGa}, \eqref{eq:useGt}, \eqref{eq:usefa_single}, \eqref{eq:useGa_single} and \eqref{eq:useGd} to the situation where the agent set $\V$ is partitioned into $\U$ subsets of equivalent agents $\T = \{\V_1, \dots, \V_\U\}$. 
The terms that we can be expressed in terms of $\fa^u$, $\Ga^u$, $\Gr^u$, and $\Gc^{uv}$ are the following (for any variables $x_i$ and $y_j$ related to agents $i$ and $j$, including the case $x_i = y_j$):
\begin{itemize}
  \item  The function value of a given agent $j \in \V_u$
  \begin{equation} \label{eq:usefau}
    f_j(x_j) = \frac{1}{\N_u} \sum_{i\in\V_u} f_i(x_i) = \id_{f(x)}^T \fa^u, \qquad \text{for any $j=1,\dots,\N$,}
  \end{equation}
  where the first equality holds because $f_i(x_i) =  f_j(x_j)$ for any $i,j \in \V_u$, by definition of the symmetric agent-class solution $\Fss$ \eqref{eq:Fsu}.
  This implies that
   \begin{equation} \label{eq:usefat}
       \frac{1}{\N} \sum_{i=1}^\N f_i(x_i) = \frac{1}{\N} \sum_{u=1}^\U \sum_{j \in V_u}^\N f_j(x_j)= \frac{1}{\N} \sum_{u=1}^\U \N_u \id_{f(x)}^T \fa^u = \id_{f(x)}^T \fat,
   \end{equation}
   where vector $\fat \in \Rvec{q}$ is given by
   \begin{equation} \label{eq:def_fat}
     \fat = \frac{1}{\N} \sum_{u=1}^\U  \N_u \fa^u.
  \end{equation}
   \item The scalar product between two variables related to the same agent $i\in \V_u$:
   \begin{equation} \label{eq:useGau}
       x_i^Ty_i = \frac{1}{\N_u} \sum_{j \in V_u}^\N x_j^Ty_j = \id_x^T \Ga^u \id_y  \qquad \text{for any $i=1,\dots,\N$,}
   \end{equation}
   where the first equality holds because $x_i^Ty_i = x_j^Ty_j $ for any $i,j \in \V_u$, by definition of the symmetric agent-class solution $\Gss$ \eqref{eq:Gsu}.
   This implies that
   \begin{equation} \label{eq:useGA}
       \frac{1}{\N} \sum_{i=1}^\N x_i^Ty_i = \frac{1}{\N} \sum_{u=1}^\U \sum_{j \in V_u}^\N x_j^Ty_j = \frac{1}{\N} \sum_{u=1}^\U \N_u \id_x^T \Ga^u \id_y = \id_x^T (\Gat) \id_y,
   \end{equation}
   where matrix $\Gat \in \Rmat{p}{p}$ is given by
   \begin{equation} \label{eq:def_Gat}
     \Gat = \frac{1}{\N} \sum_{u=1}^\U  \N_u \Ga^u.
  \end{equation}
   \item The average over the $\N^2$ pairs of agents of the scalar products between two variables, each related to any agent:
   \begin{align}
       \frac{1}{\N^2} \sum_{i=1}^\N \sum_{j=1}^\N x_i^Ty_j &= \frac{1}{\N^2} \sum_{u=1}^\U \sum_{v=1}^\U \sum_{i\in\V_u} \sum_{j\in\V_v} x_i^T y_j =
       \frac{1}{\N^2} \sum_{u=1}^\U \sum_{i\in\V_u} \sum_{j\in\V_u} x_i^T y_j + \frac{1}{\N^2} \sum_{u=1}^\U \sum_{v\ne u} \sum_{i\in\V_u} \sum_{j\in\V_v} x_i^T y_j, \\
       \frac{1}{\N^2} \sum_{i=1}^\N \sum_{j=1}^\N x_i^Ty_j &= \id_x^T \qty(\frac{1}{\N^2} \sum_{u=1}^\U  \bigl(\N_u \Ga^u + \N_u(\N_u-1)\Gr^u \bigr) + \frac{1}{\N^2} \sum_{u=1}^\U \sum_{v \ne u} \N_u\N_v \, \Gc^{uv}) \id_y = \id_x^T (\Gtt) \id_y,
       \label{eq:useGtu}
   \end{align}
   where matrix $\Gtt \in \Rmat{p}{p}$ is the sum of all the $p \times p$ blocks of $\Ht$ divided by $\N^2$:
   \begin{equation} \label{eq:def_GTt}
     \Gtt = \frac{1}{\N^2} \qty(\sum_{u=1}^\U  \N_u(\Ga^u + (\N_u-1)\Gr^u) + \sum_{u=1}^\U \sum_{v \ne u} \N_u\N_v\, \Gc^{uv}).
  \end{equation}
   \item The average over the $\N$ agents of the scalar products of two centered variables related to the same agents:
   \begin{align}
     \frac{1}{\N} \sum_{i=1}^\N \qty(x_i - \xb)^T \qty(y_i - \yb) &= \frac{1}{\N^2} \sum_{i=1}^\N \sum_{j=1}^\N (x_i^Ty_i  - x_i^Ty_j)  = \frac{1}{\N^2} \sum_{u=1}^\U \sum_{v=1}^\U \sum_{i\in\V_u} \sum_{j\in\V_v} (x_i^Ty_i  - x_i^Ty_j), \\
     &  = \id_x^T \qty(\frac{1}{\N^2} \sum_{u=1}^\U \N_u(\N_u-1) (\Ga^u-\Gr^u) + \frac{1}{\N^2} \sum_{u=1}^\U \sum_{v \ne u} \N_u \N_v (\Ga^u - \Gc^{uv})) \id_y, \\
    \frac{1}{\N} \sum_{i=1}^\N \qty(x_i - \xb)^T \qty(y_i - \yb) &= \id_x^T(\Gdt) \id_y
     \label{eq:useGdu}
   \end{align}
   where $\xb = \frac{1}{\N} \sum_{i=1}^\N x_i$, $\yb = \frac{1}{\N} \sum_{i=1}^\N y_i$, and matrix $\Gdt \in \Rmat{p}{p}$ is given by
   \begin{equation} \label{eq:def_GDt}
     \Gdt = \frac{1}{\N^2} \qty(\sum_{u=1}^\U \N_u(\N_u-1) (\Ga^u-\Gr^u) + \sum_{u=1}^\U \sum_{v \ne u} \N_u \N_v (\Ga^u - \Gc^{uv}))
   \end{equation}
 \end{itemize}
 By definitions \eqref{eq:def_Gat}, \eqref{eq:def_GTt} and \eqref{eq:def_GDt}, we have $\Gat = \Gtt + \Gdt$.
Using relations \eqref{eq:usefau}, \eqref{eq:usefat}, \eqref{eq:useGau}, \eqref{eq:useGA}, \eqref{eq:useGtu} and \eqref{eq:useGdu}, we can easily adapt Propositions \ref{prop:fct_interp}-\ref{prop:initial_cond} to this general case with multiple equivalence classes of agents to build all the PEP elements for distributed optimization only based on blocks $\fa^u$, $\Ga^u$, $\Gr^u$, and $\Gc^{uv}$.

 \begin{proposition}[Function interpolation constraints] \label{prop:fct_interp_2}
   Let $\F_u$ be a set of functions for which there exist linearly Gram-representable interpolation constraints. When considering symmetric agent-class PEP solutions $(\Fss-\Gss)$ \eqref{eq:Fsu}-\eqref{eq:Gsu}, the set of interpolation constraints for $\F_u$, applied independently to each agent of a given equivalence class $\V_u$, can be expressed with $\fa^u$ and $\Ga^u$.
 \end{proposition}
 \begin{proof}
  The constraint $f_i \in \F_u$ for each $i\in \V_u$ can be written with a set of Gram-representable interpolation constraints involving function values of agent $i$ and scalar products between quantities (i.e. points and gradients) related to agent $i$. These are thus single-agent constraints, that are identical for all $i\in \V_u$ when we restrict to a symmetrized agent-class solution $(\Fss, \Gss)$ \eqref{eq:Fsu}-\eqref{eq:Gsu}. Therefore, they can all be expressed with $\fa^u$ and $\Ga^u$ using \eqref{eq:usefau} and \eqref{eq:useGau}.
\end{proof}

 \paragraph*{Algorithm description}
 The class of methods $\Ad$ that we consider is defined in Definition \ref{def:Ad} and may combine gradient evaluation, consensus steps, and linear combinations of variables. In the latter, the agents from the same equivalence class should define the same linear combinations of their local variables, otherwise they are not equivalent. Agents from different equivalence classes may apply different combinations. This allows considering heterogeneity in the algorithm parameters, for example, to analyze the effect of uncoordinated step-sizes. In any case, all the variables in the combinations are related to the same agent and can therefore be squared and formulated in PEP using $\Ga^u$, as in \eqref{eq:useGau}. 

 Concerning the consensus steps, Theorem \ref{thm:sets_perf} applies to the set of averaging matrices $\Wcl{\lm}{\lp}$, defined in Section \ref{sec:agentPEP}. Necessary interpolation constraints for this set of matrices are given in \eqref{eq:avg_pres2}, \eqref{eq:SDP_cons2} and \eqref{eq:sym_cons2}, from Corollary \ref{cor:conscons}.
 Proposition \ref{prop:consensus_interp_Gs} below shows that these interpolation constraints can be expressed in terms of $\Gtt$ \eqref{eq:def_GTt} and $\Gdt$ \eqref{eq:def_GDt}, and so in terms of $\Ga^u$, $\Gr^u$, and $\Gc^{uv}$ ($u,v=1,\dots,\U$).

 \begin{proposition}[Averaging matrix interpolation constraints] \label{prop:consensus_interp_Gss}
   The averaging matrix interpolation constraints \eqref{eq:avg_pres2}, \eqref{eq:SDP_cons2} and \eqref{eq:sym_cons2}, from Corollary \ref{cor:conscons} can be expressed with $\Ga^u$, $\Gr^u$, and $\Gc^{uv}$ ($u,v=1,\dots,\U$), as:
   \begin{align}
     (\id_{X}-\id_{Y})^T \Gtt (\id_{X}-\id_{Y}) = 0, \label{eq:avg_pres_gt_2}\\
     (\id_Y - \lm \id_X)^T \Gdt (\id_Y - \lp \id_X) \preceq 0, \label{eq:var_red_gd_2} \\
     \id_Y^T(\Gdt)\id_X  -  \id_X^T(\Gdt)\id_Y = 0, \label{eq:sym_gd_gt_2}
   \end{align}
   where $\Gtt$ and $\Gdt$ are defined in \eqref{eq:def_GTt} and \eqref{eq:def_GDt} and $\id_X, \id_Y \in \Rvec{p}$ are matrices of coefficients such that $P_i \id_X = X_i$, $P_i \id_Y = Y_i$ with $X_i, Y_i \in \Rmat{d}{\K}$ the matrices with iterates $x_i^k, y_i^k \in \Rvec{d}$ as columns.
 \end{proposition}
 \begin{proof}
   The proof is similar to the one of Proposition \ref{prop:consensus_interp_Gs} but using relation \eqref{eq:useGtu} instead of \eqref{eq:useGt} and \eqref{eq:useGdu} instead of \eqref{eq:useGd}
 \end{proof}
 \paragraph*{Points common to all agents}
As detailed in Appendix \ref{ap:explicit_equiv} in the case where all agents are equivalent, for any point $x_c$ common to all the agents in the problem, we need two elements in the PEP: 
\begin{enumerate}[leftmargin=1cm,label=(\roman*)]
  \item the definition constraints for $x_c$, e.g. \eqref{eq:common_def}, 
  \item the interpolation constraints for the new triplet $(x_c,g_c,f_c)$, if $g_c$ or $f_c$ used in the PEP.
\end{enumerate}
The addition of a new triplet of points to consider in the interpolation constraints does not alter the result of Proposition \ref{prop:fct_interp} and they can still be written with $\fa^u$ and $\Ga^u$. 
The definition constraints can be treated as any other constraint of the problem. When there are several equivalence classes of agents in PEP, this allows more types of constraints to be expressed in the compact PEP formulation, and thus enlarges the type of common points that can be defined in the problem. Optimal point $x^*$ and agent-average $\xb$ can still be expressed in the compact problem, as shown in Propositions \ref{prop:opti_cons2} and \ref{prop:xb_cons2}. 

\begin{proposition}[Optimality constraint] \label{prop:opti_cons2}
  Let $x^*$ be an optimal point for the decentralized optimization problem \eqref{opt:dec_prob}. The definition constraints for the common variable $x^*$ in a PEP restricted to symmetrized agent-class solutions $(\Fss,\Gss)$ \eqref{eq:Fsu}-\eqref{eq:Gsu}, given by the system-level stationarity constraint
  \begin{equation}
    \frac{1}{n}\sum_{i=1}^\N \nabla f_i(x_i^*)=0, \quad \text{with $x^*_i = x^*_j = x^*$ for all $i,j \in\V$},
 \end{equation}
  can be expressed with $\Ga^u$, $\Gr^u$, and $\Gc^{uv}$ ($u,v=1,\dots,\U$) as
  \begin{equation}
     \id_{g^*}^T (\Gtt) \id_{g^*} = 0 \quad \text{ and } \quad \id_{x^*}^T (\Gdt) \id_{x^*} = 0,
  \end{equation}
  where $\Gtt$ and $\Gdt$ are defined in \eqref{eq:def_GTt} and \eqref{eq:def_GDt} and $\id_{g^*}, \id_{x^*} \in \Rvec{p}$ are vectors of coefficients such that $P_i \id_{g^*} = g_i^* = \nabla f_i(x^*)$, and $P_i \id_{x^*} = x_i^*$ for all $i=1,\dots,\N$.
\end{proposition}
\begin{proof}
  The proof is similar to the one of Proposition \ref{prop:opti_cons} but using relation \eqref{eq:useGtu} instead of \eqref{eq:useGt} and \eqref{eq:useGdu} instead of \eqref{eq:useGd}.
\end{proof}
\begin{remark} Without loss of generality, the optimal solution $x^*$ of problem \eqref{opt:dec_prob} can be set to $x^*=0$. In that case, the constraint $x^*_i = x^*_j$, written as $\id_{x^*}^T (\Gdt) \id_{x^*} = 0$, is not needed to define $x^*$ properly. Indeed, it is sufficient to say that the coefficient vector is zero: $\id_{x^*} = 0$. This improves the numerical conditioning of the resulting SDP PEP.
\end{remark}

\begin{proposition}[Agent average definition] \label{prop:xb_cons2}
  We consider a PEP for distributed optimization, restricted to symmetrized agent-class solutions $(\Fss,\Gss)$ \eqref{eq:Fsu}-\eqref{eq:Gsu}. The definition constraints for $\xb$
  \begin{equation}
     \xb_i = \xb = \frac{1}{\N} \sum_{j=1}^\N x_j \quad \text{for all $i\in\V$,}
  \end{equation} 
  can be written with $\Ga^u$, $\Gr^u$, and $\Gc^{uv}$ ($u,v=1,\dots,\U$) as 
  \[ \id_{\xb}^T (\Gat)\id_{\xb} + \id_x^T \Gtt (\id_x - 2\id_{\xb}) = 0, \]
  where $\Gat$ and $\Gtt$ are defined in \eqref{eq:def_Gat} and \eqref{eq:def_GTt}, and $\id_{\xb}, \id_{x} \in \Rvec{p}$ are vectors of coefficients such that $P_i \id_{\xb} = \xb_i$, and $P_i \id_{x} = x_i$ for all $i=1,\dots,\N$.
\end{proposition}
\begin{proof}
  The proof is similar to the one of Proposition \ref{prop:xb_cons} but using relation \eqref{eq:useGA} instead of \eqref{eq:useGa} and \eqref{eq:useGtu} instead of \eqref{eq:useGt}.
\end{proof}

A point common to all agents, that can only be defined using different equivalence classes of agents, is the iterate of a given agent $x_1^\K$, at which all the local functions are evaluated $f_1(x_1^\K),\dots, f_\N(x_1^\K)$. This can be used to evaluate the performance of the worst agent, see Section \ref{sec:worst_agent}. The definition of such a point is treated in the following proposition. 

 \begin{proposition}[Specific agent definition] 
  We consider a PEP for distributed optimization, restricted to symmetrized agent-class solutions $(\Fss,\Gss)$ \eqref{eq:Fsu}-\eqref{eq:Gsu}. Let $\V_1 = \{1\}$ be one of the $\U$ equivalence classes of the PEP. If $x_1$ is a common variable used by all the agents in the PEP, its definition constraints are given by
  \begin{equation} \label{eq:wcx_def}
     x_{c,i} = x_1 \quad \text{for all $i\in\V$}
  \end{equation} 
  and can be written with $\Ga^u$ and $\Gc^{uv}$ ($u,v=1,\dots,\U$) as 
  \begin{equation} \label{eq:wcx_reform}
     \id_{x_c}^T (\Ga^u)\id_{x_c} + \id_{x}^T (\Ga^1)\id_{x} - 2 \id_{x_c}^T (\Gc^{u1})\id_{x} = 0, \qquad \text{for all $u=2,\dots,\U$,}
  \end{equation}
  where $\id_{x_c}, \id_{x_1} \in \Rvec{p}$ are vectors of coefficients such that $P_i \id_{x_c} = X_{c,i}$, and $P_i \id_{x} = x_i$ for all $i=1,\dots,\N$.
 \end{proposition}
 \begin{proof}
    Constraint \eqref{eq:wcx_def} can be written as $(x_{c,i} - x_1)^2 = 0$, which can be expanded as 
    \[ x_{c,i}^Tx_{c,i} + x_1^Tx_1 - 2x_{c,i}^Tx_1 =  0 \quad \text{for all $i\in \V$}. \]
    This constraint can be expressed as \eqref{eq:wcx_reform}, by using relation \eqref{eq:useGau} for the first two terms and definition of $\Gc^{uv}$ \eqref{eq:Gsu} for the third term.
 \end{proof}

 \paragraph*{Initial conditions and performance measures}
 We adapt Proposition \ref{prop:initial_cond} to the case where there are multiple equivalence classes of agents in the PEP, in order to present usual initial conditions and their reformulation. The proposition focuses on the initial conditions, but all the expressions involved in these constraints can also be adapted as a performance measure, using the last point $x^t$ instead of the initial point $x^0$.
 Other initial conditions or performance measures could involve only one of the equivalence classes of agents. 

 \begin{proposition}[Initial conditions]  \label{prop:initial_cond_2}
   Let $x^* \in \Rvec{d}$ denote the optimal solution of the distributed optimization problem, $x_i^0 \in \Rvec{d}$ the initial iterate of agent $i$, $\xb^0$ the average initial iterate and $R \in \R$ a constant. When the PEP is restricted to symmetric agent-class solutions $(\Fss,\Gss)$ \eqref{eq:Fsu}-\eqref{eq:Gsu}, here is a list of initial conditions that can be expressed with $\fa^u$, $\Ga^u$, $\Gr^u$, and $\Gc^{uv}$ (for $u,v=1,\dots,\U$): 
    \begin{align}
      \text{Initial}&\text{ condition} \hspace{1cm}& \text{Reformulation}& \\  
     \|x_i^0 - x^*\|^2 &\le R^2 \quad \text{for all $i=1,\dots,\N$} & (\id_{x^0}-\id_{x^*})^T \Ga^u (\id_{x^0}-\id_{x^*}) &\le R^2, \quad \text{for all $u=1,\dots,\U$} \\[-4mm]
      \frac{1}{\N} \sum_{i=1}^\N \|x_i^0 - x^*\|^2 &\le R^2, & (\id_{x^0}-\id_{x^*})^T \Gat (\id_{x^0}-\id_{x^*}) &\le R^2, \\
      \|\nabla f_i(x_i^0)\|^2 &\le R^2,  \quad \text{for all $i=1,\dots,\N$} & (\id_{g^0})^T (\Ga^u) (\id_{g^0}) &\le R^2, \quad \text{for all $u=1,\dots,\U$} \\[-4mm]
      \frac{1}{\N} \sum_{i=1}^\N \|\nabla f_i(x_i^0)\|^2 &\le R^2,  \quad \text{for all $i=1,\dots,\N$} & (\id_{g^0})^T \Gat (\id_{g^0}) &\le R^2, \\
     \frac{1}{\N} \sum_{i=1}^\N \|x_i^0 - \xb^0\|^2 &\le R^2, & (\id_{x^0})^T \Gdt (\id_{x^0}) &\le R^2, \\
      x_i^0 &= x_j^0 \quad \text{for all $i,j=1,\dots,\N$} & (\id_{x^0})^T \Gdt (\id_{x^0}) &= 0, \\
      \hspace*{-5mm} \frac{1}{\N}\sum_{i=1}^\N \qty(f_i(\xb^0) - f_i(x^*)) &\le R &  \hspace*{-4cm} \qty(\id_{f(\xb^0)} - \id_{f(x^*)})^T \fat \qty(\id_{f(\xb^0)} - \id_{f(x^*)})  &\le R
    \end{align}
  where $\fat$, $\Gat$ and $\Gdt$ are defined in \eqref{eq:def_fat}, \eqref{eq:def_Gat} and \eqref{eq:def_GDt}, and  $\id_{x^0}, \id_{x^*}, \id_{g^0} \in \Rvec{p}$ are vectors of coefficients such that $P_i \id_{x^0} = x_i^0$, $P_i \id_{x^*} = x^*$ and $P_i \id_{g^0} = \nabla f_i(x_i^0)$, for $i=1,\dots,\N$.
 \end{proposition}
 \begin{proof}
The proof is similar to the one of Proposition \ref{prop:initial_cond} but using relation \eqref{eq:useGau} instead of \eqref{eq:useGa}, relation \eqref{eq:useGdu} instead of \eqref{eq:useGd} and relation \eqref{eq:usefat} instead of \eqref{eq:usefa}.
 \end{proof}

\end{document}